\documentclass[lettersize,journal]{IEEEtran}

\usepackage{microtype}
\usepackage{graphicx}
\usepackage{booktabs} 
\usepackage{amsmath,amsfonts}
\usepackage{algorithmic}
\usepackage{algorithm}
\usepackage{array}
\usepackage[caption=false,font=normalsize,labelfont=sf,textfont=sf]{subfig}
\usepackage{textcomp}
\usepackage{stfloats}
\usepackage{url}
\usepackage{verbatim}
\usepackage{graphicx}
\usepackage{cite}
\usepackage{hyperref}

\usepackage{amssymb}
\usepackage{mathtools}
\usepackage{amsthm}

\usepackage[capitalize,noabbrev]{cleveref}

\newtheorem{theorem}{Theorem}[section]

\newtheorem{lemma}[theorem]{Lemma}

\theoremstyle{definition}
\newtheorem{definition}[theorem]{Definition}
\newtheorem{assumption}[theorem]{Assumption}
\theoremstyle{remark}
\newtheorem{remark}[theorem]{Remark}
\usepackage[textsize=tiny]{todonotes}
\usepackage{multirow}

\def\x{\bar{x}}
\def\y{\bar{y}}
\def\z{\bar{z}}
\def\v{\bar{v}}
\def\d{\bar{d}}
\def\t{\bar{t}}
\def\tx{\tilde{x}}
\def\ty{\tilde{y}}
\def\tv{\tilde{v}}
\def\mean{\frac{1}{n}\sum_{i=1}^n}
\def\argmin{\textrm{argmin}}

\newcommand{\mn}[1]{{\left\vert\kern-0.25ex\left\vert\kern-0.25ex\left\vert\kern0.3ex #1 
		\kern0.3ex\right\vert\kern-0.25ex\right\vert\kern-0.25ex\right\vert}}

\newcommand{\E}[1]{\mathbb{E}\Big[#1\Big]}
\newcommand{\Ek}[1]{\mathbb{E}_k\Big[#1\Big]}

\newcommand{\norm}[1]{\|#1\|^2}

\allowdisplaybreaks[3]
\newcommand{\ssldbo}{S$^3$LDBO}

\begin{document}
\title{\ssldbo: A Snapshot Single-Loop Algorithm \\for Decentralized Bilevel Optimization}

\author{Chao~Yin,
        Youran~Dong,
        Shiqian~Ma,
        Bofan~Wang,
        and~Junfeng~Yang%
\thanks{Chao Yin is with the School of Mathematics, Hohai University, Nanjing, P. R. China (e-mail: yinchao@hhu.edu.cn).}%
\thanks{Youran Dong, Bofan Wang, and Junfeng Yang are with the School of Mathematics, Nanjing University, Nanjing, P. R. China (e-mail: \mbox{yrdong@smail.nju.edu.cn}; wangbf@smail.nju.edu.cn; jfyang@nju.edu.cn).}%
\thanks{Shiqian Ma is with the Department of Computational Applied Mathematics and Operations Research, Rice University, Houston, USA (e-mail: sqma@rice.edu).}%
}

\maketitle

\begin{abstract}
Networked AI systems increasingly rely on multiple agents that collaboratively learn and adapt models over communication networks. 
In such systems, bilevel formulations naturally arise in hyperparameter optimization, data cleaning, and meta-learning, but the repeated evaluation of gradients, Jacobians, and Hessians can impose a substantial computational burden on individual agents. 
To address this challenge, we propose Snapshot-SLDBO (S$^3$LDBO), an efficient single-loop decentralized bilevel optimization algorithm that enables agents to intermittently skip expensive derivative evaluations through a snapshot mechanism. 
This mechanism can be interpreted as an autonomous computation-adaptation strategy for networked AI, where agents selectively perform costly local updates while maintaining global collaborative learning.
We establish the ergodic iteration complexity and the high probability nonergodic iteration complexity of the proposed algorithm within a deterministic setting. 
{Experimental results on hyperparameter optimization with synthetic and MNIST datasets, data hyper-cleaning on Fashion-MNIST, and decentralized meta-learning on miniImageNet demonstrate that the proposed algorithm improves computational efficiency while maintaining competitive learning performance.}
\end{abstract}


\section{Introduction}
{\IEEEPARstart{N}{etworked} AI systems, such as distributed sensing networks, edge intelligence, federated learning systems, and multi-agent autonomous platforms, require multiple agents to collaboratively learn and adapt models over communication networks \cite{xu2023edge,ye2022decentralized,chen2017multitask}. 
A central challenge in these systems is how to enable each agent to update its local model efficiently while maintaining global coordination with other agents. 
This challenge becomes more pronounced when the learning objective has a hierarchical structure, where model parameters, hyperparameters, data weights, or task-specific adaptation variables must be optimized jointly. 
Such hierarchical learning problems can often be formulated as bilevel optimization problems.}
Bilevel optimization (BO) has attracted increasing attention in recent years. It 
has found important applications in meta-learning \cite{franceschi2018bilevel, rajeswaran2019meta, ji2020convergence}, hyperparameter optimization \cite{pedregosa2016hyperparameter, franceschi2018bilevel, Zhang2023MP}, reinforcement learning \cite{hong2023two}, and adversarial learning \cite{bishop2020optimal, zhang2022revisiting}. 
To achieve state-of-the-art results in these domains, vast training datasets and distributed algorithms operating across multiple computing agents are often required.

This paper focuses on the following decentralized bilevel optimization (DBO) problem:
\begin{align}\label{DBO}
	\mathop{\min}\limits_{x \in \mathbb{R}^d}  & \;\Phi(x):=  F(x,y^*(x)) := \frac{1}{n}\sum_{i=1}^{n}F_i(x,y^*(x))\quad\nonumber\\
	\mathrm{s.t.} \   &y^*(x) \in  \mathop{\argmin}\limits_{y\in \mathbb{R}^q} f(x,y) :=  \frac{1}{n}\sum_{i=1}^{n}f_i(x,y),
\end{align}
where $n$ collaborative nodes, connected via a communication graph $\mathcal{G} = (\mathcal{V}, \mathcal{E})$, jointly solve this problem. 
Here, $\mathcal{V}$ represents the set of agents, and $\mathcal{E}$ represents the set of feasible communication links between the agents. Each node $i$ privately holds an upper-level cost function $F_i:\mathbb{R}^d\times\mathbb{R}^q\rightarrow\mathbb{R}$ and a lower-level cost function $f_i:\mathbb{R}^d\times\mathbb{R}^q\rightarrow\mathbb{R}$.
The decentralized framework eliminates the need for a central server and enables more efficient and scalable computation across distributed networks \cite{decentralized_vs_centralized_NIPS2017}.
In this paper, we propose an efficient single-loop algorithm for solving \eqref{DBO}. 

\subsection{Related work}
\paragraph{Bilevel optimization}
When utilizing gradient methods for solving \eqref{DBO}, a central challenge is how to estimate the hypergradient $\nabla \Phi(x)$, which is given by 
\begin{subequations}\label{hg}
\small
\begin{align}
&\nabla\Phi(x)=\frac{1}{n}\sum_{i=1}^{n}\left[\nabla_1 F_i\left(x, y^*(x)\right) - \nabla^2_{12} f_i\left(x, y^*(x)\right) z(x)\right],\\
&z(x) = \left[\frac{1}{n}\sum_{i=1}^{n}\nabla_{22}^2 f_i\left(x, y^*(x)\right)\right]^{-1} \frac{1}{n}\sum_{i=1}^{n}\nabla_2 F_i\left(x, y^*(x)\right).
\end{align}
\end{subequations}
%
To address this challenge, various algorithms have been proposed in the literature, leveraging approaches such as
approximate implicit differentiation (AID) \cite{ghadimi2018approximation, grazzi2020iteration, ji2021bilevel, ji2022will, chen2021closing, hong2023two, dagreou2022framework}, iterative differentiation (ITD) \cite{franceschi2018bilevel, pmlr-v89-shaban19a, ji2020convergence, grazzi2020iteration, ji2021bilevel, ji2022will, Liu2023tapami} and Neumann series-based approach \cite{ghadimi2018approximation, chen2021closing}.
Furthermore, both AID and ITD approaches involve heavy Hessian- and Jacobian-vector multiplications, which also consist of a complex double-loop structure. 
Recently,  \cite{dagreou2022framework} introduced a single-loop framework, named SOBA, to solve BO with single agent, i.e., $n=1$ in \eqref{DBO}. 
This framework approximates $z(x)$ without requiring heavy matrix-vector multiplications, thus saving effort in solving the linear system. Specifically, the SOBA algorithm maintains three sequences, which are updated as follows:
\begin{align*}
y^{k+1} = y^k - \beta_k D_y^k, 
v^{k+1} = v^k + \eta_k D_v^k,  
x^{k+1} = x^k - \alpha_k D_x^k,
\end{align*}
where $\alpha_k$, $\beta_k$ and $\eta_k$ are positive stepsizes, $D_y^k, D_v^k$ and $D_x^k$ are   respectively unbiased stochastic estimators of 
$\nabla_2f(x^k,y^k), \nabla_2F(x^k,y^k) - \nabla^2_{22}f(x^k,y^k)v^k$ and $\nabla_1F(x^k,y^k) - \nabla^2_{12}f(x^k,y^k)v^k$.
The SOBA framework was later extended by  \cite{liu2023averaged} to address the case where the lower-level problem is merely convex. \cite{Liu2021Nips} addressed the non-convexity of the lower-level function and proposed two techniques to tackle the resulting challenges.

\paragraph{Decentralized bilevel optimization}
In the past few years, a series of decentralized optimization algorithms have been proposed to solve \eqref{DBO}. 
The first algorithm for solving \eqref{DBO} was proposed in \cite{chen2022decentralized}. They introduced the DSBO algorithm, which integrates a decentralized approach to solve the linear system in \eqref{hg}. 
Subsequently, the same authors improved the DSBO algorithm by incorporating the moving average technique \cite{chen2023decentralized}.
Other works on DBO include \cite{YangGossipDBO2022}, in which a gossip-based DBO algorithm was proposed and its sample complexity was established, and \cite{pmlr-v206-gao23a}, in which momentum and variance-reduced techniques were employed to solve distributed stochastic BO problems.
All of these algorithms rely on inner-loop updates to assess the lower-level solution as well as the Hessian inverse, which demands extensive computational resources and leads to significant communication overhead, thus greatly reducing their practical applicability.
\cite{lu2022stochastic} proposed a stochastic linearized augmented Lagrangian method for solving \eqref{DBO}, which, however, also requires heavy matrix-vector multiplications when approximating the hypergradient. Recently, \cite{dong2023singleloop} extended the SOBA algorithm \cite{dagreou2022framework} to the decentralized setting and proposed the first single-loop decentralized algorithm, named SLDBO, for solving \eqref{DBO}. Similar to SOBA, SLDBO does not require heavy matrix-vector multiplications when approximating the hypergradient. Additionally, SLDBO imposes no assumptions regarding gradient heterogeneity or the boundedness of $\{v_k\}$. \cite{KongDSOBA2024} introduced a decentralized SOBA with the moving average technique and established its transient iteration complexity.
Similarly, \cite{ZhangDBO2023} employed variance reduction and gradient tracking techniques to address the stochastic DBO problem. However, the consensus error was not taken into account in the convergence analysis.
We point out that these single-loop algorithms still need to calculate gradients, Jacobian and Hessian matrices in each iteration, which can be time consuming for large-scale problems. 

In this paper, we incorporate the snapshot gradient tracking (SSGT) technique \cite{Song2021OptimalGT} -- a recent technique for decentralized optimization -- to SLDBO and demonstrate that the new algorithm \ssldbo~ is much more efficient than SLDBO.  
Unlike traditional methods that track the average gradient of decision variables, the SSGT tracks the average gradient of a snapshot point incorporating historical information.
This technique employs intermittent gradient, Jacobian, and Hessian computations, offering a more effective approach to utilizing computational resources.

\subsection{Main results} 
Contributions of this paper are summarized as follows. 
\begin{itemize}
\item We propose a novel single-loop algorithm named Snapshot-SLDBO (\ssldbo), which integrates the snapshot idea from SSGT \cite{Song2021OptimalGT} into SLDBO \cite{dong2023singleloop} and intermittently skips the computations of gradients, Jacobians and Hessians, thus significantly reducing the overall computational cost.

\item  We analyze the iteration complexity of \ssldbo, demonstrating that it achieves $\mathcal{O}(1/K)$ convergence rate, where $K$ is the iteration counter. Moreover, we establish a high probability non-ergodic convergence result for \ssldbo.

\item 
{We validate the proposed method on hyperparameter optimization, data hyper-cleaning, and decentralized meta-learning tasks. The results demonstrate improved time efficiency and competitive learning performance in networked AI scenarios.
Additionally, we also study the numerical impact of network topology and data heterogeneity on the performance of \ssldbo.}
\end{itemize}

While \ssldbo\, builds upon the foundational ideas of SLDBO and SSGT, their integration into a coherent and provably convergent algorithm presents non-trivial challenges. The primary difficulty lies in adapting the snapshot mechanism—originally conceived for single-level optimization—to the more intricate, three-sequence dynamic ($x, y, v$) inherent in bilevel algorithms like SLDBO. Our core contribution is therefore not merely the algorithmic proposal, but the rigorous convergence analysis that validates this complex integration. We demonstrate that despite intermittent computations across coupled sequences, the algorithm maintains a solid theoretical convergence guarantee, a key step toward making DBO more practical for large-scale problems. We also note that our current analysis is focused on the deterministic setting, where full gradients are accessible; extending this framework to the stochastic realm remains an important direction for future investigation.

\subsection{Organization}
The remainder of this paper is organized as follows.
In Section \ref{Preliminaries}, we introduce the notation, assumptions, and essential preliminaries which will be utilized in subsequent analysis.
In Section \ref{sec_algorithm}, we propose the \ssldbo\ algorithm and analyze its convergence.
{Numerical results on hyperparameter optimization, data hyper-cleaning, and decentralized meta-learning are presented in Section~\ref{sec_numeical} to demonstrate the effectiveness of S$^3$LDBO.}
Finally,   concluding remarks are drawn in Section \ref{sec_conclude}.

\section{Preliminaries}\label{Preliminaries}
We denote the optimal value of \eqref{DBO} as $F^*$. The gradients of $f$ with respect to $x$ and $y$ are represented by $\nabla_1f(x,y)$ and $\nabla_2 f(x,y)$ respectively, and the Jacobian matrix of $\nabla_1f$ and Hessian matrix of $f$ with respect to $y$ are denoted as $\nabla^2_{12}f(x,y)$ and $\nabla^2_{22}f(x,y)$ respectively.
Unless otherwise specified, $\|\cdot\|$ refers to the $\ell_2$ norm for vectors and the Frobenius norm for matrices. The operator norm of a matrix $Z$ is denoted by $\|Z\|_{\textrm{op}}$.
We also define the projection operator $\mathcal{P}_{r}$, which projects a given point onto a Euclidean ball with radius $r\geq 0$, i.e.,
$$\mathcal{P}_{r}[z]:= \arg\min\nolimits_{\|z'\|\leq r}\|z'-z\|=\min \left\{1, r/\|z\|\right\} z.$$

Throughout this paper, we adhere to the following standard assumptions in the literature of bilevel optimization and decentralized optimization problems.
\begin{assumption}\label{Assump1}
	The following assumptions hold for functions $F$, $f$, all $F_i$ and $f_i$ in \eqref{DBO}.
	\begin{enumerate}
		\item For any fixed $x$, $ f(x,\cdot)$ is $\sigma$-strongly convex, with $\sigma>0$ being a constant.
		
		\item The function $F_i$ is differentiable and $\nabla F_i$ is Lipschitz continuous with Lipschitz constant $L_{F,1}$. 
        \item {For all $x$, there exists a constant $L_{F,0}$ such that $\|\nabla_2 F(x, y^*(x))\| \leq L_{F,0}$.}
		
		\item The function $f_i$ is twice differentiable, and its gradient $\nabla f_i$ is Lipschitz continuous with Lipschitz constant $L_{f,1}$. Moreover, $\nabla_{12}^2 f_i(x,y)$ and $\nabla_{22}^2 f_i(x,y)$ are also Lipschitz continuous with Lipschitz constant $L_{f,2}$.

	\end{enumerate}
\end{assumption}

\begin{assumption}[Network topology]\label{Assump2}
	Suppose the communication network is represented by a nonnegative weight matrix $W = (w_{ij})\in \mathbb{R}^{n\times n}$,
	where $w_{ij}=0$ if $i\neq j$ and nodes $i$ and $j$ are not connected.
	Moreover, we assume that $W$ is symmetric and doubly stochastic, i.e. $W = W^{\mathsf{T}}$ and $W\mathbf{1_n} = \mathbf{1_n}$, where $\mathbf{1_n}$ is the all-one vector in $ \mathbb{R}^{n}$.
	Furthermore, the eigenvalues of $W$ satisfy $1 = \lambda_1> \lambda_2\geq \cdots \geq \lambda_n$ and $\rho := \max\{|\lambda_2|, |\lambda_n|\} < 1$.
\end{assumption}

The following results are instrumental in our convergence analysis. Lemma \ref{con_sqr} is a well-known result in decentralized optimization, as noted in \cite{pu2021distributed}. Lemma \ref{proj}, adopted from Lemma 3.2 in \cite{choi2023convergence}, elucidates the relationship between the consensus error before and after the projection.

\begin{lemma}\label{con_sqr}
	Consider the mixing matrix $W=(w_{ij})\in \mathbb{R}^{n\times n}$ defined in Assumption \ref{Assump2}, for any $x_1, \ldots, x_n \in \mathbb{R}^{d}$, let $\x = \frac{1}{n}\sum_{i=1}^{n}x_i$, we have
	\begin{enumerate}
		\item $	\sum\limits_{i=1}^n\big\|\sum\limits_{j=1}^n w_{i j}x_j\big\|^2\leq  \sum\limits_{j=1}^n\left\|x_j\right\|^2$, and 
		\item $\sum\limits_{i=1}^n\big\|\sum\limits_{j=1}^n w_{i j}\left(x_j-\x\right)\big\|^2\leq \rho^2 \sum\limits_{i=1}^n\left\|x_i-\x\right\|^2.$
	\end{enumerate}
\end{lemma}

\begin{lemma}\label{proj}
	For any $x_1, \ldots, x_n \in \mathbb{R}^{d}$ and $r\geq 0$, we have $$\sum_{i=1}^n\left\|\mathcal{P}_r\left[x_i\right]-\frac{1}{n} \sum_{j=1}^n \mathcal{P}_r\left[x_j\right]\right\|^2 \leq \sum_{i=1}^n\left\|x_i-\frac{1}{n} \sum_{j=1}^n x_j\right\|^2.
	$$
\end{lemma}

\section{Our \ssldbo~Algorithm}\label{sec_algorithm}

In this section, we propose the Snapshot-SLDBO algorithm (\ssldbo), which improves the SLDBO algorithm \cite{dong2023singleloop} with the idea of SSGT \cite{Song2021OptimalGT}. We start with discussing related algorithms and explaining the derivation of \ssldbo.

\subsection{Algorithm development}\label{review_algorithm}
The SSGT algorithm \cite{Song2021OptimalGT} is designed for solving the following single-level decentralized problem:
$$
\min\nolimits_{x} h(x) = \frac{1}{n}\sum\nolimits_{i=1}^n h_i(x),
$$
where $x$ is the global decision variable and each $h_i$ is a smooth and strongly convex function. 
Let $\{(\xi^k,\zeta^k):  k=0,1,2,\ldots\}$ be a sequence of two-point random variables with  $\xi^k \sim Bernoulli(p)$ and $\zeta^k \sim Bernoulli(l)/l$, where $p,l\in(0,1)$. 
With the initial values $(y^0_i,z^0_i,u^0_i,s_i^0)$ and $g_i^0 = \nabla h_i(s_i^0)$, $i=1,\ldots,n$, 
SSGT iterates for $k=0,1,2,\ldots$ as 
\begin{align*}
	x_i^{k} & = (1-\alpha-\tau) y_i^{k} + \alpha z_i^{k} + \tau u_i^{k}, \\
	z_i^{k+1} & = \frac{1}{1+\beta} \sum\nolimits_{j=1}^n w_{ij} \biggl\{z_j^{k} + \beta x_j^{k}  \\
    &\quad \;- \eta \Bigl[g_j^{k} + \zeta^{k} \left(\nabla h_j(x_j^{k}) - \nabla h_j(s_j^{k})\right)\Bigl]\biggl\},\\
	 y_i^{k+1} & = x_i^{k} + \gamma \left(z_i^{k+1} - z_i^{k}\right), \\
	s_i^{k+1} & = \xi^{k} x_i^{k} + (1-\xi^{k}) s_i^{k}, \\
	u_i^{k+1} & = \sum\nolimits_{j=1}^n w_{ij} \left(\xi^{k} x_j^{k} + (1-\xi^{k}) u_j^{k}\right), \\
	g_i^{k+1} & = \sum\nolimits_{j=1}^n w_{ij} g_j^{k} + \xi^{k} \left(\nabla h_i(x_i^{k}) - \nabla h_i(s_i^{k})\right),
\end{align*}
where $\alpha$, $\beta$, $\gamma$, $\eta$ and $\tau$ are positive constants. 
In other gradient-tracking based algorithms, when $\nabla h_i(x_i^{k+1}) - \nabla h_i(x_i^k)$ is large, a small stepsize is taken to control the consensus error, ultimately resulting in suboptimal convergence rates. The introduction of the snapshot point $s_i^k$ is designed to overcome this limitation and mitigate its impact on convergence rates. Meanwhile, the update of $x_i^k$ is expressed as a linear combination of three other variables $y_i^k, z_i^k$, and $u_i^k$, ensuring that the distance between $x_i^k$ and $s_i^k$ is not too large, and $\nabla h_i(x_i^k) - \nabla h_i(s_i^k)$ is not too large. 
Most importantly, the SSGT algorithm skips the gradient computation when $\xi^k = \zeta^k = 0$, which greatly reduces the computational burden.

The SLDBO algorithm \cite{dong2023singleloop} for solving \eqref{DBO} updates in the $k$th iteration as follows:
\begin{subequations}
	\begin{align}
		&t_{y, i}^{k+1}=\sum\nolimits_{j=1}^n w_{i j} t_{y, j}^{k}+d_{y,i}^{k+1}-d_{y, i}^{k}, 
		\\
		&y_i^{k+1}=\sum\nolimits_{j=1}^n w_{i j} (y_j^k-\beta  t_{y, j}^k),  \label{sldbo_y}\\
		&t_{v, i}^{k+1}=\sum\nolimits_{j=1}^n w_{i j} t_{v, j}^{k}+d_{v,i}^{k+1}-d_{v, i}^{k},
		\\
		&v_i^{k+1}=\mathcal{P}_{r_v}\left[\sum\nolimits_{j=1}^n w_{i j} 
		(v_j^k+\eta  t_{v, j}^k)\right],  \label{sldbo_v}\\
		&t_{x, i}^{k+1}=\sum\nolimits_{j=1}^n w_{i j}t_{x, j}^{k}+d_{x,i}^{k+1}-d_{x, i}^{k},
		\\ 
		&x_i^{k+1}=\sum\nolimits_{j=1}^n w_{i j} (x_j^k-\alpha t_{x, j}^k),   \label{sldbo_x}
	\end{align}
\end{subequations}
where  
%
$r_v:= L_{F,0}/\sigma$, and
$d_{y,i}^k$, $d_{v,i}^k$ and $d_{x,i}^k$ are defined in \eqref{yd}-\eqref{xd}, respectively.
In each iteration, the SLDBO algorithm needs to evaluate one gradient, one Jacobian, and one Hessian of the lower-level
function, and two gradients of the upper-level function. 
This brings heavy computational burden for large-scale problems.
To alleviate the burden caused by computing gradients, Jacobians and Hessians in each iteration, we propose the \ssldbo~algorithm, which skips these computations intermittently using the SSGT idea. 
The details of our \ssldbo~algorithm are described in Algorithm \ref{alg:sslDBo}.
Inspired by SSGT, we also introduce a stochastic variable $\xi^k$ in the $k$th iteration. 
It should be noted that whenever the random variable $\xi^k=0$, the algorithm does not compute any gradients, Jacobians, and Hessians.
This leads to substantial savings in computational time when dealing with large-scale problems.
In particular, as shown in \cite{dong2023singleloop}, the SLDBO algorithm computes three gradients, one Jacobian, and one Hessian in each iteration. Assuming that the amount of CPU time required for these computations is $\mathcal{T}$ in each iteration, this totals $\mathcal{T}K$ for $K$ iterations. In contrast, \ssldbo\ only needs an expected computational cost of $p\mathcal{T}K$ for the same number of iterations, where $p\in(0,1)$. This significantly reduces computational time. 

\begin{algorithm}[!htpb]
\caption{A Snapshot Single-Loop Algorithm for DBO (\ssldbo)}\label{alg:sslDBo}
\begin{algorithmic}[1]
	\REQUIRE   
	Let $K$ be the maximum iteration number and $p\in (0,1)$. 
	Initialize  $t_{x,i}^{0}=t_{y,i}^{0}=t_{v,i}^{0}=d_{\tx,i}^{0}=d_{\ty,i}^{0}=d_{\tv,i}^{0}=0$ for all $i$.     
	Given  $\{(x^{0}_i, y^{0}_i, v^{0}_i): i=1,2,\ldots,n\}$, satisfying  
	$\|v^{0}_i\|\leq r_v := L_{F,0}/\sigma$ for all $i$.       
	Let $\alpha, \beta, \eta >0$, which satisfy the upper bound conditions in \eqref{stepsize_bound}.
	\FOR {$k=0,1,\ldots,K-1$}  
	\STATE Sample $\xi^k \sim Bernoulli(p)$ 
	\FOR {$i=1,\ldots,n$}   
	\IF {$\xi^k =1$}
    \STATE \#\;{Computation.}
	\begin{align} 
		\hspace{-1em} d_{y,i}^k &=\nabla_2 f_i(x_i^k,y_i^k), 
		\label{yd} \\
		\hspace{-1em} d_{v,i}^k &=\nabla_2 F_i(x_i^k,y_i^k) - \nabla^2_{22} f_i(x_i^k,y_i^k) v_i^k, \label{vd} \\
		\hspace{-1em} d_{x,i}^k &=\nabla_1 F_i(x_i^k,y_i^k) - \nabla^2_{12} f_i(x_i^k,y_i^k) v_i^k, \label{xd}
	\end{align}  
    \STATE \#\;{Communication.}
    {\small
        \begin{align*}
		\hspace{-3.25em} 
        y_i^{k+1} &=\sum\nolimits_{j=1}^n w_{i j} \Big(y_j^k-\beta \big(t_{y,j}^k + \frac{1}{p}(d_{y,j}^k - d_{\ty,j}^k)\big)\Big),\\ 
        \hspace{-3.25em} 
        v_i^{k+1} &=\mathcal{P}_{r_v}\left\{\sum\nolimits_{j=1}^n w_{i j} \Big[v_j^k+\eta\big(t_{v,j}^k + \frac{1}{p}(d_{v,j}^k - d_{\tv,j}^k)\big)\Big]\right\},\\
        \hspace{-3.25em} 
        x_i^{k+1} &=\sum\nolimits_{j=1}^n w_{i j} \Big(x_j^k-\alpha \big(t_{x,j}^k + \frac{1}{p}(d_{x,j}^k - d_{\tx,j}^k)\big)\Big),\\\hspace{-3.25em} t_{y,i}^{k+1}&=\sum\nolimits_{j=1}^n w_{i j}(t_{y, j}^{k}+ d_{y,j}^{k}-d_{\tilde{y}, j}^{k}),\\
        \hspace{-3.25em} t_{v,i}^{k+1}&=\sum\nolimits_{j=1}^n w_{i j} (t_{v, j}^{k}+ d_{v,j}^{k}-d_{\tilde{v}, j}^{k}),\\
        \hspace{-3.25em} t_{x,i}^{k+1}&=\sum\nolimits_{j=1}^n w_{i j} (t_{x, j}^{k}+ d_{x,j}^{k}-d_{\tilde{x}, j}^{k}),
		\end{align*}
        }
	\STATE
	$\hspace{-.5em} d_{\ty,i}^{k+1} = d_{y,i}^{k}, \quad 
	d_{\tv,i}^{k+1} = d_{v,i}^{k}, \quad 
	d_{\tx,i}^{k+1} = d_{x,i}^{k}.$
	\ELSE 
    \STATE  \#\;{No computation of gradient, Jacobian and Hessian matrices.}
	\begin{align*}
	\hspace{-.5em}	
    y_i^{k+1}&=\sum\nolimits_{j=1}^n w_{i j} (y_j^k-\beta  t_{y, j}^k),\\
    \hspace{-.5em} v_i^{k+1}&=\mathcal{P}_{r_v}\left[\sum\nolimits_{j=1}^n w_{i j} (v_j^k+\eta  t_{v, j}^k)\right],\\
	\hspace{-.5em}  
    x_i^{k+1}&=\sum\nolimits_{j=1}^n w_{i j} (x_j^k-\alpha  t_{x, j}^k),\\
    \hspace{-.5em}    
    t_{y, i}^{k+1}&=\sum\nolimits_{j=1}^n w_{i j} t_{y, j}^{k}, t_{v, i}^{k+1}=\sum\nolimits_{j=1}^n w_{i j}t_{v, j}^{k},\\            
    \hspace{-.5em}   
    t_{x, i}^{k+1} &=\sum\nolimits_{j=1}^n w_{i j} t_{x, j}^{k},
	\end{align*}
	\STATE
	$
    \hspace{-.5em} 
	d_{\ty,i}^{k+1} = d_{\ty,i}^{k}, \quad 
	d_{\tv,i}^{k+1} = d_{\tv,i}^{k}, \quad
	d_{\tx,i}^{k+1} = d_{\tx,i}^{k}.
	$
	\ENDIF
	\ENDFOR
	\ENDFOR
	\ENSURE $\{(x_i^K,y_i^K,v_i^K): i=1,2,\ldots,n\}$. 
\end{algorithmic}
\end{algorithm}

\subsection{Convergence results}\label{convergence_results}
Next, we present convergence results of \ssldbo.  
First, we recall the definition of a stationary point as presented in \cite{xyy2024siamopt}.

\begin{definition}
For any $\varepsilon > 0$, a random vector $(\mathbf{x}, \mathbf{y})$, with $\mathbf{x} = [x_1; x_2; \ldots; x_n]^{\top}$ and $\mathbf{y} = [y_1; y_2; \ldots; y_n]^{\top}$,  is called a stochastic $\varepsilon$-stationary point in expectation of \eqref{DBO}
if
$$
\E{\norm{\nabla \Phi(\x)} + \mean \norm{x_i - \x} + \mean \norm{y_i - \y}} \leq \varepsilon,
$$
where 
$\x := \frac{1}{n}\sum_{i=1}^{n}x_i$ and  $\y := \frac{1}{n}\sum_{i=1}^{n}y_i$. 
\end{definition}
Our primary convergence rate results for Algorithm \ref{alg:sslDBo} are summarized in Theorem \ref{cr} and Theorem \ref{cp}, and their proofs are postponed to the Appendix \ref{appendix}.
\begin{theorem}[\textbf{Convergence in expectation}]\label{cr}
For any $k\geq 0$, define $\x^k = \frac{1}{n}\sum_{i=1}^{n}x^k_i$, $\y^k = \frac{1}{n}\sum_{i=1}^{n}y^k_i$ and $\v^k = \frac{1}{n}\sum_{i=1}^{n}v^k_i$.
For any integer $K\geq 1$, the following convergence rate results hold for Algorithm \ref{alg:sslDBo}.
\begin{enumerate}
	\item{Consensus Error.} 
    We have
    {
	\begin{align*}
		\E{\frac{1}{nK}\sum\nolimits_{k=0}^{K-1}\sum\nolimits_{i=1}^{n}\|x^k_i-\x^k\|^2}&=\mathcal{O}\left(\frac{1}{K}\right), \\	
		\E{\frac{1}{nK}\sum\nolimits_{k=0}^{K-1}\sum\nolimits_{i=1}^{n}\|y^k_i-\y^k\|^2}&=\mathcal{O}\left(\frac{1}{K}\right), \\
		\E{\frac{1}{nK}\sum\nolimits_{k=0}^{K-1}\sum\nolimits_{i=1}^{n}\|v^k_i-\v^k\|^2}&=\mathcal{O}\left(\frac{1}{K}\right).
	\end{align*}	
    }
	\item{Stationarity.} 
    There holds
    {
	$$\E{\frac{1}{K}\sum\nolimits_{k=0}^{K-1}\|\nabla \Phi(\x^k)\|^2} = \mathcal{O}\left(\frac{1}{K}\right).$$
    }
\end{enumerate}
\end{theorem}

\begin{theorem}[\textbf{Convergence in probability}]\label{cp}
{Let $\{(x_i^k, y_i^k): i=1, \ldots, n\}_{k\geq 0}$ be the sequence generated by Algorithm \ref{alg:sslDBo}. Define $\x^k = \frac{1}{n}\sum_{i=1}^{n}x^k_i$ and $\y^k = \frac{1}{n}\sum_{i=1}^{n}y^k_i$ for $k\geq 0$, and set
\[
	R_k :=
	\norm{\nabla \Phi(\x^k)}
	+
	\mean\norm{x_i^k-\x^k}
	+
	\mean\norm{y_i^k-\y^k}.
\]
}
{Then, for any $\varepsilon>0$, there holds
\[
	\lim_{k\rightarrow+\infty}
	\mathrm{Prob}\{R_k\geq\varepsilon\}=0.
\]
}
\end{theorem}

\section{Numerical experiments}\label{sec_numeical}
{
In this section, we evaluate the proposed S$^3$LDBO algorithm on three representative decentralized bilevel learning tasks, including hyperparameter optimization, data hyper-cleaning, and decentralized meta-learning. 
The first two tasks are conducted on synthetic data, MNIST, and corrupted Fashion-MNIST, while the meta-learning experiment is conducted on miniImageNet. 
These tasks cover both classical bilevel learning problems and neural-network-based multi-agent learning scenarios, and are used to assess the computational efficiency and learning performance of the proposed snapshot mechanism.

For hyperparameter optimization, we compare S$^3$LDBO with SLDBO~\cite{dong2023singleloop}, MA-DSBO~\cite{chen2023decentralized}, and SLAM~\cite{lu2022stochastic}. 
SLDBO is the most directly related single-loop DBO 
baseline, MA-DSBO is a representative double-loop method, and SLAM is a stochastic linearized augmented Lagrangian method for DBO. 
For data hyper-cleaning, we mainly compare S$^3$LDBO with SLDBO in order to isolate the effect of the snapshot mechanism against its closest single-loop counterpart. 
For decentralized meta-learning, we further include MAML~\cite{pmlr-v70-finn17a} and ANIL~\cite{ICLRRaghuRBV20} as representative meta-learning baselines. 
Other DBO algorithms are not included because they are designed for different stochastic settings, relying on different assumptions, or focusing on communication complexity rather than the local derivative-evaluation cost considered here. A major goal of the experiments is to isolate the computational savings achieved by the snapshot mechanism relative to the most direct single-loop counterpart, while also examining whether such savings translate into improved time efficiency in neural-network-based meta-learning.

Throughout all experiments, we set the number of nodes to $n=8$ and adopt a ring topology to model the decentralized communication network. 
The corresponding weight matrix $W=(w_{ij})\in\mathbb{R}^{n\times n}$ is defined as follows: for $i,j=1,\ldots,n$, $w_{ij}=w$ if $i=j$; $w_{ij}=(1-w)/2$ if $i=j\pm 1$ or $(i,j)\in\{(1,n),(n,1)\}$; and $w_{ij}=0$ otherwise. 
We set $w=0.4$, under which each node is connected to two neighboring nodes and the associated spectral constant is $\rho\approx 0.724$. 
Unless otherwise specified, all results are reported with respect to CPU time or the number of expensive local derivative evaluations, so as to directly reflect the computational savings brought by the snapshot mechanism.
}
All the experiments were performed within Python 3.8 running on a laptop with AMD Ryzen 5 3550H CPU at 2.1GHz with 16GB memory. We employed mpi4py \cite{dalcin2021mpi4py} for parallel computing.

\subsection{Hyperparameter optimization}
\subsubsection{Synthetic data}
We first carry out numerical experiments on the logistic regression problem with $\ell_2$ regularization.
Let $\psi(t) = \log (1+e^{-t})$ for $t\in\mathbb{R}$ and $d$ be the dimension of the data.
Following \cite{chen2023decentralized} and \cite{dong2023singleloop}, for each node $i$, we have
\begin{align*}
F_i(\lambda, \omega)&=\sum_{\left(x_e, y_e\right) \in \mathcal{D}_i^{\prime}} \psi\left(y_e x_e^{\top} \omega\right)
,\\
f_i(\lambda, \omega)&=\sum_{\left(x_e, y_e\right) \in \mathcal{D}_i} \psi\left(y_e x_e^{\top} \omega\right)+\frac{1}{2} \sum\nolimits_{j=1}^d e^{\lambda_j} \omega_j^2,
\end{align*}
where $\mathcal{D}_i$ and $\mathcal{D}_i^{\prime}$ denote the training and  testing datasets on node $i$, respectively.
We aim to determine the optimal hyperparameter $\lambda$ such that $\omega^*(\lambda)$ minimizes the logistic loss evaluated on the testing dataset.

We utilize synthetic heterogeneous data generated in the same manner as in \cite{chen2023decentralized} and \cite{dong2023singleloop}. 
Specifically, the data distribution of $x_e$ on node $i$ follows a normal distribution with mean $0$ and variance $i^2 \cdot r^2$, where $r$ is the heterogeneity rate. In our experiments, we set $r$ to $1$. For the corresponding response variable, we let $y_e=x_e^{\top} w+0.1 z$, where $z$ is sampled from the standard normal distribution.
The training and testing datasets used in all algorithms each consist of 20,000 samples. 
In our experiments, we compare \ssldbo\, with SLDBO \cite{dong2023singleloop}, MA-DSBO \cite{chen2023decentralized} and SLAM\cite{lu2022stochastic}. 
In SLDBO and \ssldbo, the parameter $r_v$ is set to 20, which can be chosen via grid search, 
and the values of both  ${\alpha}$ and ${\eta}$ are assigned as 0.5. Furthermore, we set ${\beta}$ to 1.2 and the maximum number of iterations is set to $10^3$, which are the same as in \cite{dong2023singleloop}.	
MA-DSBO utilizes two key parameters, namely $T$ and $N$. Here, $T$ stands for the number of iterations carried out in the inner loop, while $N$ represents the number of Hessian-inverse-gradient product iterations.
MA-DSBO requires sufficient inner-loop iterations to estimate the LL solution and the hypergradient accurately. 
SLAM has one key parameter $b$, which is associated with matrix-vector products used for estimating the inverse of
the Hessian.
In our experiments, we set $T=5$, $N=5$ and $b=5$.

To visualize the performance of the three algorithms being compared, we present the curves of train loss, test loss, classification accuracy, and hypergradient norm against CPU time (in seconds) in Fig. \ref{fig:d60_time} and Fig. \ref{fig:d300_time}, corresponding to $d = 60$ and $d = 300$ respectively.
As the dimension of the data increases, the computational demands for calculating the gradients, Jacobian, and Hessian matrices rise significantly.
The \ssldbo\ algorithm attains higher numerical efficiency due to its ability to skip certain computationally intensive steps, thereby enhancing the overall performance.
In this experiment, we tested two values of $p$, namely $p=0.3$ and $p=0.75$. 
These values were chosen to represent a low-frequency ($p=0.3$) versus a high-frequency ($p=0.75$) update regime. This allows us to empirically study the trade-off between per-iteration cost and convergence speed discussed.
It can be seen that $p = 0.3$ yields a better performance compared to $p = 0.75$.
Furthermore, by comparing Fig. \ref{fig:d60_time} and Fig. \ref{fig:d300_time}, it can be seen that the advantage of setting $p = 0.3$ becomes clearer for problems with larger dimensions compared to when $p = 0.75$. 

\begin{figure}[!t]
\centering
\includegraphics[width=\linewidth]{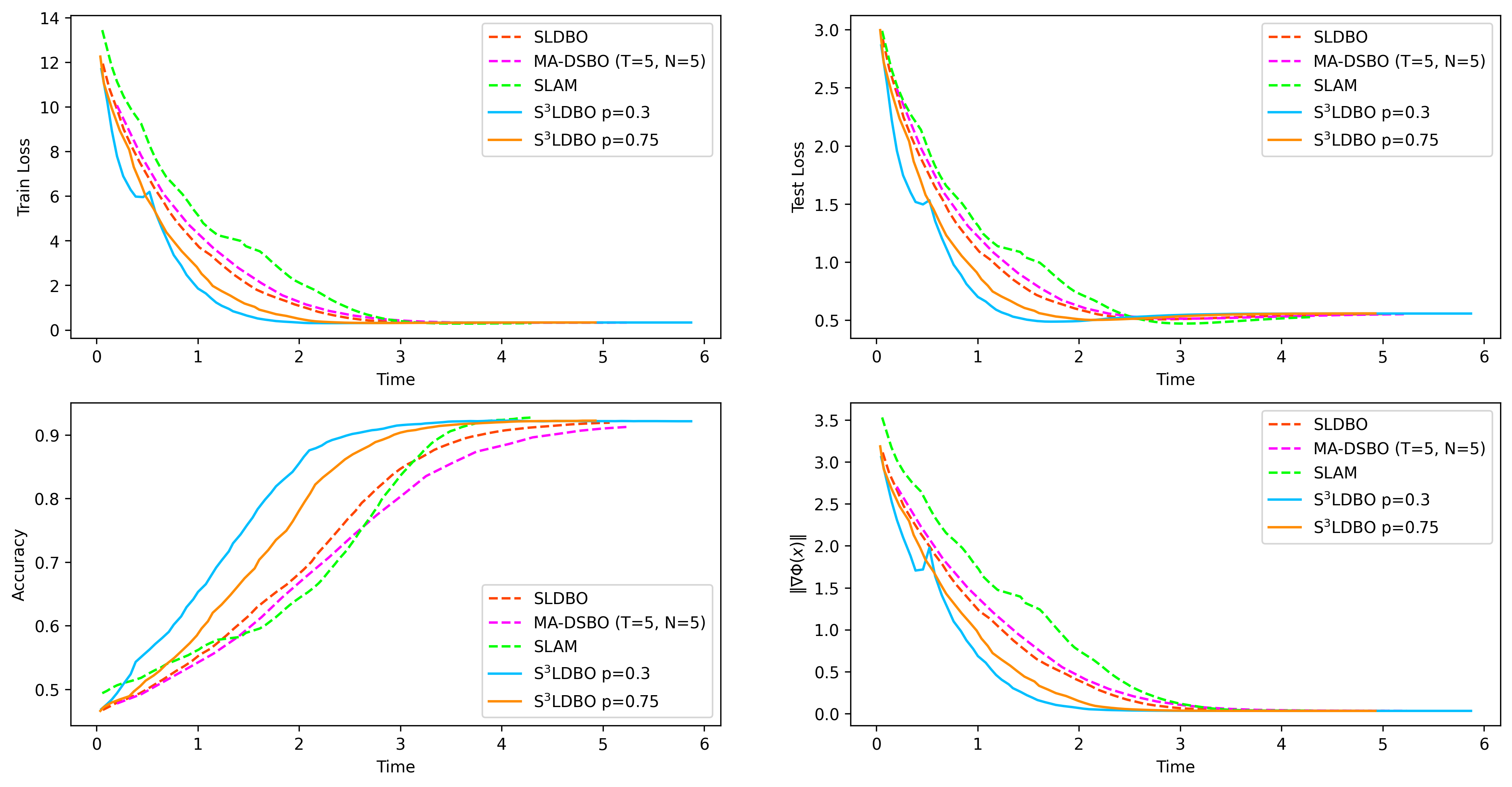}
\caption{Comparison of \ssldbo, SLDBO, SLAM and MA-DSBO on synthetic data with respect to CPU time ($d=60$).}
\label{fig:d60_time}
\end{figure}

\begin{figure}[!t]
\centering
\includegraphics[width=\linewidth]{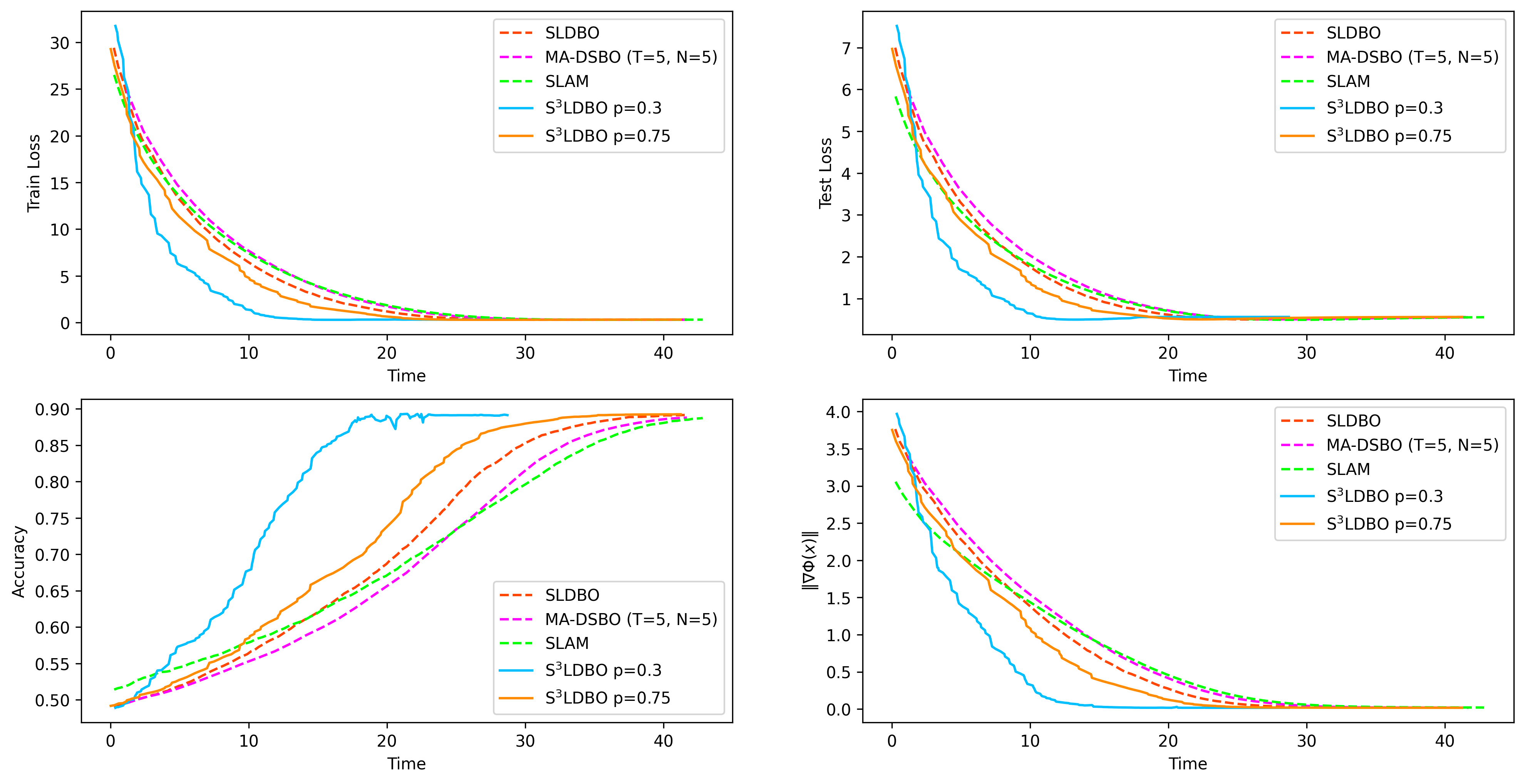}
\caption{Comparison of \ssldbo, SLDBO, SLAM and MA-DSBO on synthetic data with respect to CPU time ($d = 300$).}
\label{fig:d300_time}
\end{figure}

Fig. \ref{fig:d60_grad} and Fig. \ref{fig:d300_grad} display the comparison results between two single-loop algorithms, \ssldbo \,and SLDBO, with respect to train loss, test loss, accuracy, and hypergradient magnitude versus the number of lower-level function's gradient evaluations, where
Fig. \ref{fig:d60_grad} corresponds to $d=60$ and Fig. \ref{fig:d300_grad} corresponds to $d=300$.
It is noteworthy that within each iteration of the SLDBO algorithm, the iteration requires one computation of the lower-level function's gradient, two computations of the upper-level function's gradient, one computation of the lower-level function's Jacobian matrix, and one computation of the lower-level function's Hessian matrix. 
In Fig. \ref{fig:d60_grad} and Fig. \ref{fig:d300_grad}, the horizontal axis represents the number of lower-level function gradient computations in one node, based on which we present comparative curve results.
It can be seen from both figures that \ssldbo \,with $p=0.3$ performs the best, followed by \ssldbo \,with $p=0.75$ and SLDBO. 

\begin{figure}[!t]
\centering
\includegraphics[width=\linewidth]{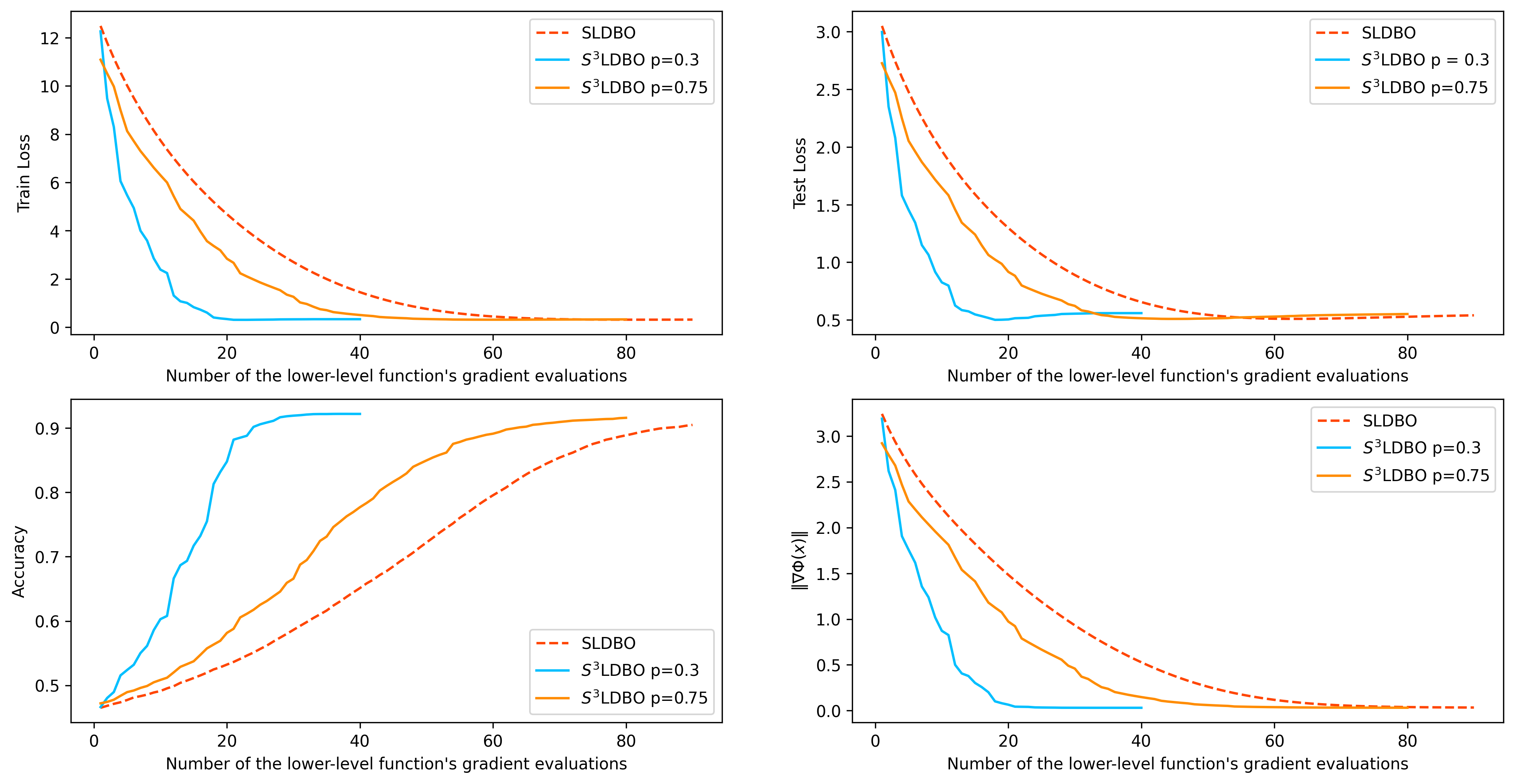}
\caption{Comparison between \ssldbo\, and SLDBO on synthetic data with respect to the number of gradient computations ($d$ = 60).}
\label{fig:d60_grad}
\end{figure}

\begin{figure}[!t]
\centering
\includegraphics[width=\linewidth]{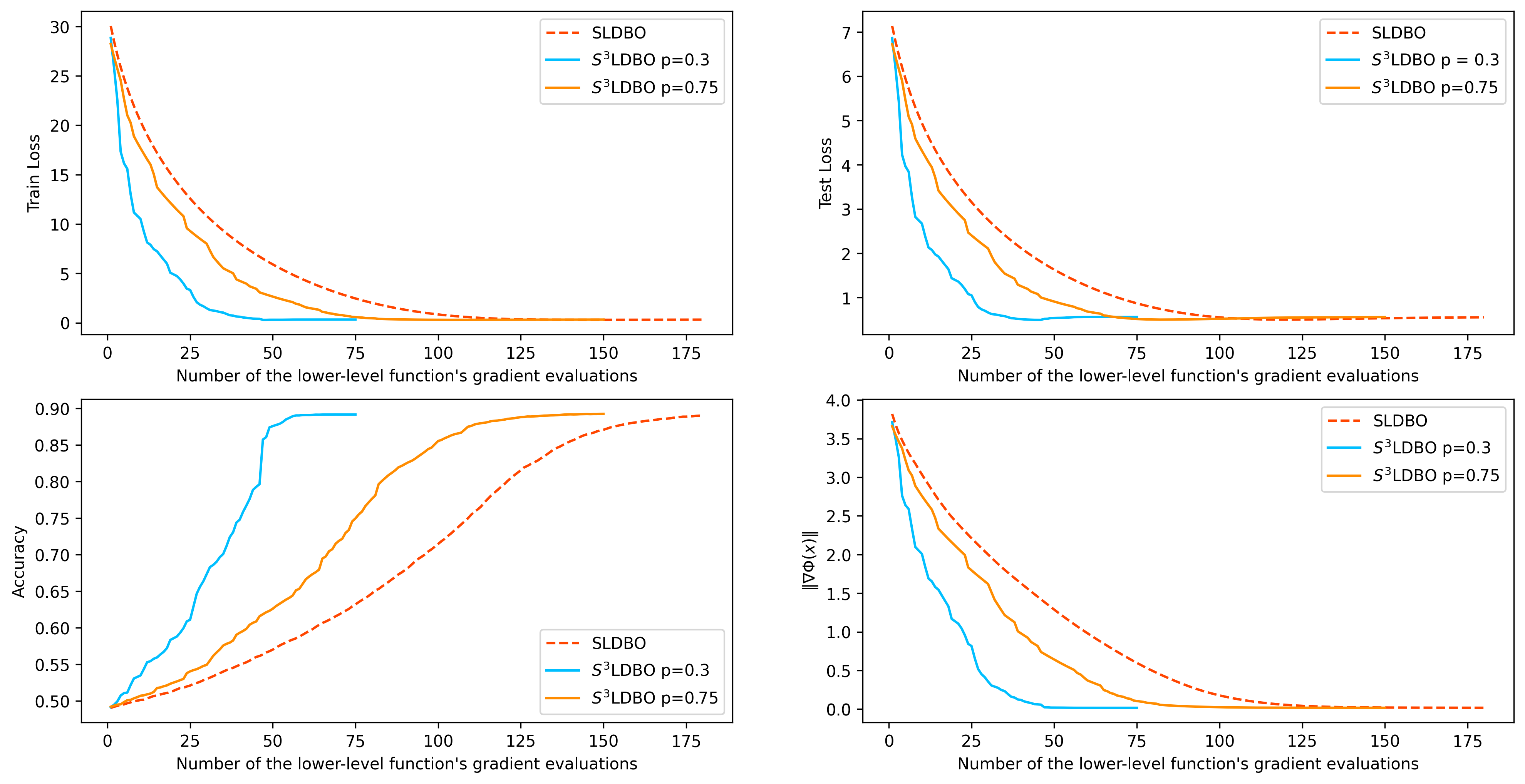}
\caption{Comparison between \ssldbo\, and SLDBO on synthetic data with respect to the number of gradient computations ($d$ = 300).}
\label{fig:d300_grad}
\end{figure}

Fig. \ref{fig:net} showcases the performance of \ssldbo\ across four graphs, which is showed in Fig. \ref{fig:net_topology}.
It can be observed that the performance of \ssldbo\ improves as the connectivity of the graph increases. Specifically, \ssldbo\ performs best when using the complete graph, followed by the random, the grid and the line graphs. This indicates that a network with better connectivity can generally accelerate the convergence process.

\begin{figure}[!t]
\centering
\includegraphics[width=\linewidth]{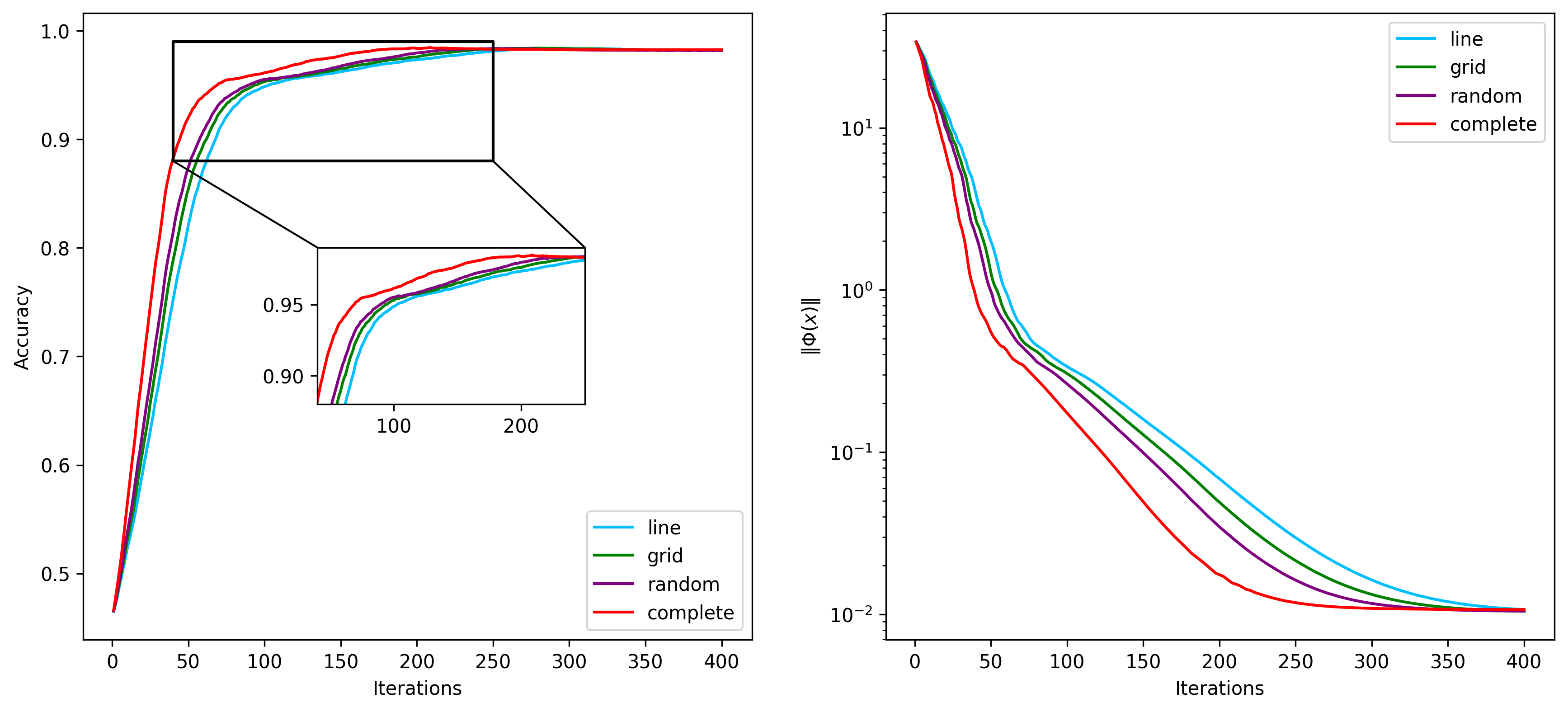}
\caption{Performance of \ssldbo\, with different networks.}
\label{fig:net}
\end{figure}

\begin{figure}[!t]
\centering
\includegraphics[width=\linewidth]{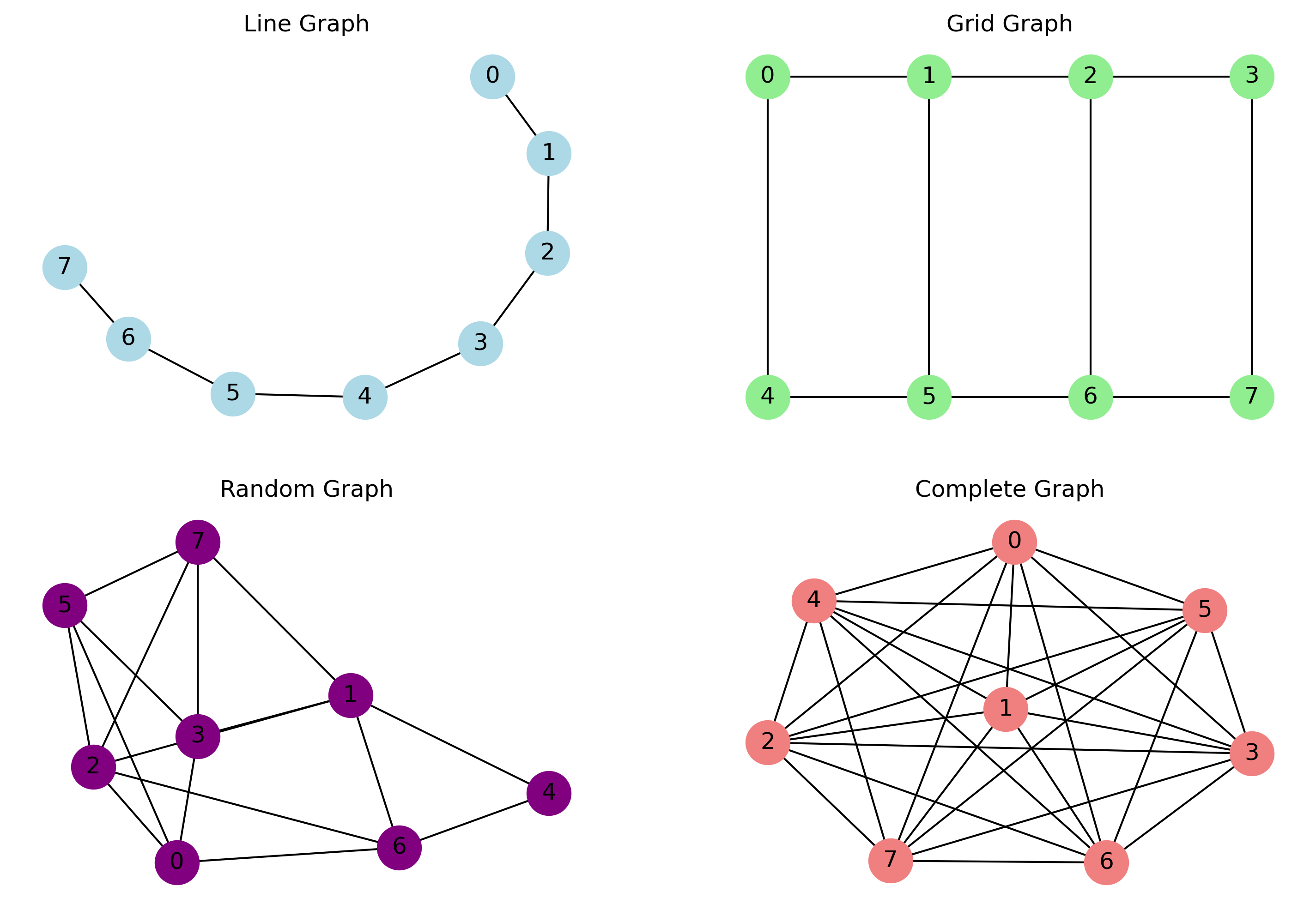}
\caption{Four distinct network structures.}
\label{fig:net_topology}
\end{figure}

\begin{figure}[!t]
\centering
\includegraphics[width=\linewidth]{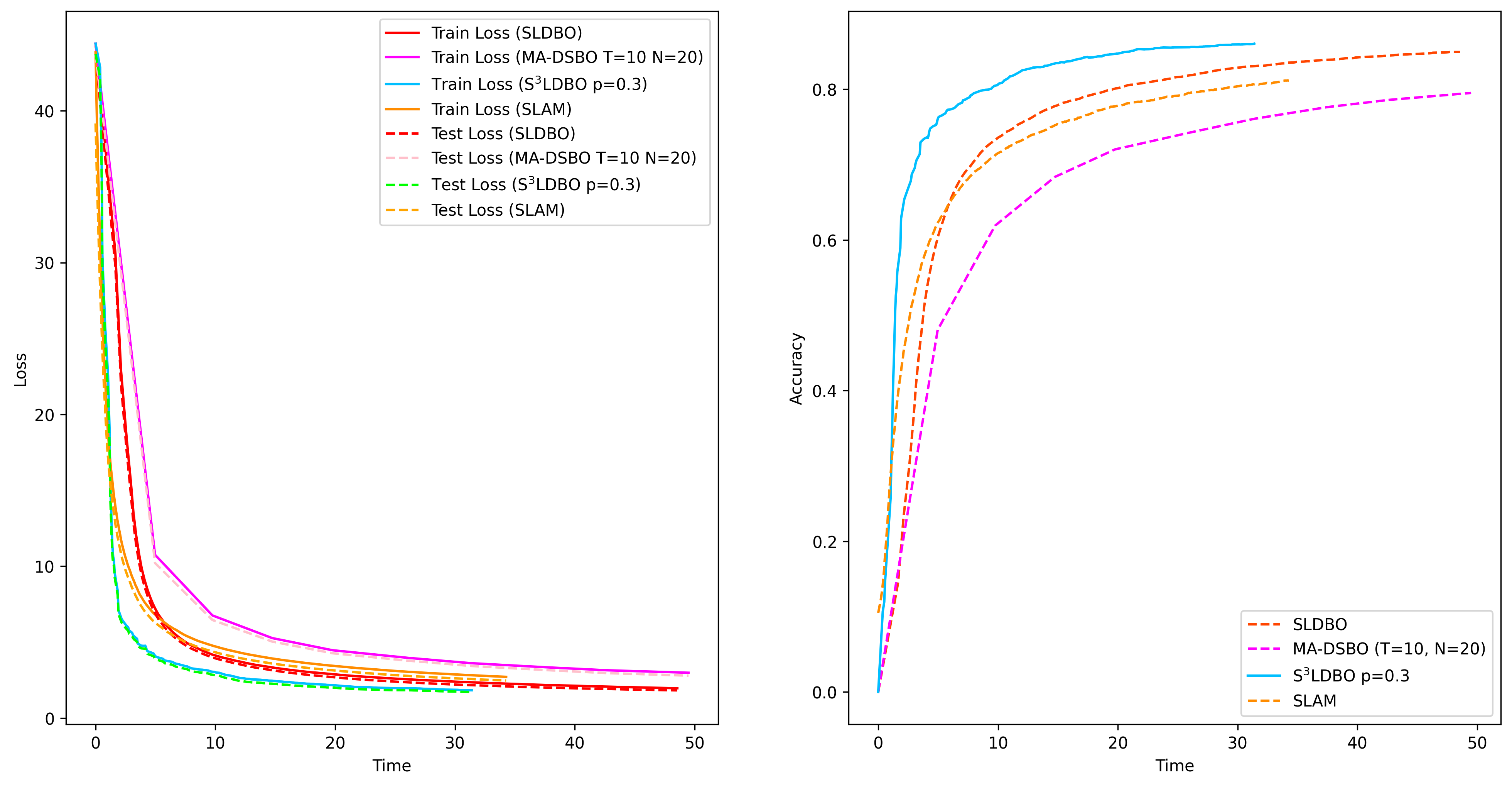}
\caption{Comparison of train loss, test loss, and classification accuracy for \ssldbo, SLDBO, SLAM and MA-DSBO on the MNIST dataset with respect to CPU time.}
\label{fig:slmnist}
\end{figure} 

\subsubsection{Real-world data}
Similar to the case of synthetic data, we define $\mathcal{D}_i$ and $\mathcal{D}_i^{\prime}$ as the training and testing datasets for node $i$, respectively.
Next, we apply \ssldbo, SLDBO and MA-DSBO to solve 
hyperparameter optimization problem using the MNIST dataset \cite{lecun1998gradient}.
In particular, the objective and constraint functions in \eqref{DBO} are given by 
\begin{align*}
F_i(\lambda, \omega)&= \frac{1}{|\mathcal{D}_i^{\prime}|}\sum_{(x_e,y_e)\in\mathcal{D}_i^{\prime}} L(x_e^{\top}\omega, y_e), \\
f_i(\lambda, \omega)&= \frac{1}{|\mathcal{D}_i|}\sum_{(x_e,y_e)\in \mathcal{D}_i}L(x_e^{\top}\omega, y_e) + \frac{1}{cd}\sum_{{m=1}}^{c}\sum_{j=1}^{d}e^{\lambda_j}\omega^2_{mj},
\end{align*}
where $\omega \in \mathbb{R}^{c\times d}$ is the model parameter, $L$ represents the cross entropy loss, and $|S|$ denotes the cardinality of a set $S$.
In this experiment, $c$ represents the number of classes and $d$ represents the number of features.
The values of $c$ and $d$ are set to $10$ and $784$ respectively.
The training and testing datasets each consist of 60,000 samples, with balanced representation across all classes.
For both \ssldbo\, and SLDBO, the hyperparameters were set as follows: $r_v=100$, ${\alpha}={\eta}=6$, and ${\beta}=12$. For MA-DSBO, we set $T=10$ and $N=20$. For SLAM, we set $b=5$.
In the experiment, we only tested $p=0.3$ for \ssldbo. 
The comparison results of test loss, train loss, and classification accuracy for the three algorithms with respect to CPU time are presented in Fig. \ref{fig:slmnist}. 
These results clearly demonstrate that all the compared algorithms can efficiently solve this problem. Moreover, our proposed algorithm \ssldbo\, performs the best in terms of improved convergence rate and classification accuracy, followed by SLDBO, SLAM and MA-DSBO.

\subsection{Data hyper-cleaning}
In this subsection, we compare the proposed \ssldbo\, with SLDBO on a data hyper-cleaning problem \cite{pmlr-v89-shaban19a} for the Fashion-MNIST dataset \cite{Fashionmnist2017}. The dataset consists of 60,000 images for training and 10,000 images for testing. 
Within the training phase, the 60,000 training images are partitioned into two distinct subsets: a training set of 50,000 images (denoted as $\mathcal{S}_{\mathcal{T}}$) and a validation set of 10,000 images (denoted as $\mathcal{S}_{\mathcal{V}}$).
In the data hyper-cleaning problem, we aim to train a classifier in a corrupted setting. Here, each training data's label is replaced by a random class number with probability {\tt prob} (i.e. the corruption rate). We again use a decentralized ring network with $n$ = 8 clients. 
This problem can be modeled as the bilevel optimization problem \eqref{DBO} with 
\begin{align*}
F_i(\lambda, w) &= \frac{1}{|\mathcal{S}^{i}_{\mathcal{V}}|}\sum_{(x_e, y_e)\in  \mathcal{S}^{i}_{\mathcal{V}}} L(w^{\top}x_e, y_e),\\
f_i(\lambda, w) &= \frac{1}{|\mathcal{S}^{i}_{\mathcal{T}}|}\sum_{(x_e, y_e)\in \mathcal{S}^{i}_{\mathcal{T}}} \sigma(\lambda_e) L(w^{\top}x_e, y_e  ) +\tilde{C}\|w\|^2.
\end{align*}
Here, $L$ represents the cross-entropy loss, $\tilde{C}$ is a regularization parameter, $\sigma(\cdot)$ is the sigmoid function, $\mathcal{S}^{i}_{\mathcal{T}}$ and $\mathcal{S}^{i}_{\mathcal{V}}$ respectively denote the training and validation sets of the $i$th client.
Following the description in \cite{Noniid2020}, we assume that on each client, training examples are independently drawn with class labels following a categorical distribution over $M$ classes parameterized by a vector $\mathbf{q} = (q_1,\ldots,q_M)$ where $q_i \geq 0$ for all $i$ and $\|\mathbf{q}\|_1 = 1$. 
To synthesize a population of non-identical clients, we draw $\mathbf{q}$ from a Dirichlet distribution $\mathbf{q} \sim Dir (\tilde{\alpha} {\mathbf{p}_{cls}})$. Here, ${\mathbf{p}_{cls}}$ represents a prior class distribution over $M$ classes, and $\tilde{\alpha} > 0$ is a concentration parameter that controls the identicalness among clients.  
In the experiment, we set $\tilde{C} = 0.005$.

In Fig. \ref{fig:data_distr_diri}, we display the data class distributions for both the i.i.d. case and the Dirichlet case with $\tilde{\alpha} = 0.1$, {$\mathbf{p}_{cls} =  (1,\ldots,1)$} and $M$ = 10. 
For both cases, our rule for corrupting the $10$ labels is to contaminate them using any one of the other 9 types.
Fig. \ref{fig:dataclean_distribution} shows the results of upper-level train loss, test accuracy and F1-score of the two algorithms for the two types of data distributions with corruption rate {\tt prob}$~= 0.4$,  from which it is clear that data heterogeneity does not influence the performance of the two algorithms. This is because neither of the two compared algorithms depends on the assumption of data heterogeneity.

\begin{figure}[!t]
\centering
\begin{minipage}[t]{\linewidth}
\includegraphics[width=0.46\textwidth, height=0.4\textwidth]{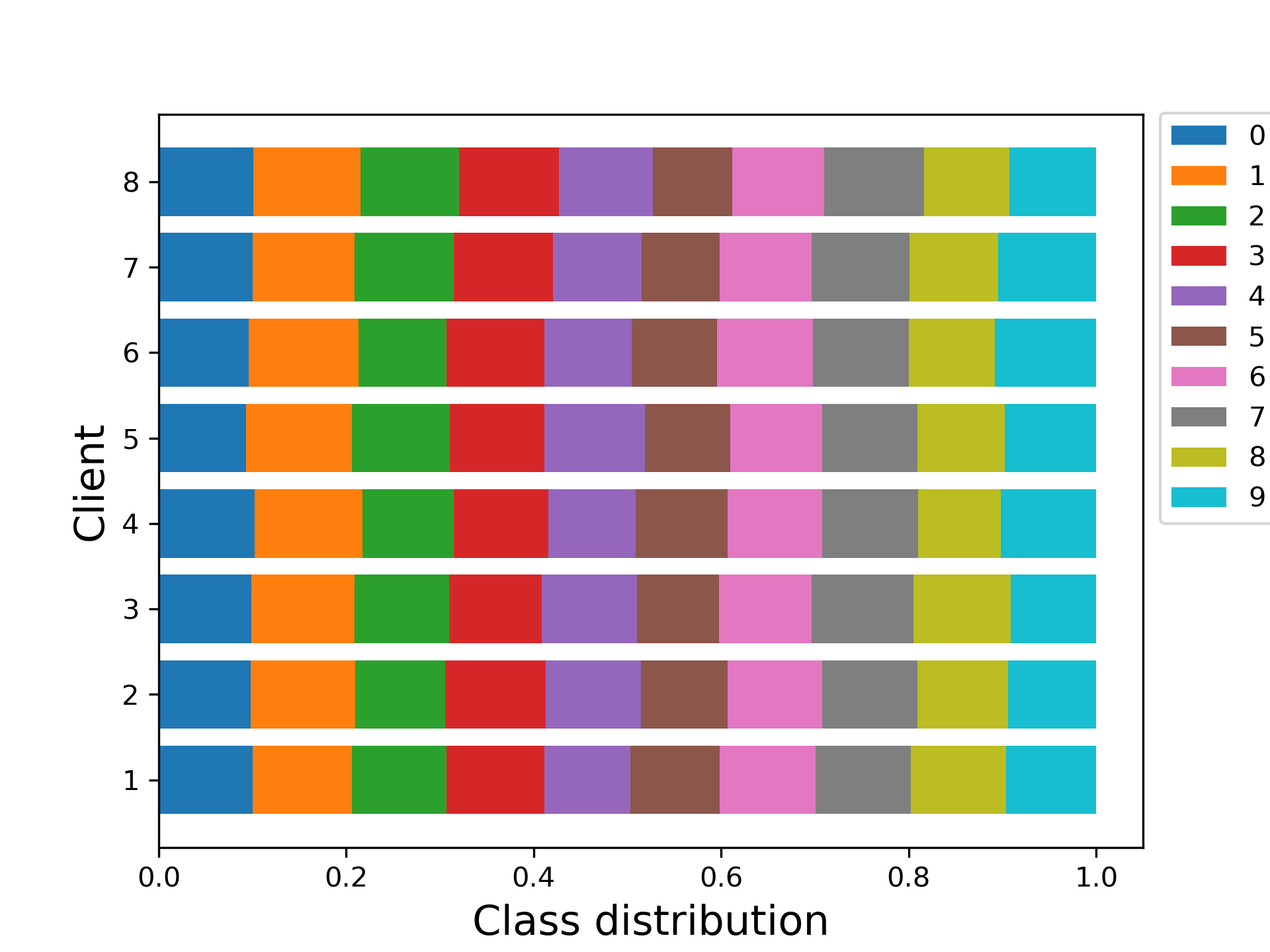}
\quad 
\includegraphics[width=.46\textwidth, height=0.4\textwidth]{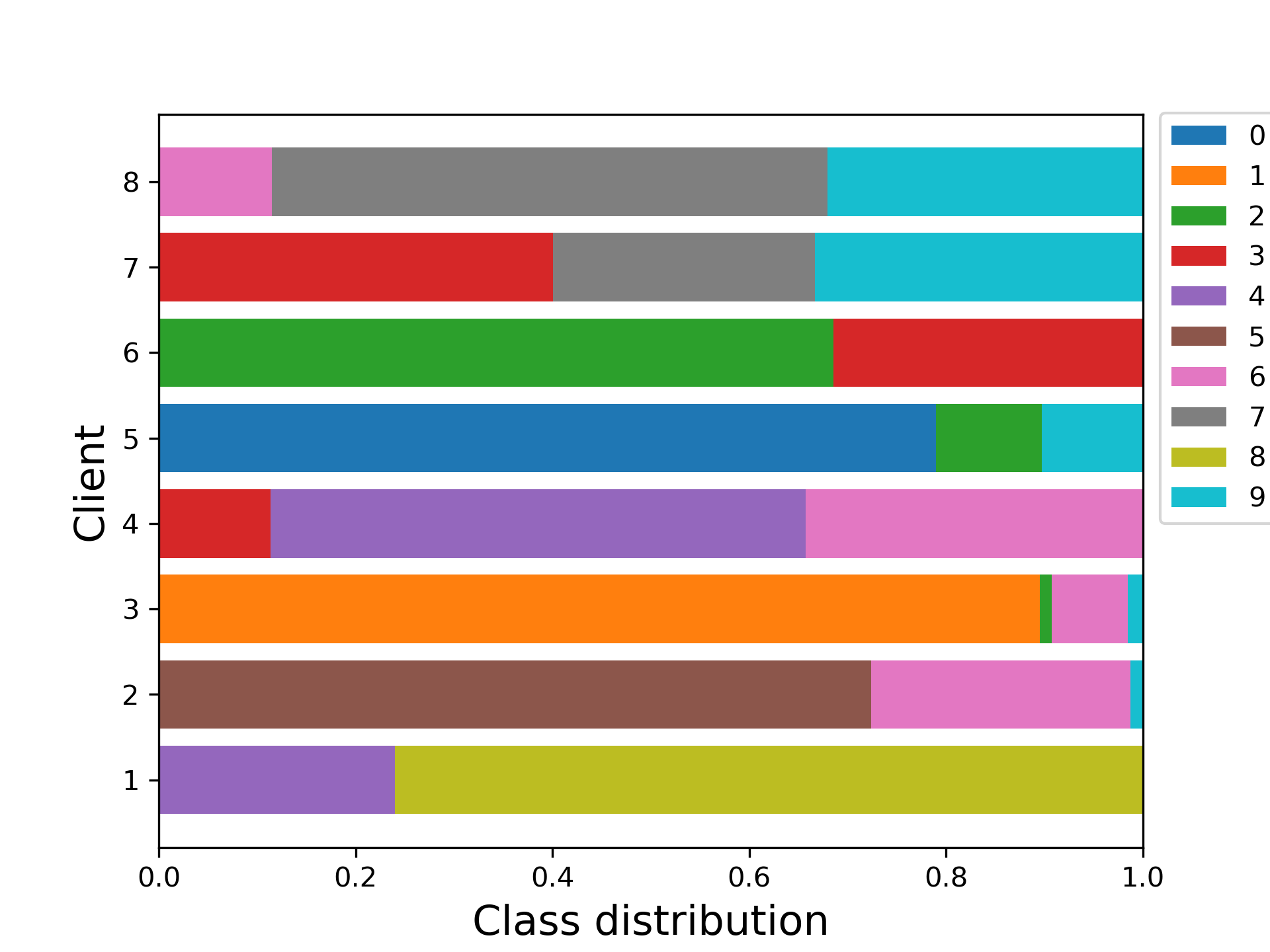}
\end{minipage}
\caption{Data class distributions on each client. The left figure is the i.i.d. case and the right figure is the Dirichlet case with $\tilde{\alpha} = 0.1$ and $\mathbf{p}_{cls} =  (1,\ldots,1)$.}
\label{fig:data_distr_diri}
\end{figure}

\begin{figure}[!t]
\centering
\includegraphics[width=\linewidth]{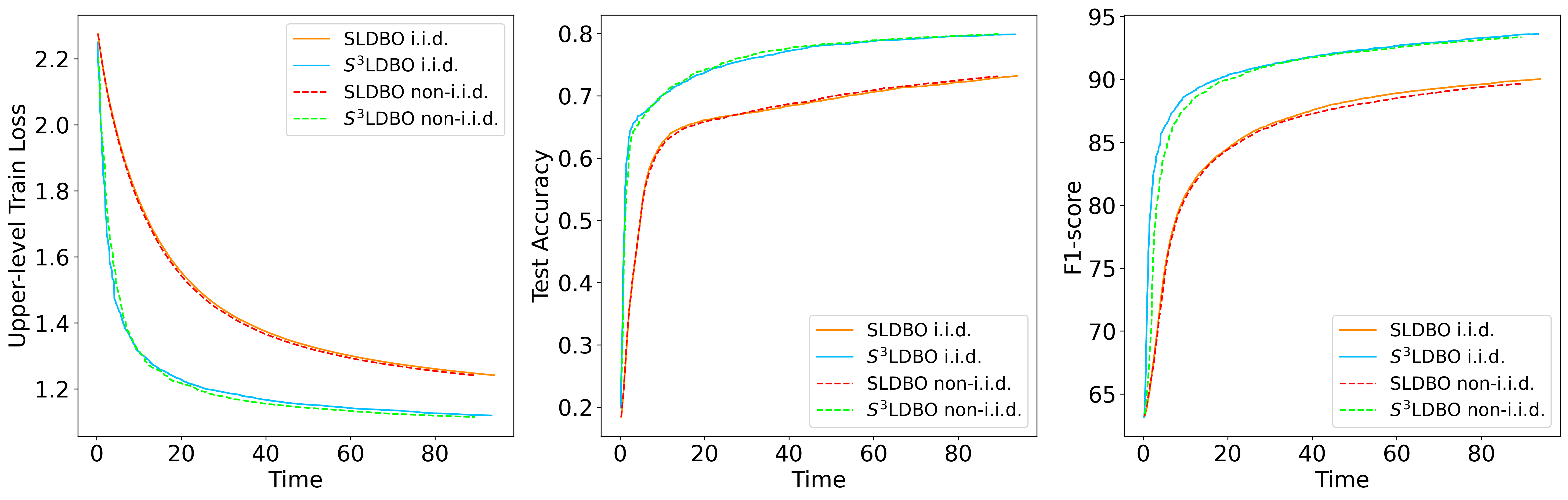}
\caption{Comparison of upper-level train loss, test accuracy, and F1-score for \ssldbo\, and SLDBO on the Fashion-MNIST dataset with i.i.d. and Dirichlet data distributions.}
\label{fig:dataclean_distribution}
\end{figure}
To clearly illustrate the performance of the two compared algorithms,  Fig. \ref{fig:dataclean_time} shows the results
of the upper-level train loss, classification accuracy and F1-score with corruption rates {\tt prob}$~= 0.4$ and {\tt prob}$~= 0.7$ as the CPU time progresses.
Here, classification accuracy refers to the accuracy of classifying data whose labels have not been corrupted, and F1-score measures the quality of the data cleaner \cite{pmlr-v202-shen23c}. 
The detailed experimental results of the final test accuracy, F1-score, and the consumed CPU time (in seconds) are summarized in Table \ref{table}, with an additional corruption rate {\tt prob}$~= 0.1$.
Compared to SLDBO, our algorithm \ssldbo\ achieves a 35\% reduction in CPU time, 
enhances accuracy by approximately 4\%, and increases the F1-score by 2 to 4 points.
From these results, it is evident that \ssldbo\ outperforms SLDBO in terms of both speed and accuracy.

\begin{figure}[!t]
\centering
\includegraphics[width=\linewidth]{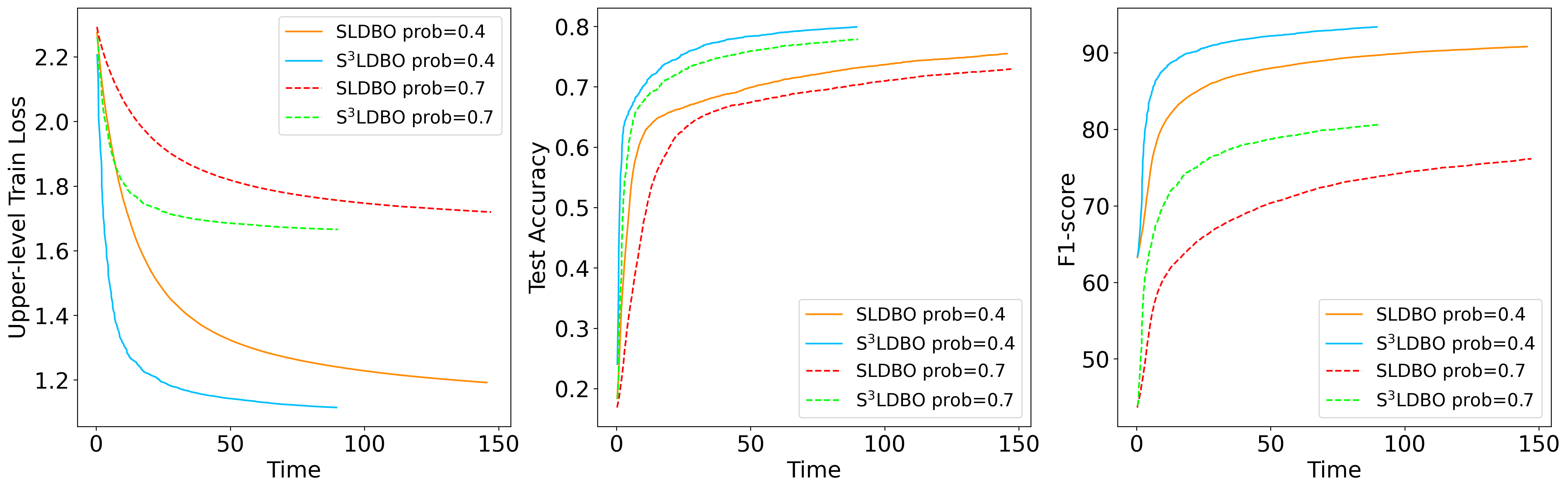}
\caption{Comparison of \ssldbo\, and SLDBO on the Fashion-MNIST dataset regarding CPU time. Two corruption rates are tested: {\tt prob}$~= 0.4$ and {\tt prob}$~= 0.7$.}
\label{fig:dataclean_time}
\end{figure}

\begin{table}[!t]
\caption{Comparison of solution quality and CPU time between S$^3$LDBO and SLDBO under different corruption rates.}
\label{table}
\centering
\begin{tabular}{ccccc}
\hline
Corruption Rate & Algorithm & Test Acc. & F1-score & CPU Time (s)\\
\hline
$\mathrm{prob}=0.1$ & SLDBO & 0.7721 & 92.88 & 143.29\\
                  & S$^3$LDBO & 0.8125 & 94.39 & 93.98\\
                  \midrule
$\mathrm{prob}=0.4$ & SLDBO & 0.7671 & 89.97 & 146.19\\
                  & S$^3$LDBO & 0.8027 & 92.04 & 92.29\\
                  \midrule
$\mathrm{prob}=0.7$ & SLDBO & 0.6752 & 75.38 & 144.40\\
                  & S$^3$LDBO & 0.7127 & 78.93 & 89.60\\
\hline
\end{tabular}
\end{table}
{
\subsection{Decentralized Meta-Learning}

To further evaluate the applicability of the proposed algorithm to networked AI systems, we conduct experiments on a decentralized meta-learning task \cite{pmlr-v70-finn17a, ICLRRaghuRBV20} using the miniImageNet dataset\cite{miniIamgent2016nips}. 
Meta-learning aims to learn a shared representation or initialization that can rapidly adapt to new tasks using only a small number of task-specific samples. 
This setting naturally leads to a bilevel structure, where the lower-level problem performs task-specific adaptation on the support set and the upper-level problem evaluates the adapted model on the query set.

We consider a collection of tasks $\{\mathcal{T}_r:r=1,\ldots,R\}$. 
For each task $\mathcal{T}_r$, let $x$ denote the shared model parameter and $y_r$ denote the task-specific parameter. 
The task loss is defined as
\begin{equation*}
    \ell_r(x,y_r)
    =
    \mathbb{E}_{\xi\in\mathcal{D}_r}
    L(x,y_r;\xi),
\end{equation*}
where $L$ denotes the supervised learning loss. 
In the decentralized setting, the data associated with each task are distributed over $n$ agents. 
For node $i$ and task $\mathcal{T}_r$, we denote the local dataset by $\mathcal{D}_{r,i}$, which is further split into a support set $\mathcal{D}^{\mathrm{tr}}_{r,i}$ and a query set $\mathcal{D}^{\mathrm{val}}_{r,i}$.

The decentralized meta-learning problem \cite{2022KayaalpTSP} can be formulated as the DBO problem \eqref{DBO}. 
For each node $i$, the upper-level and lower-level objectives are given by
\begin{equation*}
    F_i(x,y)
    =
    \frac{1}{R}
    \sum_{r=1}^{R}
    \left[
    \frac{1}{|\mathcal{D}^{\mathrm{val}}_{r,i}|}
    \sum_{\xi\in\mathcal{D}^{\mathrm{val}}_{r,i}}
    L(x,y_r;\xi)
    \right],
\end{equation*}
and
\begin{equation*}
    f_i(x,y)
    =
    \frac{1}{R}
    \sum_{r=1}^{R}
    \left[
    \frac{1}{|\mathcal{D}^{\mathrm{tr}}_{r,i}|}
    \sum_{\xi\in\mathcal{D}^{\mathrm{tr}}_{r,i}}
    L(x,y_r;\xi)
    + \mathcal{R}(y_r)
    \right],
\end{equation*}
where $y=(y_1,\ldots,y_R)$ collects the task-specific parameters and $\mathcal{R}(y_r)$ is a regularization term. 
The lower-level objective adapts the task-specific parameter using the support set, while the upper-level objective measures the query-set performance after adaptation. 
Compared with the previous logistic-regression-based experiments, this experiment is more representative of networked AI applications because it involves neural-network models, task-level adaptation, and distributed task data.

In our implementation, we use a 5-way 10-shot few-shot classification setting. 
The dataset is miniImageNet \cite{miniIamgent2016nips},
which is generated from ImageNet \cite{ILSVRC15} and consists of 100 classes
with each class containing 600 images of size 84 × 84.
Each task contains 5 classes, with 10 support samples and 10 query samples per class. 
The decentralized network consists of $n=8$ nodes. 
We use 1000 training tasks, 200 validation tasks, and 600 testing tasks. 
The batch size is set to 32, and all algorithms are trained for 1500 iterations. 
The models are evaluated every 25 iterations, and each evaluation is averaged over 200 episodes. 
The same task splits, batch size, and evaluation protocol are used for all compared methods.

We compare S$^3$LDBO with SLDBO, MAML\cite{pmlr-v70-finn17a} and ANIL\cite{ICLRRaghuRBV20}. 
MAML\cite{pmlr-v70-finn17a}  is a representative gradient-based meta-learning baseline, while ANIL\cite{ICLRRaghuRBV20} restricts task-specific adaptation mainly to the final layers.
For a fair comparison, the stepsizes of all compared methods are tuned from the same candidate set on the validation tasks. 
For S$^3$LDBO, we set the upper-level stepsize to $\alpha = 0.002$, the lower-level stepsize to $\beta = 0.004$, and the auxiliary-variable stepsize to $\eta = 0.001$. 
For SLDBO, we use the same stepsizes. 
For MAML and ANIL, the meta stepsize and the inner-loop adaptation stepsize are set to $\alpha_{\mathrm{meta}} = 0.002$ and $\alpha_{\mathrm{in}} = 0.002$, respectively. 
These values are selected to provide stable training behavior for all methods under the same miniImageNet task split.
For the snapshot probability in S$^3$LDBO, we use $p=0.6$ in this experiment. 
This value is chosen as a moderate snapshot frequency that balances computational saving and update accuracy: a smaller $p$ further reduces derivative-evaluation cost but may introduce larger snapshot errors, whereas a larger $p$ makes the method closer to SLDBO and weakens the computational advantage. 
Since the influence of $p$ has already been investigated in the synthetic and data hyper-cleaning experiments, the decentralized meta-learning experiment focuses on validating whether the proposed snapshot mechanism remains effective in a more computationally demanding neural-network-based setting. 
Therefore, we report the representative result with $p=0.6$.

\begin{figure}[!t]
\centering
\begin{minipage}[t]{\linewidth}
\includegraphics[width=0.47\textwidth, height=0.4\textwidth]{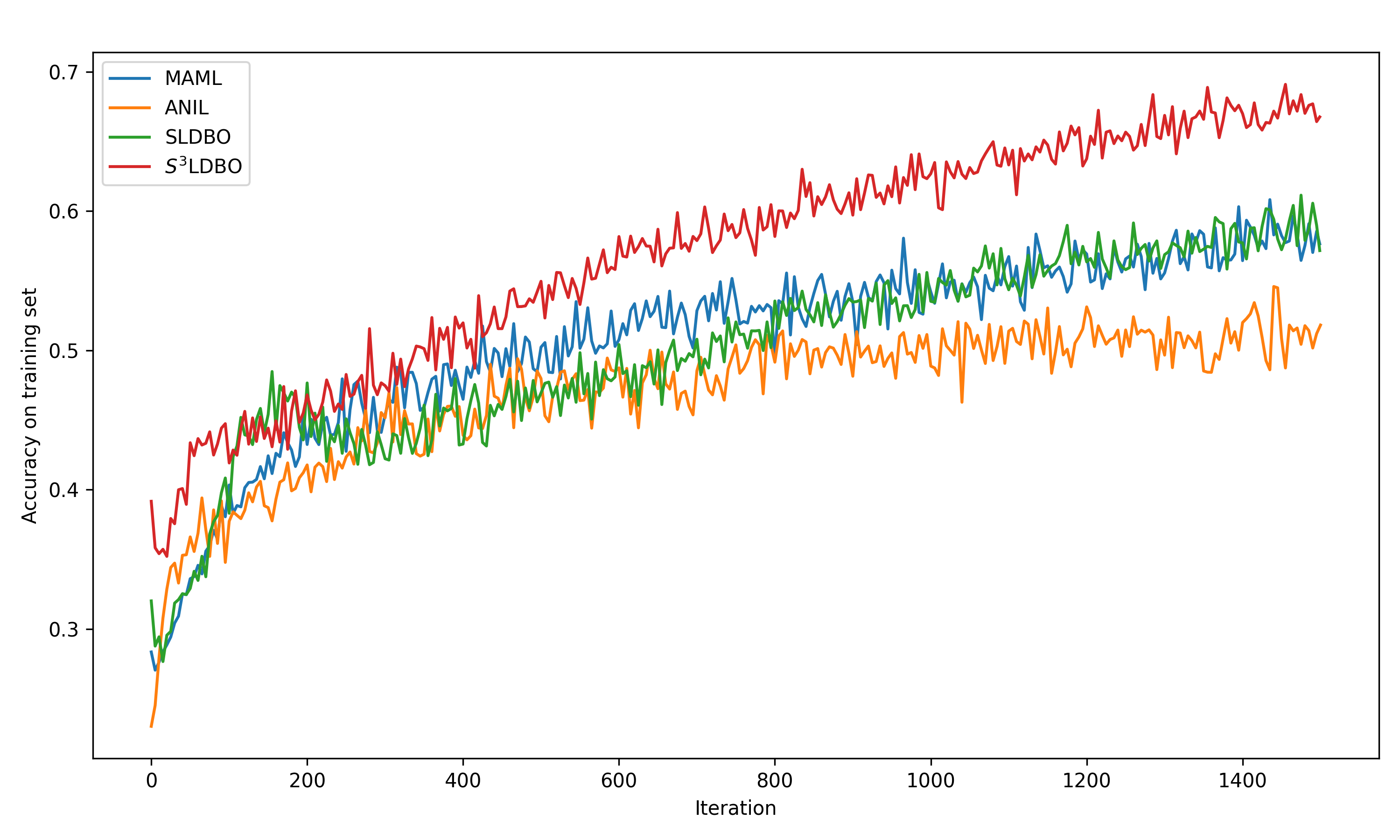}
\quad 
\includegraphics[width=.47\textwidth, height=0.4\textwidth]{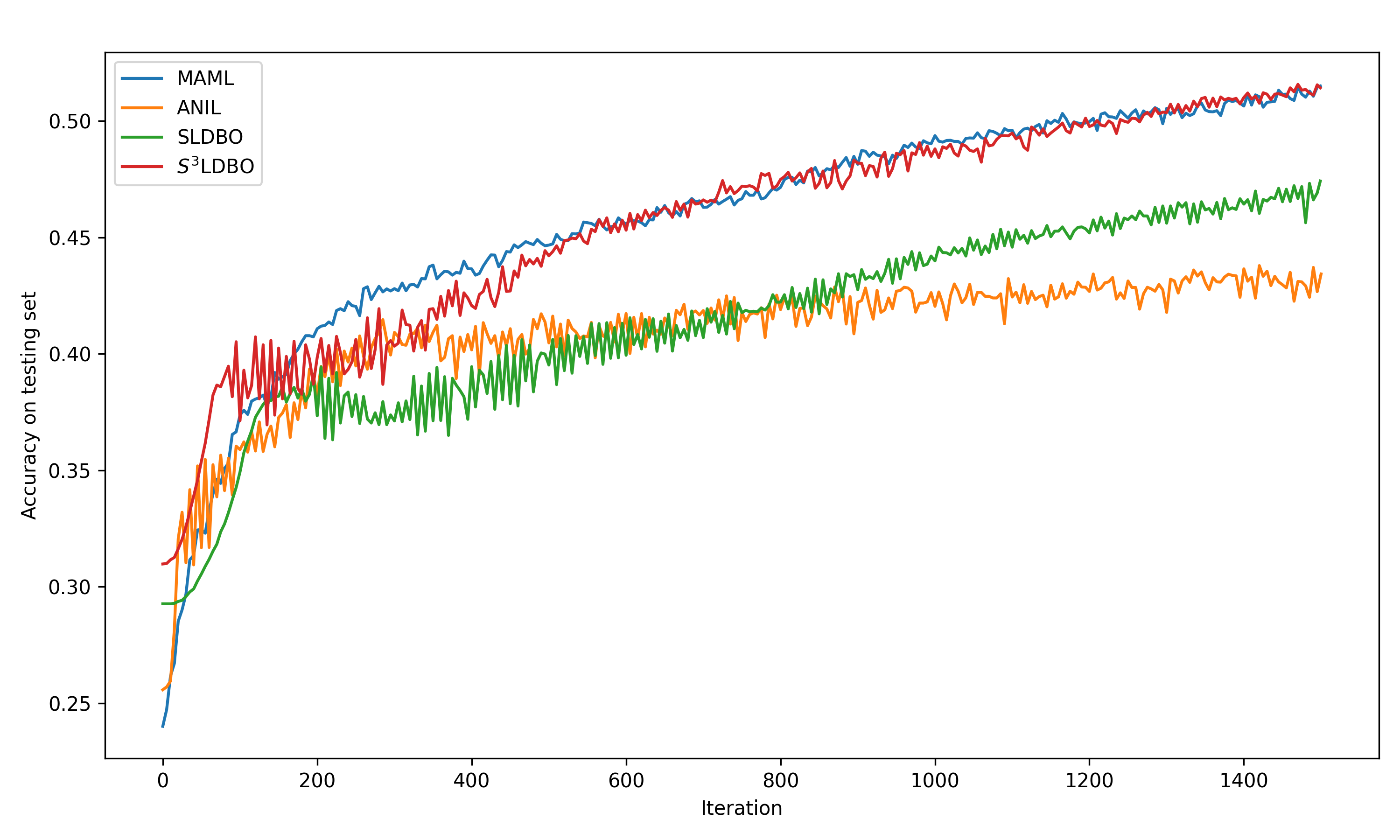}
\end{minipage}
\caption{The accuracy on training set (left) and testing set (right) of different algorithms for the meta-learning problem with respect to the number of iterations.}
\label{fig:meta_learning_iteration}
\end{figure}

\begin{figure}[!t]
\centering
\begin{minipage}[t]{\linewidth}
\includegraphics[width=0.47\textwidth, height=0.4\textwidth]{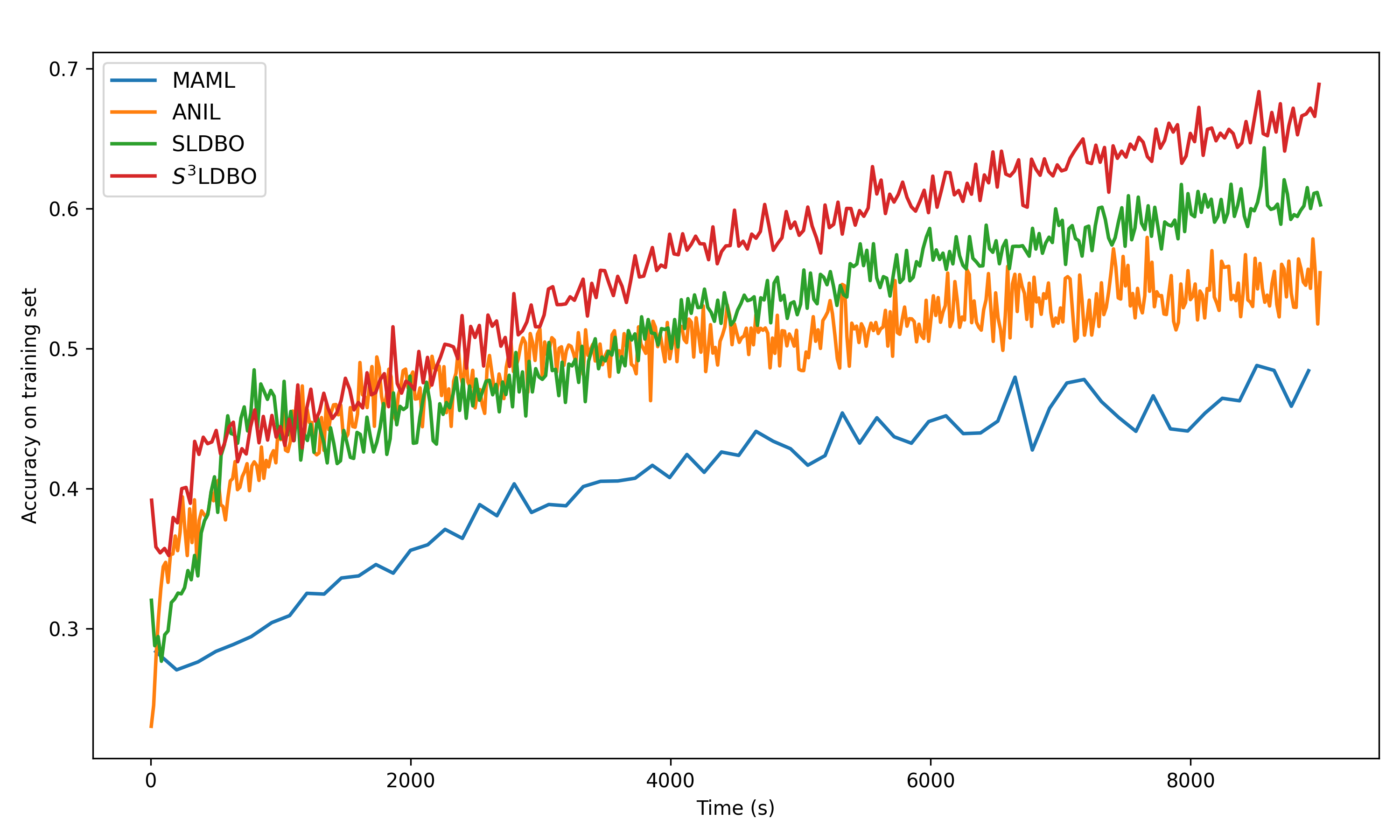}
\quad 
\includegraphics[width=.47\textwidth, height=0.4\textwidth]{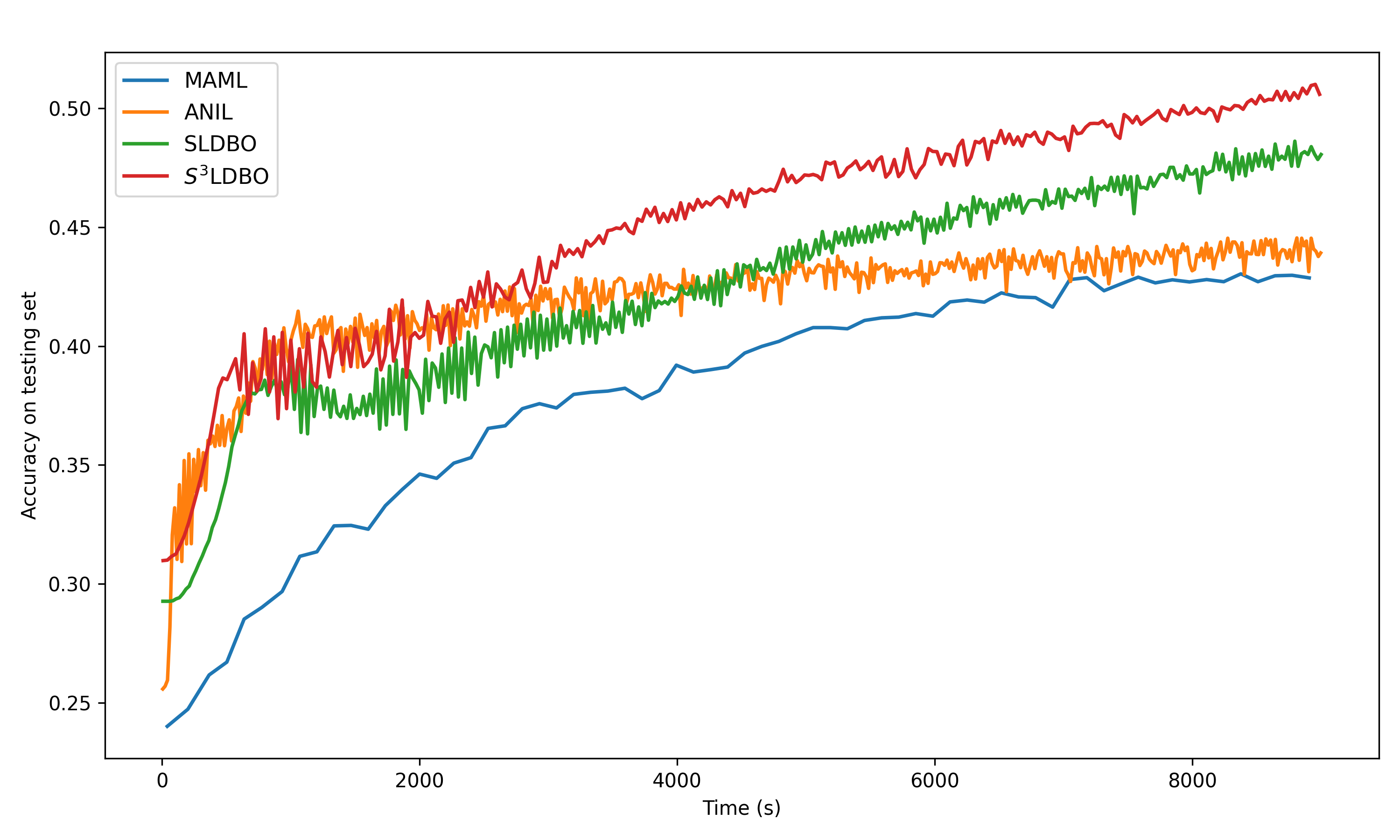}
\end{minipage}
\caption{The accuracy on training set (left) and testing set (right) of different algorithms for the meta-learning problem with respect to CPU time.}
\label{fig:meta_learning_time}
\end{figure}

Fig.~\ref{fig:meta_learning_iteration} and Fig.~\ref{fig:meta_learning_time} compare different algorithms on the meta-learning problem in terms of training and testing accuracy, where Fig.~\ref{fig:meta_learning_iteration} reports the results with respect to the number of iterations and Fig.~\ref{fig:meta_learning_time} reports the results with respect to CPU time.
Under the iteration-based comparison, S$^3$LDBO achieves the highest training accuracy and obtains testing accuracy comparable to MAML, while outperforming SLDBO and ANIL. 
This indicates that the proposed method maintains strong task adaptation performance during the iterative optimization process. 
Under the CPU-time-based comparison, S$^3$LDBO achieves the best performance on both training and testing tasks within the same time budget. 
These results show that the snapshot mechanism not only preserves competitive learning performance in terms of iterations, but also improves practical time efficiency by reducing expensive derivative evaluations. 
Overall, the meta-learning experiment demonstrates that S$^3$LDBO is effective for decentralized task adaptation and is suitable for networked AI applications with limited computational resources.
}

\section{Concluding remarks}\label{sec_conclude}
    In this paper, we propose a novel single-loop algorithm, \ssldbo, which integrates the SLDBO framework with the snapshot gradient tracking technique to efficiently solve decentralized bilevel optimization problems. Unlike other single-loop algorithms that require the computation of full gradients, Jacobian, and Hessian matrices in every iteration -- a time-consuming task in practical settings, \ssldbo\, allows agents to skip these calculations in some iterations, which significantly reduces computational costs in large-scale distributed optimization. 
    We establish the ergodic iteration complexity and the high probability nonergodic iteration complexity of \ssldbo.
    {Numerical experiments on hyperparameter optimization, data hyper-cleaning, and decentralized meta-learning confirm the efficacy of the proposed algorithm.
    These results suggest that snapshot-based DBO algorithm is a promising tool for scalable and adaptive learning in networked AI systems.}
    A key avenue for future research is the extension of \ssldbo\, to the stochastic setting.
    Developing an algorithm that can effectively handle variance from both stochastic data sampling and the snapshot mechanism would require new analytical tools and could significantly broaden the practical applicability of this work to large-scale machine learning problems.

\bibliography{main}
\bibliographystyle{IEEEtran}
\appendices
\onecolumn
\section{Complete proof of the convergence results}\label{appendix}

In this appendix, we provide the proof of our convergence results, i.e., Theorem \ref{cr} and Theorem \ref{cp}. Assumptions \ref{Assump1} and \ref{Assump2} are used throughout the proof. 

Recall that $r_v:= L_{F,0}/\sigma$ and $\rho$ is defined in Assumption \ref{Assump2}. Moreover, we define the following constants for the analysis:
\begin{equation}\label{def-constants-main-text}
	\left\{
	\begin{array}{l}
		L_v:=\left(L_{F,1}+L_{f,2}r_v\right)\left(1+ L_{f,1}/\sigma \right), \quad 
		L_1:=L_{F,1}+L_{f,2}r_v, 
		\medskip  
		\\
		L_{\Phi}:= L_{F, 1}+\frac{2 L_{F, 1} L_{f, 1}+L_{f, 2} L_{F, 0}^2}{\sigma}+\frac{2 L_{f, 1} L_{F, 0} L_{f, 2}+L_{f, 1}^2 L_{F, 1}}{\sigma^2}+\frac{L_{f, 2} L_{f, 1}^2 L_{F, 0}}{\sigma^3},
	\end{array}
	\right.
\end{equation}
where constants such as $\sigma$, 
$L_{F,0}$, $L_{F,1}$, $L_{f,1}$, $L_{f,2}$ are defined in Assumptions \ref{Assump1} and \ref{Assump2}. 

\subsubsection{Notation, constants, and roadmap}
For convenience in the subsequent convergence proof, we rewrite the $k$th iteration of Algorithm \ref{alg:sslDBo} in a more concise form as follows, together with \eqref{yd}, \eqref{vd} and \eqref{xd}:
\begin{subequations}\label{ssldbo_concise}
	\begin{align}
		z_{y,i}^{k} &= t_{y,i}^k + \xi^k(d_{y,i}^k - d_{\ty,i}^k)/p,&
		y_i^{k+1}&=\sum\nolimits_{j=1}^n w_{i j} (y_j^k-\beta  z_{y, j}^k), \label{yupdate}
		\\
		\tilde{y}_i^{k+1} &= \xi^k y_i^{k} + (1-\xi^k)\tilde{y}_i^{k},& 	   
		t_{y, i}^{k+1}&=\sum\nolimits_{j=1}^n w_{i j} \big(t_{y, j}^{k}+ \xi^k (d_{y,j}^{k}-d_{\tilde{y}, j}^{k})\big), \label{tyupdate}
		\\
		z_{v,i}^{k} &= t_{v,i}^k + \xi^k (d_{v,i}^k - d_{\tv,i}^k)/p,& v_i^{k+1}&=\mathcal{P}_{r_v}\left[\sum\nolimits_{j=1}^n w_{i j}( v_j^k+\eta  z_{v, j}^k)\right], \label{vupdate}
		\\
		\tilde{v}_i^{k+1} &= \xi^k v_i^{k} + (1-\xi^k)\tilde{v}_i^{k},&
		t_{v, i}^{k+1}&=\sum\nolimits_{j=1}^n w_{i j} \big(t_{v, j}^{k}+\xi^k (d_{v,j}^{k}-d_{\tilde{v}, j}^{k})\big), \label{tvupdate}
		\\
		z_{x,i}^{k} &= t_{x,i}^k +  \xi^k (d_{x,i}^k - d_{\tx,i}^k)/p,&
		x_i^{k+1}&=\sum\nolimits_{j=1}^n w_{i j} (x_j^k-\alpha z_{x,j}^k),\label{xupdate} 
		\\ 
		\tilde{x}_i^{k+1} &= \xi^k x_i^{k} + (1-\xi^k)\tilde{x}_i^{k},& 
		t_{x, i}^{k+1}&=\sum\nolimits_{j=1}^n w_{i j}\big( t_{x, j}^{k}+\xi^k (d_{x,j}^k-d_{\tilde{x}, j}^{k})\big),\label{txupdate}\\
		d_{\ty,i}^{k+1} &
		= \xi^{k}d_{y,i}^{k} + (1-\xi^{k})d_{\ty,i}^{k} 
		,&
		d_{\tv,i}^{k+1} &
		= \xi^{k}d_{v,i}^{k} + (1-\xi^{k})d_{\tv,i}^{k} 
		,\\
		d_{\tx,i}^{k+1} &
		= \xi^{k}d_{x,i}^{k} + (1-\xi^{k})d_{\tx,i}^{k}.
	\end{align}
\end{subequations}
From \eqref{ssldbo_concise}, we know that $d_{\ty,i}^{k+1}, d_{\tv,i}^{k+1}$ and $d_{\tx,i}^{k+1}$ have the following representations, which will be useful in the subsequent analysis: 
\begin{align}
	d_{\ty,i}^{k+1} &= \nabla_2 f_i(\tx_i^{k+1},\ty_i^{k+1}), \label{tyd}\\
	d_{\tv,i}^{k+1} &= \nabla_2 F_i(\tx_i^{k+1},\ty_i^{k+1}) - \nabla^2_{22} f_i(\tx_i^{k+1},\ty_i^{k+1}) \tv_i^{k+1},\label{tvd}\\
	d_{\tx,i}^{k+1} &= \nabla_1 F_i(\tx_i^{k+1},\ty_i^{k+1}) - \nabla^2_{12} f_i(\tx_i^{k+1},\ty_i^{k+1}) \tv_i^{k+1}.\label{txd}
\end{align}
For the analysis, we further define
\begin{equation}\label{def:bar9}
	\left\{
	\begin{array}{lll}
		{\x^{k}} := \frac{1}{n}\sum_{i=1}^{n}x_{i}^k, & {\y^{k}} := \frac{1}{n}\sum_{i=1}^{n}y_{i}^k, & {\v^{k}} := \frac{1}{n}\sum_{i=1}^{n}v_{i}^k, \smallskip\\
		\d_x^k := \frac{1}{n}\sum_{i=1}^{n}d_{x,i}^k, & \d_y^k := \frac{1}{n}\sum_{i=1}^{n}d_{y,i}^k, & \d_v^k := \frac{1}{n}\sum_{i=1}^{n}d_{v,i}^k, \smallskip\\
		\d_{\tx}^k := \frac{1}{n}\sum_{i=1}^{n}d_{\tx,i}^k,  & {\d_{\ty}^{k}} := \frac{1}{n}\sum_{i=1}^{n}d_{\ty,i}^k,  & \d_{\tv}^k := \frac{1}{n}\sum_{i=1}^{n}d_{\tv,i}^k,
		\smallskip \\
		\t_x^k := \frac{1}{n}\sum_{i=1}^{n}t_{x,i}^k, & \t_y^k := \frac{1}{n}\sum_{i=1}^{n}t_{y,i}^k, & \t_v^k := \frac{1}{n}\sum_{i=1}^{n}t_{v,i}^k,
		\smallskip \\
		\z_x^k := \frac{1}{n}\sum_{i=1}^{n}z_{x,i}^k, & \z_y^k := \frac{1}{n}\sum_{i=1}^{n}z_{y,i}^k, & \z_v^k := \frac{1}{n}\sum_{i=1}^{n}z_{v,i}^k,
	\end{array}
	\right.
\end{equation}
and two additional constants
\begin{equation}\label{def-constants-proof}
	C_1   := 3\max\{(L_{F,1}+r_{v}L_{f,2})^2,L_{f,1}^2\}
	\text{~~and~~}
	C_2  := L_{f,1}^2.  
\end{equation}

\paragraph{Roadmap of the proof.}
Here we briefly describe the roadmap of the proof of Theorem \ref{cr}.
First, we define the Lyapunov function
\begin{align}
	V_k = \;
	& F(\x^k, y ^{*}(\x^k)) - F^{*} + a_1 \|\y^k - y^{*}(\x^k)\|^2 + a_2 \|\v^k - v^{*}(\x^k)\|^2 + \frac{a_3}{n}\sum_{i=1}^n \|x_i^k -\x^k\|^2 \nonumber\\
	& + \frac{a_4}{n} \sum_{i=1}^n \|y_i^k -\y^k\|^2 +\frac{a_5}{n} \sum_{i=1}^n \|v_i^k - \v^k\|^2 + \frac{a_6 \alpha^2}{n} \sum_{i=1}^n\|z_{x,i}^k - \z_{x}^k\|^2 + \frac{a_7 \beta^2}{n}\sum_{i=1}^n \|z_{y,i}^k - \z_{y}^k\|^2  \nonumber \\
	& + \frac{a_8 \eta^2}{n} \sum_{i=1}^n \|z_{v,i}^k - \z_{v}^k\|^2 + \frac{a_9}{n} \sum_{i=1}^n \|x_i^k - \tx_i^k\|^2 + \frac{a_{10}}{n} \sum_{i=1}^n \|y_i^k - \ty_i^k\|^2 + \frac{a_{11}}{n} \sum_{i=1}^n \|v_i^k -\tv_i^k\|^2,\label{Lyapunov}
\end{align}
where  $v^*(x) := \left[\nabla^2_{22} f(x,y^*(x))\right]^{-1}\nabla_2 F(x,y^*(x))$ and $a_1, a_2, \ldots, a_{11}$ are constants. We will show that, for another set of positive constants $A_1, A_2, \ldots, A_{12}$, the following inequality holds 
\begin{align*}
	\E{V_{k+1} - V_{k}} 
	\leq & -\frac{\alpha}{2} \E{\norm{\nabla \Phi(\x_{k})}} -A_1\E{\norm{\d_{x}^k}} -A_2\E{\norm{\y^k - y^{*}(\x^k)}} - A_3 \E{\norm{\v^k - v^{*}(\x^k)}} \\
	& - \frac{A_4}{n}\E{\sum_{i=1}^n\norm{x_i^k -\x^k}} - \frac{A_5}{n}\E{\sum_{i=1}^n\norm{y_i^k -\y^k}} - \frac{A_6}{n}\E{\sum_{i=1}^n\norm{v_i^k -\v^k}}\\
	& - \frac{A_7}{n}\E{\sum_{i=1}^n\norm{z_{x,i}^k - \z_{x}^k}} - \frac{A_8}{n}\E{\sum_{i=1}^n\norm{z_{y,i}^k - \z_{y}^k}} - \frac{A_9}{n}\E{\sum_{i=1}^n\norm{z_{v,i}^k - \z_{v}^k}}\\ 
	& - \frac{A_{10}}{n}\E{\sum_{i=1}^n\norm{x_i^k -\tx_{i}^k}} - \frac{A_{11}}{n}\E{\sum_{i=1}^n\norm{y_i^k -\ty_{i}^k}} - \frac{A_{12}}{n} \E{\sum_{i=1}^n \norm{v_i^k -\tv_{i}^k}}.
\end{align*}
Then, the proof can be obtained by telescoping the above inequality.

\subsubsection{Some lemmas for consensus error of Algorithm \ref{alg:sslDBo}}
Recall that ${\x^{k}}$, ${\y^{k}}$, ${\v^{k}}$, $\d_x^k$, $\d_y^k$, $\d_v^k$, $\d_{\tx}^k$, $\d_{\ty}^k$, $\d_{\tv}^k$, $\t_x^k$, $\t_y^k$, $\t_v^k$, $\z_x^k$, $\z_y^k$ and $\z_v^k$ are defined in \eqref{def:bar9}. 
Let $\mathcal{F}_k$ be the sigma field generated by $\{\xi^j\}_{0\leq j\leq k-1}$ and $\mathbb{E}_k[\cdot]$ be the conditional expectation taken over $\mathcal{F}_k$. Then, for any $k \geq 0$, the variables $\{x_i^k, y_i^k, v_i^k, d_{x,i}^k, d_{y,i}^k, d_{v,i}^k, \tx_i^k, \ty_i^k, \tv_i^k, d_{\tx,i}^k, d_{\ty,i}^k, d_{\tv,i}^k, t_{x,i}^k, t_{y,i}^k, t_{v,i}^k\}$ are measurable with respect to $\mathcal{F}_k$. 
We first prove some useful lemmas.
\begin{lemma}
	For any $0 \leq k \leq K$, there hold
	\begin{align}
		\bar{t}_{x}^k = \bar{d}_{\tilde{x}}^k, \quad \t_{y}^k &= \d_{\ty}^k,\quad \t_{v}^k = \d_{\tv}^k ; \label{bar_t}
		\\
		\Ek{\bar{z}_{x}^k} = {\bar{d}_{x}^k}, \quad  \Ek{\bar{z}_{y}^k} &= {\bar{d}_{y}^k}, \quad 
		\Ek{\bar{z}_{v}^k} = {\bar{d}_{v}^k}. \label{bar_z}
	\end{align}
\end{lemma}
\begin{proof}
	First, by initialization, we have the relation \eqref{bar_t} for $k=0$. 
	Suppose we have shown the relation \eqref{bar_t} for $0, \ldots, k-1$. Next, we prove the relation for $k$.
	
	If $\xi^{k-1} = 0$, then $\tx_i^{k} = \tx_i^{k-1}, \ty_i^{k} = \ty_i^{k-1}$ and $\tv_i^{k} = \tv_i^{k-1}$. By induction, we have 
	$
	\t_{x}^{k} = \t_{x}^{k-1} = \d_{\tx}^{k-1} = \d_{\tx}^{k}.
	$
	If $\xi^{k-1} = 1$, then $\tx_i^{k} = x_i^{k-1}, \ty_i^{k} = y_i^{k-1}$ and $ \tv_i^{k} = v_i^{k-1}$. Again, by induction, we have
	$$
	\t_{x}^{k} = \t_{x}^{k-1} + \frac{1}{n}\sum_{i=1}^n (d_{x,i}^{k-1} - d_{\tx, i}^{k-1}) = \d_{\tx}^{k-1} + \d_{x}^{k-1} -\d_{\tx}^{k-1} = \d_{x}^{k-1} = \d_{\tx}^{k}.
	$$
	Following similar steps, we can show $\t_{y}^{k} = \d_{\ty}^{k}$ and $\t_{v}^{k} = \d_{\tv}^{k}$.
	From the update of $z_{x,i}^k$ in \eqref{xupdate},  we have
	$\bar{z}_{x}^k = \bar{t}_{x}^k + \xi^k (\d_{x}^k - \d_{\tx}^k)/p$. With $\t_{x}^k = \d_{\tx}^k$ and $\Ek{ \xi^k/p }=1$, we obtain $\Ek{\bar{z}_{x}^k} = {\bar{d}_{x}^k}$. Similarly, we also have $\Ek{\bar{z}_{y}^k} = {\bar{d}_{y}^k}$ and $\Ek{\bar{z}_{v}^k} = {\bar{d}_{v}^k}$. 
\end{proof}

\begin{lemma}\label{zx_bounded}
	The sequences $\{z_{x,i}^k\}$,  $\{d_{x,i}^k\}$ and $\{d_{\tx,i}^k\}$ generated by \eqref{ssldbo_concise} satisfy
	\begin{align}
		\Ek{\sum_{i=1}^{n}\|z_{x,i}^{k+1}-\z_x^{k+1}\|^2} &
		\leq \rho \Ek{\sum_{i=1}^n \|z_{x,i}^k - \z_{x}^k\|^2} + \frac{1}{(1-\rho)^2\rho p}{\sum_{i=1}^n \Big(\|d_{x,i}^k - d_{\tx,i}^k\|^2+\Ek{\|d_{x,i}^{k+1} - d_{\tx,i}^{k+1} \|^2}\Big)}. \label{zx-barzx}
	\end{align}
\end{lemma}

\begin{proof}
	By using the update of $z_{x,i}^{k+1}$ in \eqref{xupdate} and defining $\triangle_{x,i}^{k+1} = \xi^{k+1} (d_{x,i}^{k+1} - d_{\tx,i}^{k+1})$, we have 
	\begin{align*}
		z_{x,i}^{k+1}  
		= t_{x,i}^{k+1} + \frac{1}{p} \triangle_{x,i}^{k+1} = \sum_{j=1}^n w_{ij} (t_{x,j}^{k} + \triangle_{x,j}^{k}) + \frac{1}{p}\triangle_{x,i}^{k+1}
		= \sum_{j=1}^n w_{ij} \big(z_{x,j}^{k} + (1-\frac{1}{p}) \triangle_{x,j}^{k}\big) + \frac{1}{p}\triangle_{x,i}^{k+1}.
	\end{align*}

{
    Therefore, we have
\begin{align}
	&\;\Ek{\sum_{i=1}^n \|z_{x,i}^{k+1} - \z_{x}^{k+1}\|^2} 
	=\;\Ek{\sum_{i=1}^n \|
	\sum_{j=1}^n w_{ij}(z_{x,j}^{k}- \frac{1-p}{p}\triangle_{x,j}^{k})
	+ \frac{1}{p}\triangle_{x,i}^{k+1} - \z_{x}^{k} + \z_{x}^{k} -\z_{x}^{k+1}
	\|^2} \nonumber\\
	=&\;\Ek{\sum_{i=1}^n \|
	\sum_{j=1}^n w_{ij}z_{x,j}^{k} - \z_{x}^{k}
	+ \frac{1}{p}\triangle_{x,i}^{k+1}
	- \frac{1-p}{p}\sum_{j=1}^n w_{ij}\triangle_{x,j}^{k}
	\|^2
	+ n\|\z_{x}^{k} - \z_{x}^{k+1}\|^2 }
	+ \Ek{2\sum_{i=1}^n
	\left\langle z_{x,i}^{k+1} - \z_{x}^{k},
	\z_{x}^{k} - \z_{x}^{k+1}\right\rangle}. \nonumber
\end{align}
For the cross term, since the second factor is independent of $i$, we have
\begin{align}
	&\;2\sum_{i=1}^n
	\left\langle z_{x,i}^{k+1} - \z_{x}^{k},
	\z_{x}^{k} - \z_{x}^{k+1}\right\rangle
	=\;2\left\langle
	\sum_{i=1}^n (z_{x,i}^{k+1} - \z_x^k),
	\z_x^k - \z_x^{k+1}
	\right\rangle \nonumber\\
	=&\;2n\left\langle
	\z_x^{k+1} - \z_x^k,
	\z_x^k - \z_x^{k+1}
	\right\rangle 
	=\;-2n\|\z_x^{k+1}-\z_x^k\|^2. \nonumber
\end{align}
Therefore,
\begin{align}
	&\;\Ek{\sum_{i=1}^n \|z_{x,i}^{k+1} - \z_{x}^{k+1}\|^2} \nonumber\\
	=&\;\Ek{\sum_{i=1}^n \|
	\sum_{j=1}^n w_{ij}z_{x,j}^{k} - \z_{x}^{k}
	+ \frac{1}{p}\triangle_{x,i}^{k+1}
	- \frac{1-p}{p}\sum_{j=1}^n w_{ij}\triangle_{x,j}^{k}
	\|^2
	- n\|\z_x^{k+1}-\z_x^k\|^2} \nonumber\\
	\leq&\;\Ek{\sum_{i=1}^n \|
	\sum_{j=1}^n w_{ij}z_{x,j}^{k} - \z_{x}^{k}
	+ \frac{1}{p}\triangle_{x,i}^{k+1}
	- \frac{1-p}{p}\sum_{j=1}^n w_{ij}\triangle_{x,j}^{k}
	\|^2} \nonumber\\
	\leq&\;\rho \Ek{\sum_{i=1}^n \|z_{x,i}^k - \z_{x}^k\|^2}
	+ \frac{1}{1-\rho}\Ek{\sum_{i=1}^n
	\|\frac{1}{p}\triangle_{x,i}^{k+1}
	- \frac{1-p}{p}\sum_{j=1}^n w_{ij}\triangle_{x,j}^{k}
	\|^2} \nonumber\\
	\leq&\; \rho \Ek{\sum_{i=1}^n \|z_{x,i}^k - \z_{x}^k\|^2}
	+ \frac{1}{(1-\rho)^2\rho p}
	{\sum_{i=1}^n \Big(
	\|d_{x,i}^k - d_{\tx,i}^k\|^2
	+\Ek{\|d_{x,i}^{k+1} - d_{\tx,i}^{k+1}\|^2}
	\Big)}. \nonumber
\end{align}
where the first inequality follows from dropping the nonpositive term
``$-n\|\z_x^{k+1}-\z_x^k\|^2$". The second inequality follows from expanding the square, applying
$2\langle a,b\rangle \leq \frac{\rho}{1-\rho}\|a\|^2
+ \frac{1-\rho}{\rho}\|b\|^2$, and then using Lemma \ref{con_sqr}.
This introduces a prefactor $1/\rho$, which combines with the $\rho^2$ contraction in Lemma \ref{con_sqr} and yields the factor $\rho$.
}
    And the last inequality follows from
	\begin{align*}  
		&\Ek{\sum_{i=1}^n \|\frac{1}{p}\triangle_{x,i}^{k+1} - \frac{1-p}{p}\sum_{j=1}^n w_{ij}\triangle_{x,j}^{k}\|^2}\\ 
		\leq & \frac{1}{(1-\rho)p^2}
		\Ek{(\xi^{k+1})^2}
		\Ek{\sum_{i=1}^n \|d_{x,i}^{k+1} - d_{\tx,i}^{k+1}\|^2} + \frac{(1-p)^2}{\rho p^2}
		\Ek{(\xi^{k})^2}
		\sum_{i=1}^n \|\sum_{j=1}^n w_{ij}(d_{x,j}^{k} - d_{\tx,j}^{k})\|^2 \\
		= & \frac{1}{(1-\rho)p}\Ek{\sum_{i=1}^n \|d_{x,i}^{k+1} - d_{\tx,i}^{k+1}\|^2} + \frac{(1-p)^2}{\rho p} {\sum_{i=1}^n \|\sum_{j=1}^n w_{ij}(d_{x,j}^{k} - d_{\tx,j}^{k})\|^2} \\
		\leq &\frac{1}{\rho(1-\rho)p}{\sum_{i=1}^n \Big(\Ek{\|d_{x,i}^{k+1} - d_{\tx,i}^{k+1}\|^2}+\|d_{x,i}^{k} - d_{\tx,i}^{k} \|^2\Big)},
	\end{align*}
	where the last inequality follows from multiplying the first term by $1/\rho$, the second term by $1/\big((1-\rho)(1-p)^2\big)$, and then applying Lemma \ref{con_sqr} (a). 
\end{proof}
\begin{remark}\label{rmkA1}
	The following inequalities can be derived by following similar steps as in Lemma \ref{zx_bounded}. Details are omitted to avoid redundancy.
	\begin{align*}
		\Ek{\sum_{i=1}^{n}\|z_{y,i}^{k+1}-\z_y^{k+1}\|^2} &
		\leq \rho \Ek{\sum_{i=1}^n \|z_{y,i}^k - \z_{y}^k\|^2} + \frac{1}{(1-\rho)^2\rho p}{\sum_{i=1}^n \Big(\|d_{y,i}^k - d_{\ty,i}^k\|^2+\Ek{\|d_{y,i}^{k+1} - d_{\ty,i}^{k+1} \|^2}\Big)},   
		\\
		\Ek{\sum_{i=1}^{n}\|z_{v,i}^{k+1}-\z_v^{k+1}\|^2} &
		\leq \rho \Ek{\sum_{i=1}^n \|z_{v,i}^k - \z_{v}^k\|^2} + \frac{1}{(1-\rho)^2\rho p}{\sum_{i=1}^n \Big(\|d_{v,i}^k - d_{\tv,i}^k\|^2+\Ek{\|d_{v,i}^{k+1} - d_{\tv,i}^{k+1} \|^2}\Big)}.  
	\end{align*}
\end{remark}

Then, we will bound the terms related to the consensus error in \eqref{Lyapunov}.
\begin{lemma}\label{bar}
	The sequence $\{v_i^k\}$ generated by \eqref{ssldbo_concise} satisfies
	\begin{align*}
		\Ek{\sum_{i=1}^{n}\|v_i^{k+1}-\v^{k+1}\|^2} \leq \rho {\sum_{i=1}^{n}\|v_i^{k} - \v^k\|^2} + \frac{\rho^2\eta ^2}{1-\rho} \Ek{\sum_{i=1}^{n}\|z_{v,i}^k - \z_v^k\|^2}.
	\end{align*}
\end{lemma}
\begin{proof}
	In fact, we have
	\begin{align*}
		\Ek{\sum_{i=1}^{n}\|v_i^{k+1}&-\v^{k+1}\|^2}
		=\Ek{\sum_{i=1}^{n} \|\mathcal{P}_{r_v}\left[\sum\nolimits_{j=1}^{n}w_{ij}(v_j^{k}+\eta z_{v,j}^k)\right]-\frac{1}{n}\sum_{s=1}^{n}\mathcal{P}_{r_v}\left[\sum\nolimits_{j=1}^{n}w_{sj}(v_j^{k}+\eta z_{v,j}^k)\right]\|^2}\\
		& \leq  \Ek{\sum_{i=1}^{n} \|\sum_{j=1}^{n}w_{ij}(v_j^{k}+\eta z_{v,j}^k)-\frac{1}{n}\sum_{s=1}^{n}\sum_{j=1}^{n}w_{sj}(v_j^{k}+\eta z_{v,j}^k)\|^2}\\
		&\leq  \frac{1}{\rho}  \Ek{\sum_{i=1}^{n}\|\sum\nolimits_{j=1}^{n}w_{ij}v_j^{k} - \v^k\|^2} 
		+ \frac{\eta^2}{1-\rho} 
		\Ek{\sum_{i=1}^{n}\left\|\sum\nolimits_{j=1}^{n}w_{ij}z_{v,j}^k - \z_v^k\right\|^2}\\
		&\leq  \rho {\sum_{i=1}^{n}\|v_i^{k} - \v^k\|^2} + \frac{\rho^2\eta ^2}{1-\rho} \Ek{\sum_{i=1}^{n}\|z_{v,i}^k - \z_v^k\|^2} ,
	\end{align*}
	where the first inequality follows from Lemma \ref{proj},
	the second follows from Cauchy-Schwarz inequality and the notation defined in \eqref{def:bar9},
	the third follows from Lemma \ref{con_sqr} (b). 
\end{proof}

\begin{remark}\label{bar_re}
	By following similar deductions as in Lemma \ref{bar}, we can derive
	\begin{align*}
		\Ek{\sum_{i=1}^{n}\|x_i^{k+1}-\x^{k+1}\|^2} \leq \rho{\sum_{i=1}^{n}\|x_i^{k} - \x^k\|^2} + \frac{\rho^2\alpha ^2}{1-\rho} \Ek{\sum_{i=1}^{n}\|z_{x,i}^k - \z_x^k\|^2},\\ 
		\Ek{\sum_{i=1}^{n}\|y_i^{k+1}-\y^{k+1}\|^2} \leq \rho{\sum_{i=1}^{n}\|y_i^{k} - \y^k\|^2} + \frac{\rho^2\beta ^2}{1-\rho} \Ek{\sum_{i=1}^{n}\|z_{y,i}^k - \z_y^k\|^2}. 
	\end{align*}
	The details are omitted.
\end{remark}

Next, we bound the term $\Ek{\sum_{i=1}^n \|d_{x,i}^k-d_{\tx,i}^k\|^2}$ in \eqref{zx-barzx}, as well as $\Ek{\sum_{i=1}^n \|d_{y,i}^k-d_{\ty,i}^k\|^2}$ and $\Ek{\sum_{i=1}^n \|d_{v,i}^k-d_{\tv,i}^k\|^2}$. 
\begin{lemma}\label{dk+1-dk}
	The sequences  $\{d_{x,i}^{k}\}$, $\{d_{v,i}^{k}\}$ and $\{d_{y,i}^{k}\}$ generated by \eqref{ssldbo_concise} satisfy
	\begin{align}
		\Ek{\sum_{i=1}^{n}\|d_{x,i}^{k}-d_{\tx,i}^{k}\|^2}&\leq C_1 {\sum_{i=1}^{n}\left(\|x_i^k-\tx_i^{k}\|^2+\|y_i^k-\ty_i^{k}\|^2+\|v_i^k-\tv_i^{k}\|^2\right)},\label{dx}\\
		\Ek{\sum_{i=1}^{n}\|d_{v,i}^{k}-d_{\tv,i}^{k}\|^2}&\leq C_1 {\sum_{i=1}^{n}\left(\|x_i^k-\tx_i^{k}\|^2+\|y_i^k-\ty_i^{k}\|^2+\|v_i^k-\tv_i^{k}\|^2\right)},\label{dv}\\
		\Ek{\sum_{i=1}^{n}\|d_{y,i}^{k}-d_{\ty,i}^{k}\|^2}&\leq C_2  {\sum_{i=1}^{n}\left(\|x_i^k-\tx_i^{k}\|^2+\|y_i^k-\ty_i^{k}\|^2\right)}\label{dy},
	\end{align}
	where $C_1$ and $C_2$ are defined in \eqref{def-constants-proof}.
\end{lemma}

\begin{proof}
	By the definition of $d_{x,i}^k$ in \eqref{xd}  
	and the representation of $d_{\tx,i}^k$ in \eqref{txd}, it follows that 
	\begin{align*}
		&\Ek{\sum_{i=1}^{n}\|d_{x,i}^{k}-d_{\tx,i}^{k}\|^2}
		\leq 3\Ek{\sum_{i=1}^{n}\|\nabla_1 F_i(x^k_i, y^k_i)-\nabla_1F_i(\tx^{k}_i, \ty^{k}_i)\|^2}\\
		&\ \ \quad  +3\Ek{\sum_{i=1}^{n}\|(\nabla^2_{12} f_i (x^k_i, y^k_i) -\nabla^2_{12} f_i (\tx^{k}_i, \ty^{k}_i)) v^{k}_i\|^2} +3\Ek{\sum_{i=1}^{n}\|\nabla^2_{12} f_i (\tx^{k}_i, \ty^{k}_i)(v^{k}_i-\tv^{k}_i)\|^2}\\
		\leq & \, 3(L_{F,1}+r_{v}L_{f,2})^2 \Ek{\sum_{i=1}^{n}\|x^k_i-\tx^{k}_i\|^2}+ 3(L_{F,1}+r_{v}L_{f,2})^2\Ek{\sum_{i=1}^{n}\|y^k_i-\ty^{k}_i\|^2} \\
		& +3L_{f,1}^2\Ek{\sum_{i=1}^{n}\|v^k_i-\tv^{k}_i\|^2}
		\leq C_1 {\sum_{i=1}^{n}\left(\|x_i^k-\tx_i^{k}\|^2+\|y^k_i-\ty^{k}_i\|^2+\|v_i^k-\tv_i^{k}\|^2\right)},
	\end{align*}
	where the second inequality is due to Assumption \ref{Assump1} (b) and (c) and $\|v_i^k\| \leq r_v$ for all $i$.
	Similarly, by \eqref{vd} and \eqref{tvd}, we can  derive 
	\begin{align*}
		\Ek{\sum_{i=1}^{n}\|d_{v,i}^{k}-d_{\tv,i}^{k}\|^2}\leq C_1 {\sum_{i=1}^{n}\left(\|x_i^k-\tx_i^{k}\|^2+\|y^k_i-\ty^{k}_i\|^2+\|v_i^k-\tv_i^{k}\|^2\right)},
	\end{align*}
	and by 
	\eqref{yd} and \eqref{tyd}, we can also derive
	\begin{align*}
		&\Ek{\sum_{i=1}^{n}\|d_{y,i}^{k}-d_{\ty,i}^{k}\|^2}\leq \Ek{\sum_{i=1}^{n}\|\nabla_2 f_i(x_i^k,y_i^{k})-\nabla_2 f_i(\tx_i^{k},\ty_i^{k})\|^2}\\
		&\ \ \leq L_{f,1}^2\Ek{\sum_{i=1}^{n}\|x_i^k-\tx_i^{k}\|^2} + L_{f,1}^2\Ek{\sum_{i=1}^{n}\|y_i^k-\ty_i^{k}\|^2}
		= C_2 {\sum_{i=1}^{n}\left(\|x_i^k-\tx_i^{k}\|^2+\|y_i^k-\ty_i^{k}\|^2\right)}.
	\end{align*}
	This completes the proof.
\end{proof}

\begin{lemma}\label{lem:xk+1-txk+1}
	The sequences $\{x^k_i\}$ and $\{\tx^k_i\}$ generated by \eqref{ssldbo_concise} satisfy
	\begin{align}
		\Ek{\sum_{i=1}^n\|x_i^{k+1} - \tx_i^{k+1}\|^2} \leq 
		\frac{1-p}{\rho}{\sum_{i=1}^n\|x_i^k - \tx_i^k\|^2} + \frac{1}{1-\rho}\Ek{\sum_{i=1}^n\|x_i^{k+1} - x_i^{k}\|^2}.\label{xk+1-txk+1}
	\end{align}
\end{lemma}

\begin{proof}
	From the update of $\tx_i^{k+1}$ in \eqref{txupdate}, we get 
	\begin{align*}
		\Ek{\sum_{i=1}^n\|x_i^{k+1} - \tx_i^{k+1}\|^2}
		&= \Ek{\sum_{i=1}^n\|x_i^{k+1} - x_i^{k} + x_i^k - \tx_i^{k+1}\|^2}\\
		&\leq \frac{1}{1-\rho}\Ek{\sum_{i=1}^n\|x_i^{k+1} - x_i^{k}\|^2} + 
		\frac{1}{\rho}\Ek{\sum_{i=1}^n\|(1-\xi^k)(x_i^k - \tx_i^k)\|^2}\\
		&= \frac{1}{1-\rho}\Ek{\sum_{i=1}^n\|x_i^{k+1} - x_i^{k}\|^2} + \frac{1-p}{\rho}{\sum_{i=1}^n\|x_i^k - \tx_i^k\|^2},
	\end{align*}
	where the inequality is due to the Cauchy-Schwarz inequality.
\end{proof}

\begin{remark}
	The following inequalities can be derived by following similar steps as in Lemma \ref{lem:xk+1-txk+1} (details are omitted to avoid redundancy):
	\begin{align}
		\Ek{\sum_{i=1}^n\|v_i^{k+1} - \tv_i^{k+1}\|^2} \leq 
		\frac{1-p}{\rho}{\sum_{i=1}^n\|v_i^k - \tv_i^k\|^2} + \frac{1}{1-\rho}\Ek{\sum_{i=1}^n\|v_i^{k+1} - v_i^{k}\|^2}, \label{vk+1-tvk+1}\\
		\Ek{\sum_{i=1}^n\|y_i^{k+1} - \ty_i^{k+1}\|^2} \leq 
		\frac{1-p}{\rho}{\sum_{i=1}^n\|y_i^k - \ty_i^k\|^2} + \frac{1}{1-\rho}\Ek{\sum_{i=1}^n\|y_i^{k+1} - y_i^{k}\|^2}.\label{yk+1-tyk+1}
	\end{align}
\end{remark}

Next, we bound the term $\Ek{\sum_{i=1}^n \|v_i^{k+1} - v_i^k\|^2}$ in \eqref{vk+1-tvk+1}.

\begin{lemma}\label{lem:vk+1-vk}
	The sequence $\{v_i^k\}$ generated by \eqref{ssldbo_concise} satisfies
	\begin{align*}
		\Ek{\sum_{i=1}^n \|v_i^{k+1} - v_i^k\|^2} \leq& 8{\sum_{i=1}^n \|v_i^k -\v^k\|^2} + 4\eta^2 \Ek{\sum_{i=1}^n \|z_{v,i}^k -\z_{v}^k\|^2} + 4n\eta^2{\|\d_{v}^k\|^2}
		+4\eta^2\frac{1-p}{p}{\sum_{i=1}^n\|{d}_{v,i}^k - {d}_{\tv,i}^k\|^2}.
	\end{align*}
\end{lemma}

\begin{proof}
	It follows from the update of $z_{v,i}^k$ in \eqref{vupdate} that 
	$$
	\bar{z}_{v}^k = \bar{t}_v^k + \xi^k (\d_{v}^k - \d_{\tv}^k) / p
	\stackrel{(\ref{bar_t})}=\d_{\tv}^k + \xi^k (\d_{v}^k - \d_{\tv}^k)/p.
	$$
	Then, we have  $\z^k_{v}-\d^k_{v} = (1-\xi^k / p) (\d_{\tv}^k-\d_{v}^k)$. 
	Since $\Ek{\xi^k}=p$, we have $\Ek{1- \xi^k/p}=0$. Thus, we have
	\begin{align}
		\Ek{\langle \bar{z}_{v}^k-\bar{d}_v^k, \bar{d}_v^k \rangle}
		=\Ek{\langle (1- \xi^k/p) (\d_{\tv}^k - \d_{v}^k), \d_{v}^k \rangle} 
		= \Ek{1-\xi^k/p} \langle \d_{\tv}^k - \d_{v}^k, \d_{v}^k\rangle=0. \label{z-dxd=0}
	\end{align}
	By following the update of $v_i^{k+1}$ in \eqref{vupdate} , we derive 
	\begin{align*}
		&\Ek{\sum_{i=1}^{n} \| v_{i}^{k+1} - v_{i}^{k} \|^2} = \Ek{\sum_{i=1}^{n} \| \mathcal{P}_{r_v} [\sum_{j=1}^{n} w_{ij} (v_{j}^{k} + \eta z_{v,j}^{k})] - v_{i}^{k} \|^2]} \leq \Ek{\sum_{i=1}^{n} \| \sum_{j=1}^{n} w_{ij} (v_{j}^{k} + \eta z_{v,j}^{k}) - v_{i}^{k} \|^2} \\&\leq 4 \Ek{\sum_{i=1}^{n} \| \sum_{j=1}^{n} w_{ij} v_{j}^{k} - \bar{v}^{k} \|^2} + 4 \Ek{\sum_{i=1}^{n} \| \bar{v}^{k} - v_{i}^{k} \|^2} + 4 \eta^2\Ek{ \sum_{i=1}^{n} \| \sum_{j=1}^{n} w_{ij} z_{v,j}^{k} - \bar{z}_{v}^{k} \|^2} + 4 \eta^2 \Ek{\sum_{i=1}^{n} \| \bar{z}_{v}^{k} \|^2} \\
		&\leq 8 \Ek{\sum_{i=1}^{n} \| v_{i}^{k} - \bar{v}^{k} \|^2} + 4 \eta^2\Ek{ \sum_{i=1}^{n} \| z_{v,i}^{k} - \bar{z}_{v}^{k} \|^2} + 4 n\eta^2 \Ek{\| \bar{d}_{v}^{k} \|^2 + \|\bar{z}_{v}^k - \bar{d}_{v}^k\|^2} \\
		&\leq 8{\sum_{i=1}^{n} \| v_{i}^{k} - \bar{v}^{k} \|^2} + 4 \eta^2\Ek{ \sum_{i=1}^{n} \| z_{v,i}^{k} - \bar{z}_{v}^{k} \|^2} + 4n\eta^2{\| \bar{d}_{v}^{k} \|^2} + 4\eta^2\frac{1-p}{p}{\sum_{i=1}^n\|{d}_{v,i}^k - {d}_{\tv,i}^k\|^2},
	\end{align*}
	where the first inequality holds because the projection operator is non-expansive, the second   follows from $\left(\sum_{l=1}^{s} a_l\right)^2\leq s\sum_{l=1}^{s} a_l^2$,  the third    is due to Lemma \ref{con_sqr}  and \eqref{z-dxd=0}, and 
	the last one follows from
	\begin{align}
		n\Ek{\|\d_{v}^k - \z_{v}^k\|^2}  =\, & n\Ek{\|\mean (1-\xi^k/p)(d_{\tv,i}^k - d_{v,i}^k)\|^2}=  \frac{1}{n}  \Ek{\|\sum_{i=1}^n (1-\xi^k/p)(d_{\tv,i}^k - d_{v,i}^k)\|^2} \nonumber\\
		\leq\, &  
		\Ek{\sum_{i=1}^n \|(1- \xi^k/p)(d_{\tv,i}^k - d_{v,i}^k)\|^2} = \Ek{(1-\xi^k/p)^2}\sum_{i=1}^n \|d_{\tv,i}^k - d_{v,i}^k\|^2 \nonumber \\
		=\, & \big(p(1-1/p)^2 + (1-p)\big){\sum_{i=1}^n \|d_{\tv,i}^k - d_{v,i}^k\|^2}
		=  \frac{1-p}{p}{\sum_{i=1}^n \|d_{\tv,i}^k - d_{v,i}^k\|^2}. \label{bard-barz}
	\end{align}
	This completes the proof. 
\end{proof}

\begin{remark}\label{rem:xk+1-xk}
Similarly, we can deduce the following inequalities. The details are omitted.
\begin{align*}
	\Ek{\sum_{i=1}^n \|x_i^{k+1} - x_i^k\|^2} \leq& 8 {\sum_{i=1}^n \|x_i^k -\x^k\|^2} + 4\alpha^2\Ek{ \sum_{i=1}^n \|z_{x,i}^k -\z_{x}^k\|^2} + 4n\alpha^2{\|\d_{x}^k\|^2}
	+4\alpha^2\frac{1-p}{p}{\sum_{i=1}^n\|{d}_{x,i}^k - {d}_{\tx,i}^k\|^2},\\
	\Ek{\sum_{i=1}^n \|y_i^{k+1} - y_i^k\|^2} \leq& 8{\sum_{i=1}^n \|y_i^k -\y^k\|^2} + 4\beta^2 \Ek{\sum_{i=1}^n \|z_{y,i}^k -\z_{y}^k\|^2} + 4n\beta^2{\|\d_{y}^k\|^2}
	+4\beta^2\frac{1-p}{p}{\sum_{i=1}^n\|{d}_{y,i}^k - {d}_{\ty,i}^k\|^2}.
\end{align*}
\end{remark}
Next, we bound $\Ek{\sum_{i=1}^n (\norm{d_{\star,i}^k - d_{\tilde{\star},i}^k} + \norm{d_{\star,i}^{k+1} - d_{\tilde{\star},i}^{k+1}})}$ for $\star = x, y, v$ respectively. 
By combining \eqref{dx}-\eqref{yk+1-tyk+1},
we derive 
\begin{align}
&\Ek{\sum_{i=1}^n \big(\norm{d_{x,i}^k - d_{\tx,i}^k} + \norm{d_{x,i}^{k+1} - d_{\tx,i}^{k+1}}\big)} \nonumber  \\
\leq &C_1 \Ek{\sum_{i=1}^n \big(\norm{x_i^{k+1} - \tx_{i}^{k+1}} + \norm{y_i^{k+1} - \ty_{i}^{k+1}} + \norm{v_i^{k+1} - \tv_{i}^{k+1}} + \norm{x_i^{k} - \tx_{i}^{k}} + \norm{y_i^{k} - \ty_{i}^{k}} + \norm{v_i^{k} - \tv_{i}^{k}}\big)} \nonumber \\
\leq & C_1\sum_{i=1}^n\Big(
\frac{1-p+\rho}{\rho}\big(\norm{x_i^k -\tx_{i}^k} + \norm{y_i^k -\ty_{i}^k} +\norm{v_i^k -\tv_{i}^k}\big)
\nonumber\\
& \quad + \frac{1}{1-\rho}\Ek{\sum_{i=1}^n \|x_i^{k+1} - x_i^k\|^2+\sum_{i=1}^n \|y_i^{k+1} - y_i^k\|^2+\sum_{i=1}^n \|v_i^{k+1} - v_i^k\|^2}\Big). 
\label{dxk+dxk+1}
\end{align}
Similarly, we can deduce 
\begin{align}
&\Ek{\sum_{i=1}^n \big(\norm{d_{y,i}^k - d_{\ty,i}^k} + \norm{d_{y,i}^{k+1} - d_{\ty,i}^{k+1}}\big)} \nonumber \\
\leq \, & C_2{\sum_{i=1}^n\Big(
	\frac{1-p+\rho}{\rho}\big(\norm{x_i^k -\tx_{i}^k} + \norm{y_i^k -\ty_{i}^k}\big)
	+ \frac{1}{1-\rho}\Ek{\sum_{i=1}^n \|x_i^{k+1} - x_i^k\|^2+\sum_{i=1}^n \|y_i^{k+1} - y_i^k\|^2}
	\Big)
}, \label{dyk+dyk+1}\\
\text{and} \nonumber \\
&\Ek{\sum_{i=1}^n \big(\norm{d_{v,i}^k - d_{\tv,i}^k} + \norm{d_{v,i}^{k+1} - d_{\tv,i}^{k+1}}\big)} \nonumber  \\
\leq \, & C_1\sum_{i=1}^n\Big(
\frac{1-p+\rho}{\rho}\big(\norm{x_i^k -\tx_{i}^k} + \norm{y_i^k -\ty_{i}^k} + \norm{v_i^k -\tv_{i}^k}\big)
\nonumber\\
& \quad + \frac{1}{1-\rho}\Ek{\sum_{i=1}^n \|x_i^{k+1} - x_i^k\|^2+\sum_{i=1}^n \|y_i^{k+1} - y_i^k\|^2+\sum_{i=1}^n \|v_i^{k+1} - v_i^k\|^2}\Big). 
\label{dvk+dvk+1}
\end{align}
Next, we recall a Lemma from \cite{dong2023singleloop}, which is useful for bounding  $\norm{\bar{d}_{y}^k}$ and $ \norm{\bar{d}_{v}^k}$.
\begin{lemma}{\cite[Lemma~A.7]{dong2023singleloop}}
The sequences $\{d_{y,i}^k\}$ and $\{d_{v,i}^k\}$ generated by \eqref{ssldbo_concise} satisfy
\begin{align}
	\| \bar{d}_{y}^k \|^2 \leq& \frac{2L_{f,1}^2}{n} \sum_{i=1}^{n} \left( \| x_{i}^{k} - \bar{x}^{k} \|^2 + \| y_{i}^{k} - \bar{y}^{k} \|^2 \right) + 2L_{f,1}^2 \| \bar{y}^{k} - y^{*}(\bar{x}^{k}) \|^2, \label{bardy}\\
	\| \bar{d}_{v}^k \|^2 \leq& \frac{5 (L_{f,2} r_{v} + L_{F,1})^2}{n} \sum_{i=1}^{n} \left( \| x_{i}^{k} - \bar{x}^{k} \|^2 + \| y_{i}^{k} - \bar{y}^{k} \|^2 \right) + \frac{5 L_{f,1}^2}{n} \sum_{i=1}^{n} \| v_{i}^{k} - \bar{v}^{k} \|^2 \nonumber 
	\\
	&+ 5 (L_{F,1} + r_{v} L_{f,2})^2 \| \bar{y}^{k} - y^{*}(\bar{x}^{k}) \|^2 + 5 L_{f,1}^2 \| \bar{v}^{k} - v^{*}(\bar{x}^{k}) \|^2. \label{bardv}
\end{align}
\end{lemma}

Now, by combining \eqref{zx-barzx}, the inequalities in Remark \ref{rmkA1}, Lemma \ref{lem:vk+1-vk}, Remark \ref{rem:xk+1-xk} and \eqref{dxk+dxk+1}-\eqref{bardv}, we can establish the following results for 
$\Ek{\mean \norm{z_{\star,i}^{k+1} - \z_{\star}^{k+1}}}$ for $\star = x, y, v$ 
respectively. Again, details are omitted for simplicity. 
\small{
\begin{align}
	&\Ek{ \mean \norm{z_{x,i}^{k+1} - \z_{x}^{k+1}}} \nonumber\\ 
	\leq & 
	\Big(\rho + \frac{4\alpha^2 C}{(1-\rho)^3 \rho p}\Big)\Ek{\mean \norm{z_{x,i}^{k} - \z_{x}^{k}}} 
	+ \frac{4\beta^2 C}{(1-\rho)^3\rho p}\Ek{\mean \norm{z_{y,i}^{k} - \z_{y}^{k}}} + \frac{4\eta^2 C}{(1-\rho)^3\rho p}\Ek{\mean \norm{z_{v,i}^{k} - \z_{v}^{k}}} \nonumber\\
	& + \Big(\frac{(1-p+\rho)C}{(1-\rho)^2 \rho^2 p} + \frac{4(\alpha^2+\beta^2+\eta^2)C^2}{(1-\rho^3)p^2}\Big){\mean \big(\norm{x_i^k - \tx_i^k} + \norm{y_i^k - \ty_i^k} + \norm{v_i^k - \tv_i^k}\big)} + \frac{4\alpha^2 C}{(1-\rho)^3 \rho p}{\norm{\d_{x}^k}} \nonumber\\
	& + \frac{(8+8L_{f,1}^2\beta^2 + 20\eta^2 L_1^2)C}{(1-\rho)^3 \rho p} \Big({\mean \norm{x_i^k - \x^k}} + {\mean \norm{y_i^k - \y^k}}\Big) + \frac{8C}{(1-\rho)^3 \rho p}{\mean \norm{v_i^k - \v^k}} \nonumber\\
	& + \frac{(8L_{f,1}^2\beta^2  + 20\eta^2 L_1^2 )C}{(1-\rho)^3 \rho p} {\norm{\y^k - y^*(\x^k)}} + \frac{20\eta^2 L_1^2 C}{(1-\rho)^3 \rho p} {\norm{\v^k - v^*(\x^k)}}, \label{bound_for_zx-barz}
\end{align}
\begin{align}
	&\Ek{ \mean \norm{z_{y,i}^{k+1} - \z_{y}^{k+1}}} \nonumber\\ 
	\leq & 
	\Big(\rho + \frac{4\beta^2 C}{(1-\rho)^3 \rho p}\Big)\Ek{\mean \norm{z_{y,i}^{k} - \z_{y}^{k}}} 
	+ \frac{4\alpha^2 C}{(1-\rho)^3\rho p}\Ek{\mean \norm{z_{x,i}^{k} - \z_{x}^{k}}} \nonumber\\
	& + \Big(\frac{(1-p+\rho)C}{(1-\rho)^2 \rho^2 p} + \frac{4(\alpha^2+\beta^2+\eta^2)C^2}{(1-\rho^3)p^2}\Big){\mean \big(\norm{x_i^k - \tx_i^k} + \norm{y_i^k - \ty_i^k}\big)} + \frac{8L_{f,1}^2\beta^2 C}{(1-\rho)^3 \rho p} {\norm{\y^k - y^*(\x^k)}} \nonumber\\
	& + \frac{(8+8L_{f,1}^2\beta^2)C}{(1-\rho)^3 \rho p} \Big({\mean \norm{x_i^k - \x^k}} + {\mean \norm{y_i^k - \y^k}}\Big) + 
	\frac{4\alpha^2 C}{(1-\rho)^3 \rho p}{\norm{\d_{x}^k}}, \label{bound_for_zy-barz}\\
	&\Ek{ \mean \norm{z_{v,i}^{k+1} - \z_{v}^{k+1}}} \nonumber\\ 
	\leq & 
	\Big(\rho + \frac{4\eta^2 C}{(1-\rho)^3 \rho p}\Big)\Ek{\mean \norm{z_{v,i}^{k} - \z_{v}^{k}}} 
	+ \frac{4\alpha^2 C}{(1-\rho)^3\rho p}\Ek{\mean \norm{z_{x,i}^{k} - \z_{x}^{k}}} + \frac{4\beta^2 C}{(1-\rho)^3\rho p}\Ek{\mean \norm{z_{y,i}^{k} - \z_{y}^{k}}} \nonumber\\
	& + \Big(\frac{(1-p+\rho)C}{(1-\rho)^2 \rho^2 p}+ \frac{4(\alpha^2+\beta^2+\eta^2)C^2}{(1-\rho^3)p^2}\Big) {\mean \big(\norm{x_i^k - \tx_i^k} + \norm{y_i^k - \ty_i^k} + \norm{v_i^k - \tv_i^k}\big)} + \frac{4\alpha^2 C}{(1-\rho)^3 \rho p}{\norm{\d_{x}^k}} \nonumber\\
	& + \frac{(8+8L_{f,1}^2\beta^2 + 20\eta^2 L_1^2)C}{(1-\rho)^3 \rho p} \Big({\mean \norm{x_i^k - \x^k}} + {\mean \norm{y_i^k - \y^k}}\Big) + \frac{8C}{(1-\rho)^3 \rho p}{\mean \norm{v_i^k - \v^k}} \nonumber\\
	& + \frac{(8L_{f,1}^2\beta^2  + 20\eta^2 L_1^2 )C}{(1-\rho)^3 \rho p} {\norm{\y^k - y^*(\x^k)}} + \frac{20\eta^2 L_1^2 C}{(1-\rho)^3 \rho p} {\norm{\v^k - v^*(\x^k)}},\label{bound_for_zv-barz}
\end{align}
}
where $C = \max \{C_1,C_2\}$ and $L_1$ is defined in \eqref{def-constants-main-text}.

\subsubsection{Some lemmas for convergence rate of Algorithm \ref{alg:sslDBo}}
Recall that $\Phi(x) = F(x,y^*(x))$ denotes the overall objective function. 
For any $k>0$, $\Phi(\x^k)$ and $\nabla \Phi(\x^k)$ are also measurable with respect to $\mathcal{F}_k$. 
\begin{lemma}\label{dlemma1a}
The sequence $\{({x}_i^k, {y}_i^k,{v}_i^k)\}$ generated by \eqref{ssldbo_concise} satisfies
\begin{align*}
	\Ek{\Phi(\bar{x}^{k+1})-\Phi(\bar{x}^k)}
	\leq & -\frac{\alpha}{2}{\|\nabla \Phi (\bar{x}^k)\|^2}
	-\frac{(1/\alpha
		-L_{\Phi})}{2}\Ek{\|\bar{x}^{k+1}-\bar{x}^k\|^2} +\frac{5\alpha L_{f,1}^2}{2}{\|\v^{k}-v^* (\x^k)\|^2} \nonumber\\
	&+\frac{5\alpha \left(L_{F,1} + L_{f,2}r_v\right)^2}{2}	{\|\y^{k}-y^* (\x^k)\|^2}
	+\frac{5\alpha L^2_{f,1}}{2n}{ \sum_{i=1}^n \|v_i^k-\v^k\|^2} \nonumber\\
	&+\frac{5\alpha\left(L_{F,1}+ r_vL_{f,2}\right)^2}{2n} {\sum_{i=1}^n \left(\|x_i^k-\x^k\|^2+ \|y_i^k-\y^k\|^2\right)}+\frac{\alpha(1-p)}{2np}{\sum_{i=1}^n \|d_{x,i}^k - d_{\tx,i}^k\|^2}, 
\end{align*}
where $L_{\Phi}$ is defined in \eqref{def-constants-main-text}.
\end{lemma}
\begin{proof}
From the updates of $x_i^{k+1}$ and $z_{x,i}^k$ in \eqref{xupdate}, we have 
\begin{equation}\label{xiter}
	\x^{k+1}=\x^{k}-\alpha\bar{z}_{x}^k.
\end{equation}
It follows from Lemma 2.2 in \cite{ghadimi2018approximation} that $\nabla \Phi (x)$ is $L_{\Phi}$-Lipschitz continuous.
On the other hand, from Lemma 5.7 in \cite{beck2017first} we derive
\begin{align}\label{phi0}
	\Ek{\Phi(\x^{k+1})-\Phi(\x^k)}	
	\leq \, & \Ek{\big\langle \nabla \Phi(\x^k), \x^{k+1}-\x^k \big\rangle+\frac{L_{\Phi}}{2}\|\x^{k+1}-\x^k\|^2}\nonumber\\
	\stackrel{(\ref{xiter})}
	= \, & -\alpha \big\langle \nabla \Phi (\x^k), \Ek{\z^k_{x}} \big\rangle 
	+\frac{L_{\Phi}}{2}
	\Ek{\|\x^{k+1}-\x^k\|^2}\nonumber\\
	= \, & -\frac{\alpha}{2}{\|\nabla \Phi(\x^k)\|^2}
	-\frac{\alpha}{2}
	\Ek{\|\z^k_{x}\|^2} +\frac{\alpha}{2}\Ek{\|\nabla \Phi (\x^k)-\z^k_{x}\|^2}
	+\frac{L_{\Phi}}{2}\Ek{\|\x^{k+1}-\x^k\|^2}
	\nonumber \\
	\stackrel{(\ref{xiter})}=\, & -\frac{\alpha}{2}{\|\nabla \Phi(\x^k)\|^2}
	-\frac{1}{2\alpha}\Ek{\|\x^{k+1}-\x^k\|^2} +\frac{\alpha}{2}{\|\nabla \Phi (\x^k)-\d^k_{x}\|^2}
	\nonumber
	\\&+ \frac{\alpha}{2}\Ek{\|\bar{d}_{x}^k - \bar{z}_{x}^k\|^2}+\frac{L_{\Phi}}{2}\Ek{\|\x^{k+1}-\x^k\|^2},
\end{align}
where the last equality is due to $\Ek{\langle\nabla \Phi (\x^k)-\d^k_{x}, \bar{d}_{x}^k - \bar{z}_{x}^k\rangle}=\Ek{1-\xi^k/p}\langle\nabla \Phi (\x^k)-\d^k_{x}, \d_{x}^k - \d_{\tx}^k\rangle=0$.
By definition, we have
$\nabla \Phi (\x^k)=\nabla_1 F(\x^k,y^*(\x^k))-\nabla^2_{12}f(\x^k,y^*(\x^k)) v^*(\x^k)$ and
$
	\d^k_x = \frac{1}{n}\sum_{i=1}^n\big(\nabla_1 F_i(x^k_i, y^k_i) - \nabla^2_{12} f_i(x^k_i, y^k_i) v^k_i\big).
$
For convenience, we define
\begin{equation*}
	\left\{
	\begin{array}{l}
		\Delta_1 := \left\|\nabla_1 F(\x^k,y^* (\x^k))-\nabla^2_{12}f (\x^k,y^* (\x^k)) v^* (\x^k)-\nabla_1 F(\x^k,\y^k)+\nabla^2_{12}f (\x^k,\y^k) \v^k\right\|, \smallskip \\
		\Delta_2 := \left\|\nabla_1 F(\x^k,\y^k)-\nabla^2_{12}f(\x^k,\y^k) \v^k- \frac{1}{n}\sum_{i=1}^n\Big(\nabla_1 F_i(x^k_i, y^k_i) - \nabla^2_{12} f_i(x^k_i, y^k_i) v^k_i\Big)\right\|.
	\end{array}
	\right.
\end{equation*}
From the triangle inequality and $\|v_i\|\leq r_v$ for all $i$, we derive
\begin{align}\label{phi2}
	\Delta_1	\leq \, & \Big\|\nabla_1 F(\x^k, y^* (\x^k))-\nabla_1 F(\x^k,\y^{k})\Big\|
	+\left\|\Big(\nabla^2_{12} f (\x^k,\y^k) - \nabla^2_{12} f (\x^k, y^* (\x^k))\Big)  \v^k\right\| \nonumber\\
	\, & +\left\|\nabla^2_{12} f(\x^k, y^* (\x^k))\big(\v^{k}-v^* (\x^k)\big)\right\|\nonumber\\
	\leq \, &
	\left(L_{F,1} + L_{f,2} r_v\right) \|\y^k-y^* (\x^k)\|
	+L_{f,1} \|\v^{k}-v^* (\x^k)\|.
\end{align}
By using the triangle inequality again and considering $F=\frac{1}{n}\sum_{i=1}^n F_i$ and $f=\frac{1}{n}\sum_{i=1}^n f_i$, we derive
\begin{align}\label{phi3}
	\Delta_2	\leq \, & \frac{1}{n}\sum_{i=1}^n \left\|\nabla_1 F_i(\x^k,\y^k)-\nabla^2_{12}f_i(\x^k,\y^k) \v^k- \nabla_1 F_i(x^k_i, y^k_i) + \nabla^2_{12} f_i(x^k_i, y^k_i) v^k_i)\right\|\nonumber\\
	\leq \, & \frac{1}{n} \sum_{i=1}^n\left\| \nabla_1 F_i(\x^k,\y^k)- \nabla_1 F_i(x^k_i, y^k_i)\right\| + \frac{1}{n} \sum_{i=1}^n\left\|\Big(\nabla^2_{12} f_i(\x^k,\y^k)-\nabla^2_{12} f_i(x_i^k, y_i^k)\Big)  v^k_i\right\| \nonumber\\
	\, & + \frac{1}{n}\sum_{i=1}^n\Big\|\nabla^2_{12} f_i(\x^k,\y^{k})(\v^k-v_i^k)\Big\|\nonumber\\
	\leq \, & \frac{L_{F,1}+ r_v L_{f,2}}{n}\sum_{i=1}^n \big(\|x_i^k-\x^k\|+ \|y_i^k-\y^k\|\big)
	+ \frac{L_{f,1}}{n}\sum_{i=1}^n \|v_i^k-\v^k\|.
\end{align}
It is easy to observe from the triangle inequality that ${\|\nabla \Phi(\x^k)-\d^k_x\|} \leq {\Delta_1 + \Delta_2}$.
Then, by using \eqref{phi2}, \eqref{phi3}  and the inequality $\left(\sum_{l=1}^{s} a_l\right)^2\leq s\sum_{l=1}^{s} a_l^2$, it is easy to derive
\begin{align*}
	{\left\|\nabla \Phi(\x^k)-\d^k_x\right\|^2}
	\leq  &\,  5\left(L_{F,1} + L_{f,2} r_v\right)^2 {\|\y^k-y^* (\x^k)\|^2} + 5L_{f,1}^2{ \|\v^{k}-v^* (\x^k)\|^2} \\
	&+\frac{5\left(L_{F,1}+ r_v L_{f,2}\right)^2}{n}{\sum_{i=1}^n\left( \|x_i^k-\x^k\|^2+ \|y_i^k-\y^k\|^2\right)}
	+ \frac{5L_{f,1}^2}{n}{\sum_{i=1}^n \|v_i^k-\v^k\|^2},
\end{align*}
which, together with \eqref{phi0} and $\Ek{\|\d^k_{x}-\z^k_{x}\|^2} \leq \frac{1-p}{np}{\sum_{i=1}^n \|d_{x,i}^k - d_{\tx,i}^k\|^2}$ (similar to \eqref{bard-barz}), yields the desired result.
\end{proof}

The next two lemmas bound $\Ek{\Vert\y^{k+1}-y^*(\x^k)\Vert^2}$ and $\Ek{\Vert \v^{k+1} - v^* (\x^k)\Vert^2}$, respectively.
\begin{lemma}\label{lemmay}
The sequence $\{(x^k_i,y^k_i,v^k_i)\}$ generated by \eqref{ssldbo_concise} satisfies
\begin{align}\label{y*}
	\Ek{\Vert\y^{k+1}-y^*(\x^k)\Vert^2 }
	\leq&\Big(1-\frac{\beta \sigma}{2}\Big){\|\y^k-y^*(\x^k)\|^2}
	+\frac{3\beta L_{f,1}^2}{n  \sigma} {\sum_{i=1}^n\left(\|\x^k-x^k_i\|^2+\|\y^k-y^k_i\|^2\right)} \nonumber\\
	&+\frac{(1-p)\beta^2 C_2}{np}{\sum_{i=1}^{n}\left(\|x_i^k-\tx_i^{k}\|^2+\|y_i^k-\ty_i^{k}\|^2\right)}.
\end{align}
\end{lemma}
\begin{proof}
First, by Cauchy-Schwarz inequality, for any $\Lambda>0$, we have
\begin{align}\label{y2p}
	\|\y^k-\beta \d_y^k \, - \, & y^*(\x^k)\|^2 =    \|\big[\y^k-\beta \nabla_2f(\x^k,\y^k)-y^*(\x^k)\big] +\beta\big[ \nabla_2f(\x^k,\y^k)- \d_y^k\big] \|^2 \nonumber\\
	\leq \, &(1+\Lambda) \|\y^k-\beta \nabla_2f(\x^k,\y^k)-y^*(\x^k) \|^2+(1+1/\Lambda)\beta ^2 \|\nabla_2f(\x^k,\y^k)-\d_y^k \|^2.
\end{align}
Note that $ \nabla_2f(\x^k,y^*(\x^k))=0$. It follows from the $\sigma$-strong convexity of $f(\bar{x}^k,\cdot)$, $L_{f,1}$-smoothness of $f$ and \cite[Theorem~2.1.12]{nesterov2018lectures}  that
\begin{align}
	\langle\y^k-y^*(\x^k),\nabla_2f(\x^k,\y^k) \rangle &=
	\langle\y^k-y^*(\x^k),\nabla_2f(\x^k,\y^k)-\nabla_2f(\x^k,y^*(\x^k)) \rangle \nonumber \\
	&\geq \frac{\sigma L_{f,1}}{\sigma+L_{f,1}}\|\y^k-y^*(\x^k)\|^2+\frac{1}{\sigma+L_{f,1}}\|\nabla_2f(\x^k,\y^k)\|^2. \label{jy-03}
\end{align}
By expanding $\Vert\y^k-\beta \nabla_2f(\x^k,\y^k)-y^*(\x^k)\Vert^2$, plugging in \eqref{jy-03}, and noting that $\sigma \leq L_{f,1}$ and $\beta \leq\frac{1}{L_{f,1}} \leq\frac{2}{\sigma+L_{f,1}}$ in Algorithm \ref{alg:sslDBo}, we obtain
\begin{align}\label{y1}
	\Vert\y^k-\beta \nabla_2f(\x^k,\y^k)-y^*(\x^k)\Vert^2
	\leq  \Big(1- \frac{2\beta \sigma L_{f,1}}{\sigma+L_{f,1}}\Big)\|\y^k-y^*(\x^k)\|^2\leq (1-\beta \sigma)\|\y^k-y^*(\x^k)\|^2.
\end{align}
It is elementary to show from $f=\frac{1}{n}\sum_{i=1}^{n}f_i$, the definition of $\d_y^k$ in \eqref{def:bar9}, the triangle inequality and Assumption \ref{Assump1} (c) that	
\begin{align}\label{y2}
	\Vert\nabla_2f(\x^k,\y^k)-\d_y^k\Vert^2
	\leq \frac{L_{f,1}^2}{n}\sum_{i=1}^n\big(\|\x^k-x^k_i\|^2+\|\y^k-y^k_i\|^2\big).
\end{align}
From \eqref{bar_z},
we know that 
$\Ek{\langle\d_{y}^k-\z_{y}^k, \y^k-\beta \d_y^k - y^*(\x^k)\rangle}=0$, and thus we have
\begin{align}
	\Ek{\norm{\y^{k+1} - y^*(\x^k)}} =\, &\Ek{\norm{\y^k-\beta \d_y^k - y^*(\x^k) + \beta(\d_y^k - \z_y^k)} }\nonumber \\
	= \, & {\norm{\y^k-\beta \d_y^k - y^*(\x^k) }} +  \Ek{\beta^2\norm{\d_y^k - \z_y^k}}. \label{y3}
\end{align}
By combining \eqref{bard-barz}, \eqref{y2p}, \eqref{y1}-\eqref{y3}, we derive
\begin{align*}
	\Ek{\norm{\y^{k+1} - y^*(\x^k)}}
	\leq \, & (1+\Lambda)(1-\beta\sigma) {\|\y^k-y^*(\x^k)\|^2}
	+\frac{1-p}{p}\frac{\beta^2}{n}{\sum_{i=1}^n \|d_{y,i}^k - d_{\ty,i}^k\|^2} \\
	& +(1+1/\Lambda)\frac{\beta^2L_{f,1}^2}{n}{\sum_{i=1}^n\big(\|\x^k-x^k_i\|^2+\|\y^k-y^k_i\|^2\big)}.
\end{align*}
The desired result \eqref{y*} can be obtained by setting $\Lambda =  \beta\sigma/2$ and using $\beta\sigma\leq 1$ and \eqref{dy}.
\end{proof}

\begin{lemma}\label{lemmav}
The sequence  $\{(x^k_i,y^k_i,v^k_i)\}$ generated by \eqref{ssldbo_concise} satisfies
\begin{align}\label{v^*}
	&\Ek{\Vert \v^{k+1} - v^* (\x^k)\Vert^2}
	\leq\left(1- {\eta  \sigma}/{2}\right) {\left\|\v^k-v^* (\x^k)\right\|^2}+\frac{3\eta \left(L_{F,1}+L_{f,2} r_v\right)^2}{ \sigma}
	{\left\| \y^k-y^* (\x^k) \right\|^2}\nonumber\\
	& \ \ +\frac{9\eta\left(L_{F,1}+L_{f,2} r_v\right)^2}{n\sigma} {\sum_{i=1}^n \left(\|x_i^k-\x^k\|^2+ \|y_i^k-\y^k\|^2\right)}+
	\frac{2\eta^2\rho^2}{n}\Ek{\sum_{i = 1}^{n}\| z_{v,i}^k-\z_v^k\|^2}\nonumber\\
	&\ \ +\frac{{9\eta L_{f,1}^2}/{\sigma} + 2\rho^2}{n}{\sum_{i=1}^n \|v_i^k-\v^k\|^2}+\frac{(1-p)\eta^2 C_1}{np}{\sum_{i=1}^{n}\left(\|x_i^k-\tx_i^{k}\|^2+\|y_i^k-\ty_i^{k}\|^2+\|v_i^k-\tv_i^{k}\|^2\right)}.
\end{align}
\end{lemma}
\begin{proof}
For convenience, we define
\begin{equation*}
	\left\{
	\begin{array}{l}
		\Delta_3 := \v^k+\eta \left[\nabla_2 F(\x^k,\y^{k})- \nabla^2_{22} f(\x^k,\y^{k})\v^k\right]-v^*(\x^k), \smallskip \\
		\Delta_4 := \nabla^2_{22} f(\x^k,\y^{k})\v^k-\nabla_2 F(\x^k,\y^{k})+\d_v^k.
	\end{array}
	\right.
\end{equation*}
Similar to \eqref{y2p}, for any $\delta>0$, we have
\begin{align}\label{v2p}
	\|\v^k+\eta \d_v^k-v^*(\x^k)\|^2 \leq (1+\delta)\left\|\Delta_3\right\|^2 +(1+1/\delta)\eta ^2\left\|\Delta_4\right\|^2.
\end{align}
First, we treat the term $\|\Delta_3\|^2$ in \eqref{v2p}.
Since $\nabla^2_{22} f(\x^k,y^* (\x^k))v^* (\x^k)=\nabla_2 F(\x^k,y^* (\x^k))$, we have the following reformulation:
$\Delta_3 = \Delta_{3,1} -\eta \big(\Delta_{3,2} + \Delta_{3,3}\big)$, where
\begin{equation*}
	\left\{
	\begin{array}{l}
		\Delta_{3,1} := \left[I - \eta\nabla^2_{22} f(\x^k,\y^k)\right]\Big(\v^k-v^* (\x^k)\Big), \smallskip \\
		\Delta_{3,2} := \left[\nabla^2_{22} f(\x^k,\y^k)-\nabla^2_{22} f(\x^k,y^* (\x^k))\right] v^* (\x^k), \smallskip \\
		\Delta_{3,3} := \nabla_2 F(\x^k,y^* (\x^k))-\nabla_2 F(\x^k,\y^k).
	\end{array}
	\right.
\end{equation*}
By Cauchy-Schwarz inequality, for any $\delta_1>0$, we have
\begin{align*}
	\|\Delta_3\|^2 	\leq
	\left(1+\delta_1\right) \| \Delta_{3,1} \|^2 + \left(1+1/{\delta_1}\right) \eta^2 \| \Delta_{3,2} + \Delta_{3,3} \|^2 .
\end{align*}	
Since $\eta\leq1/L_{f,1}$ and $f (x,\cdot)$ is $\sigma$-strongly convex, there holds
\begin{align*}
	\left\| \Delta_{3,1}\right\|
	\leq   \left\| I-\eta \nabla^2_{22} f (\x^k,\y^k)\right\|_{\textrm{op}}\left\|\v^k-v^* (\x^k)\right\|
	\leq   \left(1-\eta \sigma\right) \left\|\v^k-v^* (\x^k)\right\|.
\end{align*}
Furthermore, it is apparent from Assumption \ref{Assump1} that
$\left\|\Delta_{3,2}+\Delta_{3,3}\right\| \leq   \left(L_{f,2}r_v + L_{F,1}\right) 	\| \y^k-y^* (\x^k) \|$.
By taking $\delta_1=\eta \sigma$ and noting $\eta\sigma\leq\eta L_{f,1}\leq 1$, we obtain
\begin{align}\label{vt1}
	\left\|\Delta_3\right\|^2
	\leq &
	\left(1+\eta \sigma\right)\left(1-\eta \sigma\right)^2 \left\|\v^k-v^* (\x^k)\right\|^2+\left(1+1/{\eta \sigma}\right) \eta^2\left(L_{f,2}r_v + L_{F,1}\right)^2
	\left\| \y^k-y^* (\x^k) \right\|^2\nonumber\\
	\leq & \left(1-\eta  \sigma\right) \left\|\v^k-v^* (\x^k)\right\|^2+\frac{2\eta \left(L_{f,2}r_v + L_{F,1}\right)^2}{ \sigma}
	\left\| \y^k-y^* (\x^k) \right\|^2.
\end{align}
Next,   we treat the term $\|\Delta_4\|^2$ in \eqref{v2p}.  The definition of $\d^k_v$ implies that
\begin{align}\label{vt2}
	\left\|\Delta_4\right\|^2
	\leq &\ \  \frac{3}{n}\sum_{i=1}^n\big\|\nabla^2_{22} f_i(\x^k,\y^{k})(\v^k-v_i^k)\big\|^2
	+ \frac{3}{n}\sum_{i=1}^n\big\|[\nabla^2_{22} f_i(\x^k,\y^{k})-\nabla^2_{22} f_i(x^k_i,y^{k}_i)]v^k_i\big\|^2 \nonumber \\
	&\; +\frac{3}{n}\sum_{i=1}^n\big\|\nabla_2F_i(\x^k,\y^{k})-\nabla_2F_i(x_i^k,y_i^{k})\big\|^2\nonumber\\
	\leq &\ \ \frac{3(L_{f,2}r_v+L_{F,1})^2}{n}\sum_{i=1}^n \left(\|x_i^k-\x^k\|^2+\|y_i^k-\y^k\|^2\right)
	+ \frac{3L_{f,1}^2}{n}\sum_{i=1}^n \|v_i^k-\v^k\|^2.
\end{align}
Similar to Lemma \ref{lemmay}, we next evaluate $\sum_{i=1}^{n}\Vert v^{k+1}_i - v^* (\x^k)\Vert^2$  to bound $\Vert \v^{k+1} - v^* (\x^k)\Vert^2$.
By the update of $v^{k+1}_i$ in \eqref{vupdate}, we have
\begin{align*}
	\sum_{i=1}^{n}\Vert v^{k+1}_i - \, & v^* (\x^k)\Vert^2
	= \sum_{i=1}^{n}\left\| \mathcal{P}_{r_v}\left[\sum\nolimits_{j=1}^{n}w_{ij}\big(v_j^k + \eta z_{v,j}^k\big)\right]- v^* (\x^k)\right\|^2 \\
	\leq\, & \sum_{i=1}^{n}\left\| \sum\nolimits_{j=1}^{n}w_{ij}\big(v_j^k + \eta z_{v,j}^k\big)- v^* (\x^k)\right\|^2\\
	=\, &\sum_{i = 1}^{n}\left\| \sum\nolimits_{j=1}^{n}w_{ij}v_j^k -\v^k+ \eta \Big( \sum\nolimits_{j=1}^{n}w_{ij}z_{v,j}^k-\z_v^k\Big)\right\|^2 + \sum_{i = 1}^{n}\Vert \v^k - v^* (\x^k) + \eta \z_v^k \Vert^2\\
	\leq\, &\sum_{i = 1}^{n}\Vert \v^k - v^* (\x^k) + \eta \z_v^k \Vert^2+2\sum_{i = 1}^{n}\left\| \sum\nolimits_{j=1}^{n}w_{ij}v_j^k -\v^k\right\|^2 + 2\eta ^2\sum_{i = 1}^{n}\left\| \sum\nolimits_{j=1}^{n}w_{ij}z_{v,j}^k-\z_v^k\right\|^2\\
	\leq\, &\sum_{i = 1}^{n}\Vert \v^k - v^* (\x^k) + \eta \z_v^k \Vert^2+2\rho^2\sum_{i=1}^{n}\|v_i^k-\v^k\|^2 + 2\eta ^2\rho^2\sum_{i = 1}^{n}\| z_{v,i}^k-\z_v^k\|^2,
\end{align*}
where the second equality holds because	$\sum_{i = 1}^{n}\left[\sum_{j=1}^{n}w_{ij}v_j^k -\v^k+ \eta ( \sum_{j=1}^{n}w_{ij}z_{v,j}^k-\z_v^k)\right]=0$, and the last inequality follows from Lemma \ref{con_sqr}.
Then, by combining \eqref{dv}, \eqref{v2p}-\eqref{vt2}, we derive
\begin{align*}
	&\Ek{\Vert \v^{k+1} - v^* (\x^k)\Vert^2}  \leq \frac{1}{n}\Ek{\sum_{i=1}^{n}\Vert v^{k+1}_i - v^* (\x^k)\Vert^2} \\
	\leq \, &\Ek{\Vert \v^k - v^* (\x^k) + \eta \d_v^k + \eta (\z_{v}^k - \d_v^k)\Vert^2} + \frac{2\rho^2}{n}\Ek{\sum_{i=1}^{n}\|v_i^k-\v^k\|^2} + \frac{2\eta^2\rho^2}{n}\Ek{\sum_{i = 1}^{n}\| z_{v,i}^k-\z_v^k\|^2}\\
	=\, &\Ek{\Vert \v^k - v^* (\x^k) + \eta \d_v^k \Vert^2} +\eta^2\Ek{\|\z_{v}^k - \d_v^k\|^2}+ \frac{2\rho^2}{n}\Ek{\sum_{i=1}^{n}\|v_i^k-\v^k\|^2} + \frac{2\eta^2\rho^2}{n}\Ek{\sum_{i = 1}^{n}\| z_{v,i}^k-\z_v^k\|^2}\\
	\leq\, &(1+\delta)\left(1-\eta  \sigma\right) {\left\|\v^k-v^* (\x^k)\right\|^2}+\frac{2(1+\delta)\eta \left(L_{F,1}+L_{f,2} r_v\right)^2}{ \sigma}	{\left\| \y^k-y^* (\x^k) \right\|^2}\\
	&+ \frac{3(1+1/\delta)\eta^2}{n}\Big[ \left(L_{F,1}+L_{f,2} r_v\right)^2 {\sum_{i=1}^n\left( \|x_i^k-\x^k\|^2+ \|y_i^k-\y^k\|^2\right)}+ L_{f,1}^2{\sum_{i=1}^n \|v_i^k-\v^k\|^2}\Big]\\
	&+ \frac{2\rho^2}{n}\sum_{i=1}^{n}\|v_i^k-\v^k\|^2+\frac{2\rho^2\eta^2}{n}\Ek{\sum_{i = 1}^{n}\| z_{v,i}^k-\z_v^k\|^2}+\frac{(1-p)\eta^2 }{np}{\sum_{i=1}^{n}\|d_{v,i}^k - d_{\tv,i}^k\|^2},
\end{align*}
where the first inequality is because $\left(\sum_{l=1}^{s} a_l\right)^2\leq s\sum_{l=1}^{s} a_l^2$, and the equality is due to 
$$\Ek{\langle\v^k - v^* (\x^k) + \eta \d_v^k, \bar{z}_{v}^k - \bar{d}_{v}^k\rangle}
=  \Ek{1-\xi^k/p} \langle\v^k - v^* (\x^k) + \eta \d_v^k, \d_{\tv}^k - \d_{v}^k\rangle)=0.$$
Finally, the desired result \eqref{v^*} follows from setting $\delta= \eta\sigma/2$ and using $\eta\sigma\leq 1$ and \eqref{dv}.
\end{proof}

\begin{lemma}\label{*k+1-*k}
Let $L_v$ be defined in \eqref{def-constants-main-text}. Then, there hold
\begin{align}\label{y*k+1-y*k}
	\|y^*(\x^{k+1})-y^*(\x^{k})\| \leq \frac{L_{f,1}}{\sigma} \|\x^{k+1}-\x^{k}\|
	\text{~~and~~}
	\|v^*(\x^{k+1})-v^*(\x^{k})\| \leq \frac{L_v}{\sigma} \|\x^{k+1}-\x^{k}\|.
\end{align}
\end{lemma}
\begin{proof}
Due to the optimality of $y^*(x)$, we have $\nabla_2 f(x, y^*(x)) = 0$ for any $x$. Now, let $x$ and $x'$ be arbitrarily fixed. It follows from the $\sigma$-strong convexity of $f(x, \cdot)$ and $L_{f,1}$-Lipschitz continuity of $\nabla f$ that
\begin{align*}
	\sigma \| y^*(x) - y^*(x') \| &\leq \| \nabla_2 f(x, y^*(x)) - \nabla_2 f(x, y^*(x')) \| \\
	&= \| \nabla_2 f(x, y^*(x')) - \nabla_2 f(x', y^*(x')) \| \leq L_{f,1} \| x - x' \|.
\end{align*}
Hence, we have
\begin{equation}
	\| y^*(x) - y^*(x') \| \leq \frac{L_{f,1}}{\sigma} \| x - x' \|. \label{y*(x)-y*(x')}
\end{equation}
The first inequality in \eqref{y*k+1-y*k} follows immediately from \eqref{y*(x)-y*(x')} by taking $x = \bar{x}^{k+1}$ and $x' = \bar{x}^{k}$.
Next, to obtain the second inequality in \eqref{y*k+1-y*k}, we define
\begin{equation*}
	\left\{
	\begin{array}{l}
		\Delta_5 := \nabla_2 F(\x^k,y^* (\x^k)) - \nabla_2 F(\x^{k+1},y^* (\x^{k+1})), \smallskip \\
		\Delta_6 := \big[\nabla^2_{22} f(\x^{k+1},y^* (\x^{k+1}))-\nabla^2_{22} f(\x^k,y^* (\x^k))\big]v^* (\x^{k+1}).
	\end{array}
	\right.
\end{equation*}
By using Assumption \ref{Assump1}, we have
\begin{equation}\label{jy-05}
	\left\{
	\begin{array}{l}
		\|\Delta_5\| \leq L_{F,1}\big(\|\x^{k+1}-\x^k\|+\|y^* (\x^{k+1})-y^* (\x^k)\|\big), \smallskip \\
		\|\Delta_6\| \leq L_{f,2}r_v\big(\|\x^{k+1}-\x^k\|+\|y^* (\x^{k+1})-y^* (\x^k)\|\big).
	\end{array}
	\right.
\end{equation}
It follows from $\nabla^2_{22} f(\x^{k+1},y^* (\x^{k+1}))v^* (\x^{k+1})=\nabla_2 F(\x^{k+1},y^* (\x^{k+1}))$ that
\begin{align}\label{jy-04}
	\nabla^2_{22} f(\x^k,y^* (\x^k))(v^* (\x^k)-v^* (\x^{k+1})) = \Delta_5 + \Delta_6.
\end{align}
The $\sigma$-strong convexity of $f(\x^k,\cdot)$ implies that
\begin{align*}
	\sigma\|v^* (\x^k)& - v^* (\x^{k+1})\|
	\leq \|\nabla^2_{22} f(\x^k,y^* (\x^k))(v^* (\x^k)-v^* (\x^{k+1}))\|  \\
	& \stackrel{(\ref{jy-04})}= \|\Delta_5 + \Delta_6\|  \leq \|\Delta_5\| + \|\Delta_6\|
	\stackrel{(\ref{y*(x)-y*(x')}, \ref{jy-05})}\leq  \left(L_{F,1}+L_{f,2}r_v\right)\left(1+ {L_{f,1}}/{\sigma} \right)\|\x^{k+1}-\x^k\|,
\end{align*}
which implies the second inequality in \eqref{y*k+1-y*k} by noting the definition of $L_v$ in \eqref{def-constants-main-text}.
\end{proof}

\subsubsection{Proof of Theorem \ref{cr}.}
\begin{proof}
By using Cauchy-Schwarz inequality again, we derive
\begin{align*}
	\Ek{\|\y^{k+1}-y^*(\x^{k+1})\|^2}\leq \left(1+\frac{\beta\sigma}{4}\right)\Ek{\|\y^{k+1}-y^*(\x^{k})\|^2}+ \left(1+\frac{4}{\beta\sigma}\right)\Ek{\|y^*(\x^{k+1})-y^*(\x^{k})\|^2}.
\end{align*}
By taking into account \eqref{y*}, the first inequality in \eqref{y*k+1-y*k}, and $\beta\sigma\leq 1$, we obtain
\begin{align}\label{yf}
	\Ek{\|\y^{k+1}-&y^*(\x^{k+1})\|^2}
	\leq \Big(1-\frac{\beta \sigma}{4}\Big)\Ek{\|\y^k-y^*(\x^k)\|^2}+\frac{15\beta L_{f,1}^2}{4 n\sigma} \Ek{\sum_{i=1}^n\left(\|\x^k-x^k_i\|^2+\|\y^k-y^k_i\|^2\right)} \nonumber \\
	& +\frac{5L_{f,1}^2}{\beta\sigma^3}\Ek{\|\x^{k+1}-\x^{k}\|^2}
	+ \frac{5(1-p) \beta^2 C_2}{4np}\Ek{\sum_{i=1}^{n}\left(\|x_i^k-\tx_i^{k}\|^2+\|y_i^k-\ty_i^{k}\|^2\right)}.
\end{align}
Similarly, we can derive
\begin{align*}
	\Ek{\|\v^{k+1}-v^*(\x^{k+1})\|^2}\leq \left(1+\frac{\eta\sigma}{4}\right)\Ek{\|\v^{k+1}-v^*(\x^{k})\|^2}+ \left(1+\frac{4}{\eta\sigma}\right)\Ek{\|v^*(\x^{k+1})-v^*(\x^{k})\|^2}.
\end{align*}
By taking into account \eqref{v^*}, the second inequality in \eqref{y*k+1-y*k}, and $\eta\sigma\leq 1$, we obtain
\begin{align}\label{vf}
	\Ek{\|\v^{k+1}- \, & v^*(\x^{k+1})\|^2}
	\leq  \left(1-\frac{\eta  \sigma}{4}\right) \Ek{\left\|\v^k-v^* (\x^k)\right\|^2}+\frac{15\eta \left(L_{F,1}+L_{f,2} r_v\right)^2}{4\sigma}
	\Ek{\left\| \y^k-y^* (\x^k) \right\|^2}\nonumber\\
	& +\frac{5L_v^2}{\eta\sigma^3}\Ek{\|\x^{k+1}-\x^{k}\|^2} +\frac{45\eta\left(L_{F,1}+L_{f,2} r_v\right)^2}{4n\sigma}\Ek{\sum_{i=1}^n \left(\|x_i^k-\x^k\|^2+ \|y_i^k-\y^k\|^2\right)}\nonumber\\
	&  + \Big(\frac{45\eta L_{f,1}^2}{4\sigma}+\frac{5\rho^2}{2}\Big)\frac{1}{n}\Ek{\sum_{i=1}^{n}\|v_i^k-\v^k\|^2} + \frac{5\rho^2\eta ^2}{2n} \Ek{\sum_{i = 1}^{n}\| z_{v,i}^k-\z_v^k\|^2} \nonumber \\
	& + \frac{5(1-p) \eta^2 C_1}{4np}\Ek{\sum_{i=1}^{n}\left(\|x_i^k-\tx_i^{k}\|^2+\|y_i^k-\ty_i^{k}\|^2+\|v_i^k-\tv_i^{k}\|^2\right)}.
\end{align}
Define the Lyapunov function as in \eqref{Lyapunov}. 
By taking expectation of   the inequalities in Lemma \ref{bar}, Remark \ref{bar_re} and Lemma \ref{dlemma1a},  as well as \eqref{bound_for_zx-barz}-\eqref{bound_for_zv-barz}, \eqref{yf} and \eqref{vf}, and combining them all together, it is straightforward to derive
\begin{align}
	\E{V_{k+1} - V_{k}} \leq 
	& -\frac{\alpha}{2} \E{\norm{\nabla \Phi(\x_{k})}} - A_1 \E{\norm{\d_{x}^k}} -A_2 \E{\norm{\y^k - y^{*}(\x^k)}} - A_3 \E{\norm{\v^k - v^{*}(\x^k)}} \nonumber 
	\\
	& - \frac{A_4}{n}\E{\sum_{i=1}^n\norm{x_i^k -\x^k}} - \frac{A_5}{n}\E{\sum_{i=1}^n\norm{y_i^k -\y^k}} - \frac{A_6}{n}\E{\sum_{i=1}^n\norm{v_i^k -\v^k}} \nonumber 
	\\
	& - \frac{A_7}{n}\E{\sum_{i=1}^n\norm{z_{x,i}^k - \z_{x}^k}} - \frac{A_8}{n}\E{\sum_{i=1}^n\norm{z_{y,i}^k - \z_{y}^k}} - \frac{A_9}{n}\E{\sum_{i=1}^n\norm{z_{v,i}^k - \z_{v}^k}} \nonumber 
	\\ 
	& - \frac{A_{10}}{n}\E{\sum_{i=1}^n\norm{x_i^k -\tx_{i}^k}} - \frac{A_{11}}{n}\E{\sum_{i=1}^n\norm{y_i^k -\ty_{i}^k}} - \frac{A_{12}}{n} \E{\sum_{i=1}^n \norm{v_i^k -\tv_{i}^k}}, \label{V}
\end{align}
where the coefficients are given by
\begin{align*}
	A_1 & = \frac{\alpha}{2} - \frac{L_{\Phi}\alpha^2}{2}-\frac{4\alpha^2 (\alpha^2 a_6C_1 + a_7\beta^2 C_1  + a_8\eta^2 C_2) }{(1-\rho)^3 \rho p} - \frac{4\alpha^2 a_9}{1-\rho},\\
	A_2 & = \frac{a_1 \beta\sigma}{2} - \frac{5\alpha L_1^2}{2} - \frac{3\eta L_1^2 a_2}{\sigma} - 
	\frac{(8\beta^2 L_{f,1}^2 + 20\eta^2 L_1^2)C_1(a_6\alpha^2 + a_7 \beta^2) + 8\beta^2 L_{f,1}^2 C_2 a_8 \eta^2}{(1-\rho)^3\rho p} \\
	& \quad - \frac{(8\beta^2 L_{f,1}^2 a_{10} + 20 \eta^2 L_1^2 a_{11})}{1-\rho}, \\
	A_3 & = \frac{a_2\eta \sigma}{2} - \frac{5\alpha L_{f,1}^2}{2} - \frac{20\eta^2 L_{f,1}^2 C_1(a_6\alpha^2 + a_7\beta^2)}{(1-\rho)^3 \rho p} - \frac{20\eta^2 L_{f,1}^2 a_{11}}{1-\rho},\\
	A_4 & = {a_3(1-\rho)} - \frac{5\alpha L_1^2}{2} - \frac{3\beta L_{f,1}^2 a_1}{\sigma} -\frac{9\eta L_1^2 a_2}{\sigma} - \frac{C_1(a_6 \alpha^2 + a_7 \beta^2)(8+8\beta^2 L_{f,1}^2 + 20\eta^2 L_1^2)}{(1-\rho)^3 \rho p} \\
	& \quad - \frac{C_2 a_8 \eta^2(8+8\beta^2 L_{f,1}^2)}{(1-\rho)^3 \rho p} - \frac{8a_9}{1-\rho}-\frac{8\beta^2 L_{f,1}^2 a_{10}}{1-\rho} - \frac{20\eta^2 L_1^2 a_{11}}{1-\rho},\\
	A_5 & = {a_4(1-\rho)} - \frac{5\alpha L_1^2}{2} - \frac{3\beta L_{f,1}^2 a_1}{\sigma} -\frac{9\eta L_1^2 a_2}{\sigma} - \frac{C_1(a_6 \alpha^2 + a_7 \beta^2)(8+8\beta^2 L_{f,1}^2 + 20\eta^2 L_1^2)}{(1-\rho)^3 \rho p} \\
	& \quad - \frac{C_2 a_8 \eta^2(8+8\beta^2 L_{f,1}^2)}{(1-\rho)^3 \rho p} - \frac{(8 + 8\beta^2 L_{f,1}^2) a_{10}}{1-\rho} - \frac{20\eta^2 L_1^2 a_{11}}{1-\rho},\\
	A_6 & = a_5(1-\rho) - \frac{5\alpha L_{f,1}^2}{2} - \Big(\frac{9\eta L_{f,1}^2}{\sigma} + {2\rho^2}\Big)a_2 - \frac{8C_1 (a_6 \alpha^2 + a_7 \beta^2)}{(1-\rho)^3 \rho p} - \frac{8a_{11}}{1-\rho},\\
	A_7 & = a_6 \alpha^2 (1-\rho) - \frac{\rho^2 \alpha^2 a_3}{1-\rho} - \frac{16\alpha^2\eta^2 C_1 a_6 \alpha^2}{(1-\rho)^3 \rho p} - \frac{16\alpha^2\eta^2 C_1 a_7 \beta^2}{(1-\rho)^3 \rho p} - 
	\frac{16\alpha^2\beta^2 C_2 a_8 \eta^2}{(1-\rho)^3 \rho p} - \frac{4a_9 \alpha^2}{1-\rho},\\
	A_8 & = a_7 \beta^2 (1-\rho) - \frac{\rho^2 \beta^2 a_4}{1-\rho} - \frac{16\alpha^2\eta^2 C_1 a_6 \alpha^2}{(1-\rho)^3 \rho p} - \frac{16\alpha^2\eta^2 C_1 a_7 \beta^2}{(1-\rho)^3 \rho p} - 
	\frac{16\alpha^2\beta^2 C_2 a_8 \eta^2}{(1-\rho)^3 \rho p} - \frac{4a_{10}\beta^2}{1-\rho},\\
	A_9 & = a_8 \eta^2 (1-\rho) - 2\eta^2\rho^2 a_2 -\frac{\rho^2 \alpha^2 a_5}{1-\rho} - \frac{16\alpha^2\eta^2 C_1 a_6 \alpha^2}{(1-\rho)^3 \rho p} - \frac{16\alpha^2\eta^2 C_1 a_7 \beta^2}{(1-\rho)^3 \rho p}  - \frac{4a_{11} \eta^2}{1-\rho},\\
	A_{10} & = \frac{a_9 (p+\rho -1)}{\rho} - \frac{\rho \beta^2 C_2 a_1}{p} - \frac{\rho \eta^2 C_1 a_2}{p} - \frac{(1- p +\rho) \big((a_6 \alpha^2 + a_7 \beta^2)C_1 + a_8 \eta^2 C_2\big)}{(1-\rho)^2 \rho^2 p}\\
	&\quad \,-\frac{4(\alpha^2+\beta^2+\eta^2)C^2(a_6 \alpha^2 + a_7\beta^2 +a_8\eta^2)}{(1-\rho)^3p^2}, \\
	A_{11} & = \frac{a_{10} (p+\rho -1)}{\rho} - \frac{\rho \beta^2 C_2 a_1}{p} - \frac{\rho \eta^2 C_1 a_2}{p} - \frac{(1- p +\rho) \big((a_6 \alpha^2 + a_7 \beta^2)C_1 + a_8 \eta^2 C_2\big)}{(1-\rho)^2 \rho^2 p}\\
	&\quad \,-\frac{4(\alpha^2+\beta^2+\eta^2)C^2(a_6 \alpha^2 + a_7\beta^2 +a_8\eta^2)}{(1-\rho)^3p^2},\\
	A_{12} & = \frac{a_{11}(p+\rho -1)}{\rho} - \frac{\rho \eta^2 C_1 a_2}{p} - \frac{(1- p +\rho) (a_6 \alpha^2 + a_7 \beta^2)C_1 }{(1-\rho)^2 \rho^2 p}-\frac{4(\alpha^2+\beta^2+\eta^2)C^2(a_6 \alpha^2 + a_7\beta^2)}{(1-\rho)^3p^2}.
\end{align*}
Here, constants such as $L_{\Phi}, L_{v}$ and $L_1$ are defined in \eqref{def-constants-main-text}. Set   
\begin{align*}
	a_1 = a_2 = a_6 = a_7 =1, \; a_3 = a_4 = \frac{(1-\rho)^2}{5\rho^2}, \; a_5 = a_8 = \frac{10\rho^2}{1-\rho},  \; a_9 = a_{10} = \frac{(1-\rho)^4}{320}, \; a_{11} = \frac{(1-\rho)\rho^2}{4},
\end{align*}
and choose the stepsizes satisfying
\begin{align}\label{stepsize_bound}
	\beta  \leq \min \Big\{
	&
	\frac{1}{3C},
	\frac{(1-\rho)\sigma}{128a_{10}L_{f,1}^2}, 
	\frac{(1-\rho)^3\rho p\sigma}{128L_{f,1}^2}, 
	\frac{a_3 (1-\rho)\sigma}{24L_{f,1}^2}, 
	\frac{1}{L_{f,1}}, 
	\sqrt{\frac{a_3 (1-\rho)^4 \rho p}{576C}}, 
	\sqrt{\frac{a_3 (1-\rho)^2}{64 a_{10}L_{f,1}^2}}, 
	\frac{a_3 (1-\rho)\sigma}{48L_{f,1}^2}, \nonumber \\
	& \sqrt{\frac{a_5 (1-\rho)^4 \rho p}{80C}}, 
	\sqrt{\frac{3(1-\rho)^4 \rho p}{80}}, 
	\sqrt{\frac{a_9 (p+\rho-1)p}{3\rho^2C}},
	\sqrt{\frac{a_9 (p+\rho-1)(1-\rho)^2\rho p}{12(1-p+\rho)C}}, 
	\nonumber\\
	&\sqrt{\frac{a_{11} (p+\rho-1)(1-\rho)^2\rho p}{6(1-p+\rho)C}}, \sqrt{\frac{a_9(p+\rho-1)(1-\rho)^3 p^2}{16\rho}}, \sqrt{\frac{a_{11}a_8(p+\rho-1)(1-\rho)^3 p^2}{8\rho}}\nonumber	
	\Big\},\\
	\eta \leq \min \Big\{
	&
	\frac{1}{3a_8 C}, 
	\sqrt{\frac{3\beta \rho (1-\rho)^3 \rho p}{640 {L_1}^2}}, 
	\frac{(1-\rho)^3\rho p\sigma}{80L_{f,1}^2}, 
	\frac{(1-\rho)\sigma}{120 a_{11}L_{f,1}^2},
	\frac{a_3 (1-\rho)\sigma}{72L_{1}}, 
	\frac{1}{L_{1}}, 
	\sqrt{\frac{a_3 (1-\rho)^4 \rho p}{128a_8 C}}, 
	\nonumber\\
	&
	\sqrt{\frac{a_3 (1-\rho)^2}{160 a_{11}L_{1}^2}},
	\frac{a_3 (1-\rho)\sigma}{72 L_{1}^2}, 
	\frac{a_5 \sigma(1-\rho)}{45L_{f,1}^2}, 
	\sqrt{\frac{3(1-\rho)^4 \rho p}{80}}, 
	\sqrt{\frac{a_9 (p+\rho-1)p}{3\rho^2C}}, 
	\nonumber\\
	&
	\sqrt{\frac{a_9 (p+\rho-1)(1-\rho)^2\rho p}{12(1-p+\rho)a_8 C}}, 
	\sqrt{\frac{a_{11}(p+\rho -1)}{2\rho^2 C}},
	\sqrt{\frac{a_9(p+\rho-1)(1-\rho)^3 p^2}{16\rho}}\nonumber\Big\},\\
	\alpha \leq \min \Big\{
	&
	\frac{1}{3L_{\phi}},
	\frac{1}{3C},
	\frac{(1-\rho)^3\rho p}{24}, 
	\frac{1-\rho}{24a_9},
	\frac{\beta \sigma}{20L_{1}^2},
	\frac{\eta\sigma}{15L_{f,1}^2},
	\frac{a_3(1-\rho)}{20L_{1}^2},
	\sqrt{\frac{a_3 (1-\rho)^4 \rho p}{576C}}, 
	\frac{2a_5 (1-\rho)}{25L_{f,1}^2},
	\nonumber\\
	&
	\sqrt{\frac{a_5 (1-\rho)^4 \rho p}{80C}},
	\frac{(1-\rho)\eta}{\sqrt{5}\rho},
	\sqrt{\frac{3a_8(1-\rho)^4 \rho p}{80}},
	\sqrt{\frac{a_9 (p+\rho-1)(1-\rho)^2\rho p}{12(1-p+\rho)C}}, 
	\nonumber\\
	&\sqrt{\frac{a_{11}(p+\rho-1)(1-\rho)^2\rho p}{6(1-p+\rho)C}}, 
	\sqrt{\frac{a_9(p+\rho-1)(1-\rho)^3 p^2}{16\rho}}, \sqrt{\frac{a_{11}a_8(p+\rho-1)(1-\rho)^3 p^2}{8\rho}}
	\Big\}.
\end{align}
Assume $p+\rho>1$. Then, 
it is elementary to show from \eqref{stepsize_bound} that $ A_1, A_2, \ldots, A_{12} $ are all nonnegative. Thus, \eqref{V} implies
\[ 
\E{V_{k+1} - V_{k}} \leq -\frac{\alpha}{2} \E{\| \nabla \Phi(\bar{x}^k) \|^2}.
\]
By telescoping the above inequality, we obtain
{
\[ 
\E{\frac{1}{K} \sum\nolimits_{k=0}^{K-1} \| \nabla \Phi(\bar{x}^k) \|^2 }\leq \frac{2V_0}{\alpha K} = \mathcal{O}\left( \frac{1}{K} \right).
\]
}
The consensus error can also be established by
$
\E{V_{k+1} - V_{k}} \leq -{A_4}\E{\sum_{i=1}^{n} \| x_{i}^{k} - \bar{x}^{k} \|^2}\big/n,
$
which yields
{
\[ 
 \E{\frac{1}{nK}\sum\nolimits_{k=0}^{K-1} \sum\nolimits_{i=1}^{n} \| x_{i}^{k} - \bar{x}^{k} \|^2 }\leq \frac{V_0}{A_4 K} = \mathcal{O}\left( \frac{1}{K} \right).
\]
}
Similarly, we can derive
{
\[ 
 \E{\frac{1}{nK}\sum\nolimits_{k=0}^{K-1} \sum\nolimits_{i=1}^{n} \| y_{i}^{k} - \bar{y}^{k} \|^2 }= \mathcal{O}\left( \frac{1}{K} \right) \quad \text{and} \quad  \E{\frac{1}{nK}\sum\nolimits_{k=0}^{K-1} \sum\nolimits_{i=1}^{n} \| v_{i}^{k} - \bar{v}^{k} \|^2 }= \mathcal{O}\left( \frac{1}{K} \right),
\]
}
which completes the proof of Theorem \ref{cr}.
\end{proof}
\begin{remark}
The method of selecting the stepsizes is as follows: the subsequent $\mathbf{m}$ terms subtracted from the first term should each be no greater than $1/ \mathbf{m}$
times the value of the first term to ensure that all $A_i$ are nonnegative. Then, the appropriate range for the stepsizes can be determined.
\end{remark}

\begin{remark}
In our analysis, the condition $p >1-\rho$ is required to ensure that $A_{10}, A_{11}, A_{12}$ are nonnegative. In fact, this requirement is a flaw of the proof and not essential, which we now explain. 
In the proof of Lemma \ref{lem:xk+1-txk+1}, we used the Cauchy-Schwarz inequality $2\langle a, b \rangle \leq t \|a\|^2 +  \|b\|^2/t$ with $t = \rho/(1-\rho)$.
This choice of $t$ makes the condition $(1-p)/\rho<1$ necessary to bound the term $\Ek{\sum_{i=1}^n \|x_i^{k+1} - \tx_i^{k+1}\|^2}$. If a different $t>0$ is chosen, the condition $p>1-\rho$ can then be replaced by $p>1/(1+t)$. As $t$ can be arbitrarily large, $p>0$ can be arbitrarily small. 
\end{remark}

\subsubsection{Proof of Theorem \ref{cp}.}
\begin{proof}
From \eqref{V} and \eqref{stepsize_bound}, we can establish the following inequality:
$$
\E{V_{k+1} - V_{k}} \leq -\frac{\alpha}{2} \E{\| \nabla \Phi(\bar{x}^k) \|^2} -\frac{A_4}{n} \E{\sum_{i=1}^{n} \| x_{i}^{k} - \bar{x}^{k} \|^2} -\frac{A_5}{n} \E{\sum_{i=1}^{n} \| y_{i}^{k} - \bar{y}^{k} \|^2},
$$
which yields 
$
	\sum_{k=0}^{\infty} \E{\norm{\nabla \Phi(\x^k)} + \mean \norm{x_i^k - \x^k} + \mean \norm{y_i^k - \y^k}}  
	\leq {V_0}/{C_3} < \infty,
$
where $C_3 = \min\{\alpha/2, A_4, A_5\}$.
Therefore, the term $\norm{\nabla \Phi(\x^k)} + \mean \norm{x_i^k - \x^k} + \mean \norm{y_i^k - \y^k}$ approaches 0 in expectation as $k \rightarrow \infty$.
The desired result follows from the Markov inequality.
\end{proof}

\end{document}